\documentclass[english]{article}
\usepackage{amsmath}
\usepackage{amsthm}
\usepackage{mathptmx}
\usepackage{newtxmath}
\usepackage[T1]{fontenc}
\usepackage[latin9]{inputenc}
\usepackage[a4paper]{geometry}
\usepackage{xcolor}
\usepackage{babel}
\usepackage{textcomp}
\usepackage{mathtools}
\usepackage{enumitem}
\usepackage{graphicx}
\usepackage[all]{xy}
\usepackage[unicode=true,
 bookmarks=true,bookmarksnumbered=false,bookmarksopen=false,
 breaklinks=false,pdfborder={0 0 0},pdfborderstyle={},backref=false,colorlinks=false]
 {hyperref}
\hypersetup{pdftitle={Critical points in higher dimensions, I: Reverse order of periodic orbit creations in the Lozi family},
 pdfauthor={Dyi-Shing Ou},
 pdfkeywords={Dynamical systems, Lozi maps, symbolic dynamics, border collision bifurcations}}

\makeatletter

\providecommand{\tabularnewline}{\\}

\numberwithin{equation}{section}
\theoremstyle{plain}
\newtheorem*{thm*}{\protect\theoremname}
\theoremstyle{plain}
\newtheorem{thm}{\protect\theoremname}[section]
\theoremstyle{plain}
\newtheorem{lem}[thm]{\protect\lemmaname}
\theoremstyle{plain}
\newtheorem{cor}[thm]{\protect\corollaryname}
\theoremstyle{plain}
\newtheorem{prop}[thm]{\protect\propositionname}
\theoremstyle{remark}
\newtheorem{rem}[thm]{\protect\remarkname}

\usepackage{commath}

\usepackage{tikz}
\usetikzlibrary{quotes,angles,positioning,calc}
\usepackage{colortbl}
\usepackage{transparent}

\usepackage{color}

\newcommand{\xyR}[1]{\xydef@\xymatrixrowsep@{#1}}
\newcommand{\xyC}[1]{\xydef@\xymatrixcolsep@{#1}}

\definecolor{green}{RGB}{0, 128, 0}

\usepackage[export]{adjustbox}

\@ifundefined{showcaptionsetup}{}{%
 \PassOptionsToPackage{caption=false}{subfig}}
\usepackage{subfig}
\makeatother

\providecommand{\corollaryname}{Corollary}
\providecommand{\lemmaname}{Lemma}
\providecommand{\propositionname}{Proposition}
\providecommand{\remarkname}{Remark}
\providecommand{\theoremname}{Theorem}

\begin{document}
\title{Critical points in higher dimensions, I: Reverse order of periodic
orbit creations in the Lozi family}
\author{Dyi-Shing Ou\thanks{Faculty of Applied Mathematics, AGH University of Science and Technology,
Poland\protect \\
\hspace*{1.8em}dsou@agh.edu.pl\protect \\
\hspace*{1.8em}ORCID: \protect\href{https://orcid.org/0000-0003-0979-3724}{0000-0003-0979-3724}\protect \\
\hspace*{1.8em}© 2022 Dyi-Shing Ou. This is available under the arXiv.org
perpetual, non-exclusive license. All other rights reserved.}}
\maketitle
\begin{abstract}
We introduce a renormalization model which explains how the behavior
of a discrete-time continuous dynamical system changes as the dimension
of the system varies. The model applies to some two-dimensional systems,
including Hénon and Lozi maps. Here, we focus on the orientation preserving
Lozi family, a two-parameter family of continuous piecewise affine
maps, and treat the family as a perturbation of the tent family from
one to two dimensions.

First, we give a new prove that all periodic orbits can be classified
by using symbolic dynamics. For each coding, the associated periodic
orbit depends on the parameters analytically on the domain of existence.
The creation or annihilation of periodic orbits happens when there
is a border collision bifurcation. Next, we prove that the bifurcation
parameters of some types of periodic orbits form analytic curves in
the parameter space. This improves a theorem of Ishii (1997). Finally,
we use the model and the analytic curves to prove that, when the Lozi
family is arbitrary close to the tent family, the order of periodic
orbit creation reverses. This shows that a forcing relation (Guckenheimer
1979 and Collett and Eckmann 1980) on orbit creations breaks down
in two dimensions. In fact, the forcing relation does not have a continuation
to two dimensions even when the family is arbitrary close to one dimension. 
\end{abstract}
\begin{description}
\item [{Keywords:}] Dynamical systems, Lozi maps, symbolic dynamics, border
collision bifurcations
\end{description}
\global\long\def\C{\mathbb{C}}%

\global\long\def\R{\mathbb{R}}%

\global\long\def\Ih{I^{H}}%

\global\long\def\Iv{I^{V}}%

\global\long\def\ParaSpace{P}%

\global\long\def\ParaFull{P_{\operatorname{full}}}%

\global\long\def\ParaModel{P_{\operatorname{mod}}}%

\global\long\def\ParaNbd{P_{\operatorname{nbd}}}%

\global\long\def\Lozi{\boldsymbol{\Lambda}}%

\global\long\def\FormalPeriodic{\boldsymbol{\theta}}%

\global\long\def\EqBreak{\allowbreak}%

\global\long\def\Identity{\operatorname{Id}}%

\global\long\def\Interior#1{\operatornamewithlimits{Interior}(#1)}%

\global\long\def\Sign{\sigma}%

\global\long\def\Shift{T}%

\paragraph*{Acknowledgment}

This work was supported by the National Science Centre, Poland (NCN),
grant no. 2019/34/E/ST1/00237. The research topic was inspired from
a conversation with Liviana Palmisano during the conference ``Low-dimensional
and Complex Dynamics, 2019'' in Switzerland. The author thanks Jan
Boro\'{n}ski for discussions on formulating the arguments and improving
the context. The author thanks Sonja Štimac for suggesting the paper
\cite{BSV09} to improve the conclusions. 

\section{Introduction}

Studies show that the dimension of a dynamical system may affect the
behavior of the trajectories. One example from continuous-time dynamical
systems is the Poincaré--Bendixson theorem \cite{Po81,Po82,Be01}.
It implies that there is no chaos in two dimensions, whereas the Lorenz
attractor \cite{Lo63} gives an example of a chaotic system in three
dimensions. This means that being two dimension forces the system
to not have chaotic trajectories.

In discrete-time dynamical systems, there are also examples indicating
such constraints relief as the dimension increases from one to two.
First, the dimension may affect the number of attracting cycles. Singer
\cite{Si78} showed that a sufficiently smooth map on a compact interval
has at most a finite number of periodic attractors. On the contrary,
Newhouse and Robinson \cite{Ne74,Ne79,Ro83} showed that a smooth
map on a topological disc can have infinitely many periodic attractors.
Second, for continuous maps on a compact interval, the possible trajectories
are restricted by some forcing relations. A forcing relation means
that if a map has an orbit of type $A$, then it forces the map to
have an orbit of type $B$ whenever $B$ satisfies a prescribed forcing
condition $P(A)$. Sharkovsky \cite{Sa64} introduced a forcing relation
on the periods of periodic orbits given by the Sharkovsky ordering.
A boundary case of such ordering shows that the existence of a period
3 cycle forces the map to have cycles of any period. In fact, Li and
Yorke \cite{LY75} showed that the boundary case imposes the existence
of uncountably many points which are not asymptotically periodic.
For unimodal maps, the result was sharpened by using symbolic dynamics,
which is called the kneading theory \cite{MT88,CE80}. The phase space
is partitioned by the critical point, and orbits are encoded by the
partition. The encoding of an orbit is called an itinerary. A forcing
relation on unimodal maps is described by the itinerary of the critical
orbit \cite{Gu79,CE80}. In contrast to one dimension, for smooth
maps on a topological disc, it is not hard to show that there exists
a Hénon map \cite{He76} having only period one, two, and three cycles,
and all bounded orbits tend to one of the cycles \cite{Ou21}. Third,
the dimension may affect how the maps are classified. In one dimension,
it is possible to classify maps by a finite-parameter family \cite{Gu79,MT88}.
In two dimensions, a finite-parameter family is not enough to solve
the classification problem of a class of maps \cite{HMT17,CP18,BS21}. 

What make one-dimensional systems different from higher dimensional
systems are the number of critical (or turning) points. In one dimension,
the critical orbits govern the dynamical behavior of an interval map.
The number of periodic attractors is finite because the basin of each
attractor contains a critical point \cite{Si78}. Interval maps can
be classified by the itineraries of critical orbits \cite{Gu79,CE80,MT88}.
However, they are no longer true in two dimensions. This suggests
that a map in one dimension has only finitely many critical values,
while the number of them grows to infinitely many as the dimension
increases from one to two.

In this paper, we introduce a renormalization model to visualize the
critical values in two dimensions (Section \ref{sec:Renormalization Model}).
The idea was first announced by the author in a conference talk \cite{Ou21}.
The model applies to systems that mimic an unfolding of a homoclinic
tangency. This includes the Lozi \cite{Lo78} and Hénon \cite{He76}
families. The model gives explanations of aspects involving the change
of dimension, e.g., the number of sinks, the classification problem,
etc \cite{Ou21}. Here, we apply the model to orientation preserving
Lozi maps to explain that the forcing relation for unimodal maps no
longer holds in two dimensions.

The Lozi family is a two-parameter family of maps 
\[
\Lozi_{a,b}(x,y)=(-a|x|-by+(a-b-1),x)
\]
 where $a,b\in\mathbb{R}$ are the parameters. A Lozi map is orientation
preserving (resp. reversing) if $b>0$ (resp. $b<0$). It is a generalization
of the tent family 
\[
T_{a}(x)=-a|x|+(a-1)
\]
from one to two dimensions. A Lozi map is degenerate if $b=0$. When
$b=0$, a degenerate Lozi map is identified with the tent map having
the same parameter $a$. The parameter $b$ serves as the amount of
perturbation that is applied to the tent family. We study how the
dynamical behavior changes as we perturb the parameter $b$ near $b=0$.

When $b>0$, a Lozi map has infinitely many critical values $\{u_{m}\}_{m=2}^{\infty}$,
which depend continuously on the parameters. When $b=0$, all the
critical values degenerate into one: $u_{2}=u_{3}=\cdots$. Since
the Lozi family is a two-parameter family, the critical values $\{u_{m}(a,b)\}_{m=2}^{\infty}$
form a system of rank two. In other words, we can fully control two
critical values by perturbing the two parameters $a$ and $b$. To
illustrate the ideas, we use $u_{2}$ and $u_{3}$ to study when periodic
orbits appear, and prove the following theorem (a reformulation of
Theorem \ref{thm:Main theorem}). 
\begin{thm*}[The main theorem]
For all $\hat{b}\in(0,1)$, there exist $\overline{b}\in(0,\hat{b})$,
and two analytic curves $l_{2},l_{3}:[0,\overline{b}]\rightarrow(\sqrt{2},4)$
on the parameter space, such that the following properties hold:

For each $n\in\{2,3\}$, let $\ParaSpace\equiv\{(a,b);a>3b+1\text{ and }0\leq b\leq\overline{b}\}$
and $\ParaSpace_{n}=\{(a,b)\in\ParaSpace;a\geq l_{n}(b)\}$. The curve
$l_{n}$ splits the parameter space $\ParaSpace$ into two components:
$\ParaSpace_{n}$ and $\ParaSpace\backslash\ParaSpace_{n}$.
\begin{enumerate}
\item On $\ParaSpace_{n}$, there exists two analytic maps $\FormalPeriodic_{-,n},\FormalPeriodic_{+,n}:\ParaSpace_{n}\rightarrow\mathbb{R}^{2}$
such that $\FormalPeriodic_{-,n}(a,b)$ and $\FormalPeriodic_{+,n}(a,b)$
are periodic points of $\Lozi_{a,b}$ with the same period for all
$(a,b)\in\ParaSpace_{n}$. In fact, on the boundary $a=l_{n}(b)$,
the border collision bifurcation occurs and creates the two periodic
points. 
\item On $\ParaSpace\backslash\ParaSpace_{n}$, the periodic points do not
have a continuation.
\end{enumerate}
Moreover, the curves $l_{2}$ and $l_{3}$ have a unique intersection,
and the intersection is transversal.
\end{thm*}
In summary, the theorem says that the order of bifurcations reverses
for any small $b>0$. For unimodal maps, the forcing relation \cite{Gu79,CE80}
implies that the creation of periodic orbits obeys a particular ordering.
See Section \ref{subsec:Forcing relation (kneading)} for an explanation.
However, in the Lozi family, the ordering is reversed when $b=\overline{b}$.
Therefore, the forcing relation does not have a continuation in two
dimensions, even when the maps are arbitrary close to one dimension.

To prove the theorem, first we use the fact that all periodic orbits
of a Lozi map have an analytic continuation on the parameter space
whenever they exist. Milnor \cite[Proposition 3.1]{Is97a} showed
that all bounded orbits can be identified with itineraries by using
symbolic dynamics. A point is labeled by ``$-$'' and ``$+$''
according to the sign of the $x$-coordinate. A periodic point $\boldsymbol{z}$
with period $p$ is encoded by an itinerary $I$ with length $p$
according to the labeling of successive iterates. For each itinerary
$I$, $\Lozi_{a,b}$ has at most one $I$-periodic orbit $\FormalPeriodic$.
In addition, Ishii \cite[Section 4]{Is97a} showed that $\FormalPeriodic$
is saddle and depends analytically on the parameters $(a,b)$. Here,
we give a new proof (Theorem \ref{thm:Formal periodic orbits}) by
using the universal stable and unstable cones \cite{Mi80,Pr21}.

Second, to associate the periodic points with the renormalization
model (Corollary \ref{cor:Admissible periodic point in Cmn}), we
show that all orbits with interesting dynamical aspects are eventually
trapped inside a compact subset of the phase space (Theorem \ref{thm:global dynamics}).
This is a generalization of \cite{BSV09} from orientation reversing
maps to orientation preserving maps. In particular, we center on periodic
points $\FormalPeriodic_{\Sign,m,n}:\ParaSpace_{\Sign,m,n}\rightarrow\mathbb{R}^{2}$
satisfying the itineraries $\iota_{\Sign,m,n}=(+\underset{m-2}{\underbrace{-\cdots-}}++\underset{n-2}{\underbrace{-\cdots-}}\Sign)$,
where $\Sign\in\{-,+\}$ and $m,n\geq2$. For each $n\geq2$, $\{\FormalPeriodic_{\Sign,m,n}\}_{m\geq2}$
are the ones created by perturbing the critical value $u_{n}$ (Section
\ref{sec:Criteria of admissibility}). The pair $(\FormalPeriodic_{-,m,n},\FormalPeriodic_{+,m,n})$
is created or annihilated simultaneously at a parameter $(a,b)$ when
a border collision bifurcation \cite{Le59,NY92} occurs. The parameter
$(a,b)$ is called a $\iota_{\pm,m,n}$-bifurcation parameter.

Third, we introduce a geometrical criterion to search for the bifurcation
parameters (Proposition \ref{prop:Geometrical characterization of the BCB in C(m,n)}).
This is a method different from the pruning conditions introduced
by Ishii \cite{Is97a}. The pruning pair in Ishii's paper is defined
by the candidates of the stable and unstable manifolds, whereas here
the geometrical criterion is prescribed by the forward and backward
iterates of the critical locus $\{(x,y);x=0\}$.

Fourth, when $m$ is large enough, we show that the $\iota_{\pm,m,n}$-bifurcation
parameters form an analytic curve $a=l_{m,n}(b)$ near $b=0$, and
there are only creations but no annihilation as the parameter $a$
increases (Theorem \ref{thm:The boundary of admissibility}). This
gives an improvement of a theorem of Ishii \cite[Theorem 1.2(i)]{Is97b}
for some types of periodic orbits by using a different approach. Ishii
used the pruning conditions to prove that, for all types of periodic
orbits and near $b=0$, there are only creations but no annihilation
of periodic orbits as the parameter $a$ increases. He did not show
that the bifurcation parameters define a continuous curve. Here, we
used the geometrical criterion to prove that the $\iota_{\pm,m,n}$-bifurcation
parameters are actually the graph of an analytic curve.

Finally, by using the renormalization model and the parameter curves,
we show that, when $\overline{b}>0$ is arbitrary small, there exists
$m>0$ such that the two curves $l_{m,2}$ and $l_{m,3}$ have a unique
intersection on $(0,\overline{b})$. Figure \ref{fig:Intersections of bifurcation parameters}
is an illustration of such curves. The proof demonstrates how to control
the two critical values $u_{2}$ and $u_{3}$ by perturbing the two
parameters. An outline is given in Section \ref{sec:Renormalization Model}.

\begin{figure}
\center \includegraphics{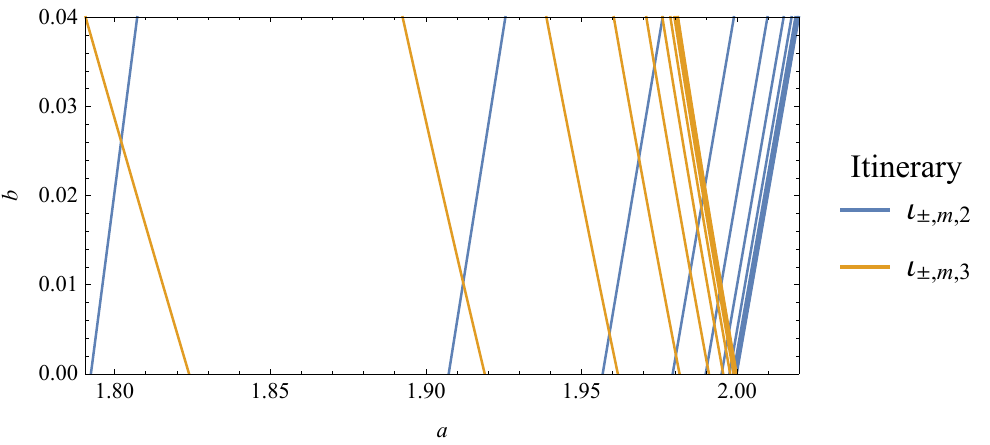}

\caption{\label{fig:Intersections of bifurcation parameters}Intersections
of $\iota_{\pm,m,n}$-bifurcation parameters. The positive-slope and
negative-slope curves are the $\iota_{\pm,m,2}$- and $\iota_{\pm,m,3}$-bifurcation
parameters respectively. For each $n$, the curves from left to right
are the ones with $m=4,5,\cdots,14$ in ascending order. For each
$m$, the two curves have a unique intersection.}
\end{figure}

Nevertheless, we still can find some patterns of the bifurcation order
from the renormalization model. If we fix the value $m$, we showed
in the main theorem and Figure \ref{fig:Intersections of bifurcation parameters}
that the bifurcation curves $l_{m,2}$ and $l_{m,3}$ intersect. Instead,
if we fix the value $n$, we see that the bifurcation curves of $\{\iota_{\pm,m,n}\}_{m=n+1}^{\infty}$
do not intersect. See Figure \ref{fig:No intersection of bifurcation parameters}.
This suggests that there might be a generalization of the forcing
relation that holds for some types of orbits. For example, Misiurewicz
and Štimac \cite{MS16} used countable many kneading sequences to
describe all possible itineraries of an orientation reversing Lozi
map. 

\begin{figure}
\begin{minipage}[t]{0.49\columnwidth}%
\subfloat[The $\iota_{\pm,m,2}$-bifurcation parameters. The curves are the
bifurcation parameters with $m=3,5,\cdots,14$ from left to right.]{\includegraphics[width=1\columnwidth]{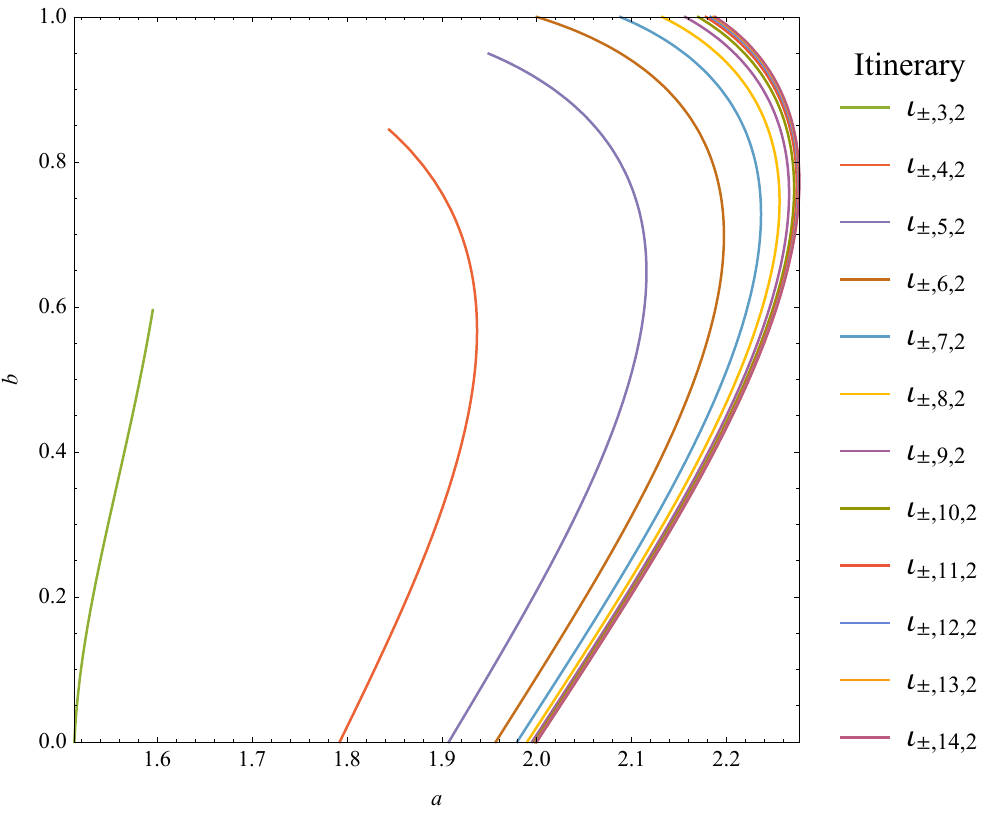}

}%
\end{minipage}\hfill{}%
\begin{minipage}[t]{0.49\columnwidth}%
\subfloat[The $\iota_{\pm,m,3}$-bifurcation parameters. The curves are the
bifurcation parameters with $m=4,5,\cdots,14$ from left to right.]{\includegraphics[width=1\columnwidth]{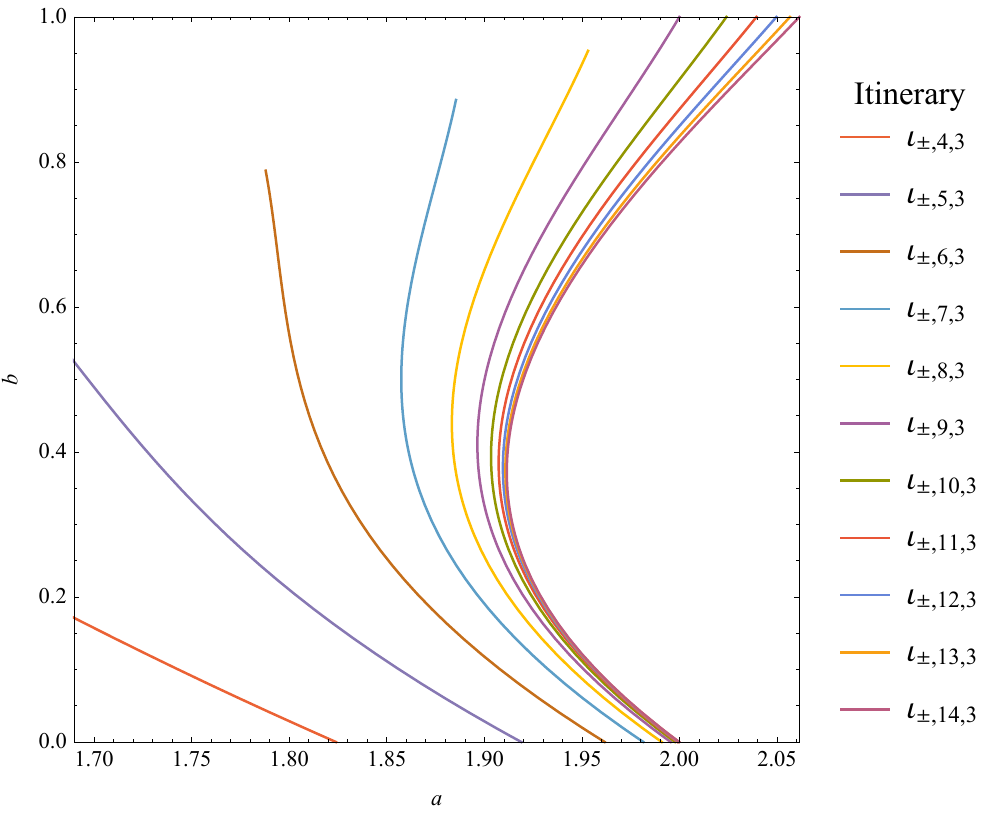}

}%
\end{minipage}

\caption{\label{fig:No intersection of bifurcation parameters}No intersection
of bifurcation parameters.}
\end{figure}

\section{\label{sec:Renormalization Model}The renormalization model}

Recall that the Lozi family is a two-parameter family of maps $\Lozi(x,y)=(-a|x|-by+(a-b-1),x)$,
where $a,b\in\mathbb{R}$ are the parameters. The map is expressed
as $\Lozi_{a,b}$ if we want to emphasize the Lozi map at a particular
parameter $(a,b)$. Let $S=\{-1,+1\}$, $\Sign\in S$, and $\mathbb{R}_{\Sign}=\Sign[0,\infty)$.
For simplicity, we may write $-$ and $+$ for the elements of $S$.
Denote the $\Sign$-affine branch of $\Lozi$ by $\Lozi_{\Sign}(x,y)=(-\Sign ax-by+(a-b-1),x)$.
Let $\ParaFull=\{(a,b);a>b+1\text{ and }0\leq b\leq1\}$ be a space
of parameters. Consider the original family $\widehat{\Lozi}_{a,b}(x,y)=(-a|x|+y+1,bx)$
introduced by Lozi \cite{Lo78}. For all $(a,b)\in\ParaFull$, we
have the semi-conjugation $\boldsymbol{H}\circ\Lozi_{a,b}=\widehat{\Lozi}_{a,-b}\circ\boldsymbol{H}$,
where $\boldsymbol{H}(x,y)=(\frac{x}{a-b-1},\frac{by}{a-b-1})$. The
map $\boldsymbol{H}$ becomes a conjugacy map when $b\neq0$.

When $(a,b)\in\ParaFull$, the map has two saddle fixed points $\boldsymbol{z}_{-}=(\zeta_{-},\zeta_{-})$
and $\boldsymbol{z}_{+}=(\zeta_{+},\zeta_{+})$, where $\zeta_{-}=-1$
and $\zeta_{+}=1-\frac{2(b+1)}{a+b+1}$. The stable and unstable multipliers
of $\boldsymbol{z}_{\Sign}$ are $-\Sign\mu$ and $-\Sign\lambda$,
with contracting and expanding directions $(-\Sign\mu,1)$ and $(-\Sign\lambda,1)$
respectively, where $\lambda=\frac{a+\sqrt{a^{2}-4b}}{2}$ and $\mu=\frac{b}{\lambda}$.
The stable $W^{S}(\boldsymbol{v})$ and unstable $W^{U}(\boldsymbol{v})$
sets of a periodic point $\boldsymbol{v}$ are unions of connected
line segments. We still call them stable and unstable manifolds, even
though they are not differentiable manifolds. For $d\in\{S,U\}$,
let $W_{0}^{d}(\boldsymbol{v})$ be the line segment of $W^{d}(\boldsymbol{v})$
containing $\boldsymbol{v}$.  
\begin{lem}
\label{lem:Expanding multiplier}If $(a,b)\in\ParaFull$, then $\lambda>1$
and $0\leq\mu<1$.
\end{lem}

Here, we introduce vertical segments and vertical strips to study
the geometry of a Lozi map. A line segment $\alpha$ is called a vertical
segment on an interval $\Iv$ if there exists an affine map $h_{\alpha}:\Iv\rightarrow\R$
such that $\alpha\cap(\R\times\Iv)=\{(h_{\alpha}(y),y);y\in\Iv\}$.
The vertical slope of $\alpha$ is the slope of $h_{\alpha}$. Suppose
that $\alpha$ and $\beta$ are disjoint vertical segments on the
interval $\Iv$.  The vertical strip $V(\alpha,\beta)\subset\R\times\Iv$
is the region bounded between $\alpha$ and $\beta$, including the
boundaries $\alpha$ and $\beta$. 

We define pullbacks of a vertical segment on $\mathbb{R}\times\Iv$,
where $\Iv$ is an interval with $0\in\Interior{\Iv}$. The image
$\Lozi(\mathbb{R}\times\Iv)$ is folded along the $x$-axis. Let $(u^{L},0)$
and $(u^{R},0)$ be the left and right boundary turning points of
$\Lozi(\mathbb{R}\times\Iv)$ respectively. Suppose that $\omega$
is a vertical segment and it intersects the $x$-axis at $(w,0)$.
Let $L(\omega)$ be the line containing $\omega$. If $u^{L}\geq w$
and $L(\omega)\cap\Lozi(\mathbb{R}\times\Iv)\subset\mathbb{R}\times\Iv$,
then $\Lozi^{-1}(\omega)\cap(\mathbb{R}\times\Iv)$ contains two vertical
segments: one is $\Pi_{-}(\omega)$ and the other is $\Pi_{+}(\omega)$,
where $\Pi_{\Sign}(\omega)=\Lozi_{\Sign}^{-1}(\omega)\cap(\mathbb{R}\times\Iv)$.
See Proposition \ref{prop:Pullback of a vertical segment} for details.
Thus, the transformations $\Pi_{-}$ and $\Pi_{+}$ define pullbacks
of a vertical segment by the two branches of $\Lozi$.

We define vertical strips on $\mathbb{R}\times\Iv$. We consider only
the orientation preserving case in this paper, i.e. $b\geq0$. Let
$\Iv=[-1,1]$. Suppose that $W_{0}^{S}(\boldsymbol{z}_{-})$ and $W_{0}^{S}(\boldsymbol{z}_{+})$
are vertical segments on $\Iv$. Let $\beta_{1}=W_{0}^{S}(\boldsymbol{z}_{+})\cap(\mathbb{R}\times\Iv)$,
$\beta_{m}=\Pi_{-}(\beta_{m-1})$ for $2\leq m<\infty$, $\beta_{\infty}=W_{0}^{S}(\boldsymbol{z}_{-})\cap(\mathbb{R}\times\Iv)$,
$\gamma_{m}=\Pi_{+}(\beta_{m})$ for $1\leq m<\infty$, and $\gamma_{\infty}=\Pi_{+}(\beta_{\infty})$.
Also, for $1\leq m\leq\infty$, let $(r_{m},0)$ be the intersection
point of $\gamma_{m}$ and the $x$-axis. Note that $\beta_{1}=\gamma_{1}$.
The vertical segments $\{\beta_{m}\}_{1\leq m<\infty}$ and $\{\gamma_{m}\}_{1\leq m<\infty}$
are subsets of $W^{S}(\boldsymbol{z}_{+})$, while $\beta_{\infty}$
and $\gamma_{\infty}$ are subsets of $W^{S}(\boldsymbol{z}_{-})$.
Let $B=V(\beta_{2},\beta_{1})$, $C=V(\gamma_{1},\gamma_{\infty})$,
$C^{+}=\{(x,y)\in\mathbb{R}\times\Iv;\gamma_{\infty}(y)\leq x\}$,
$C_{m}=V(\gamma_{m-1},\gamma_{m})$ for $2\leq m<\infty$, and $D=V(\beta_{\infty},\gamma_{\infty})$.
See Figure \ref{fig:Renormalization_Model} for an illustration. The
sets $\{C_{m}\}_{2\leq m<\infty}$ form a partition of $C$. 

\begin{figure}
\subfloat[\label{fig:Beta_and_Gamma}The $\beta$ and $\gamma$ stable manifolds.]{\includegraphics[scale=0.7]{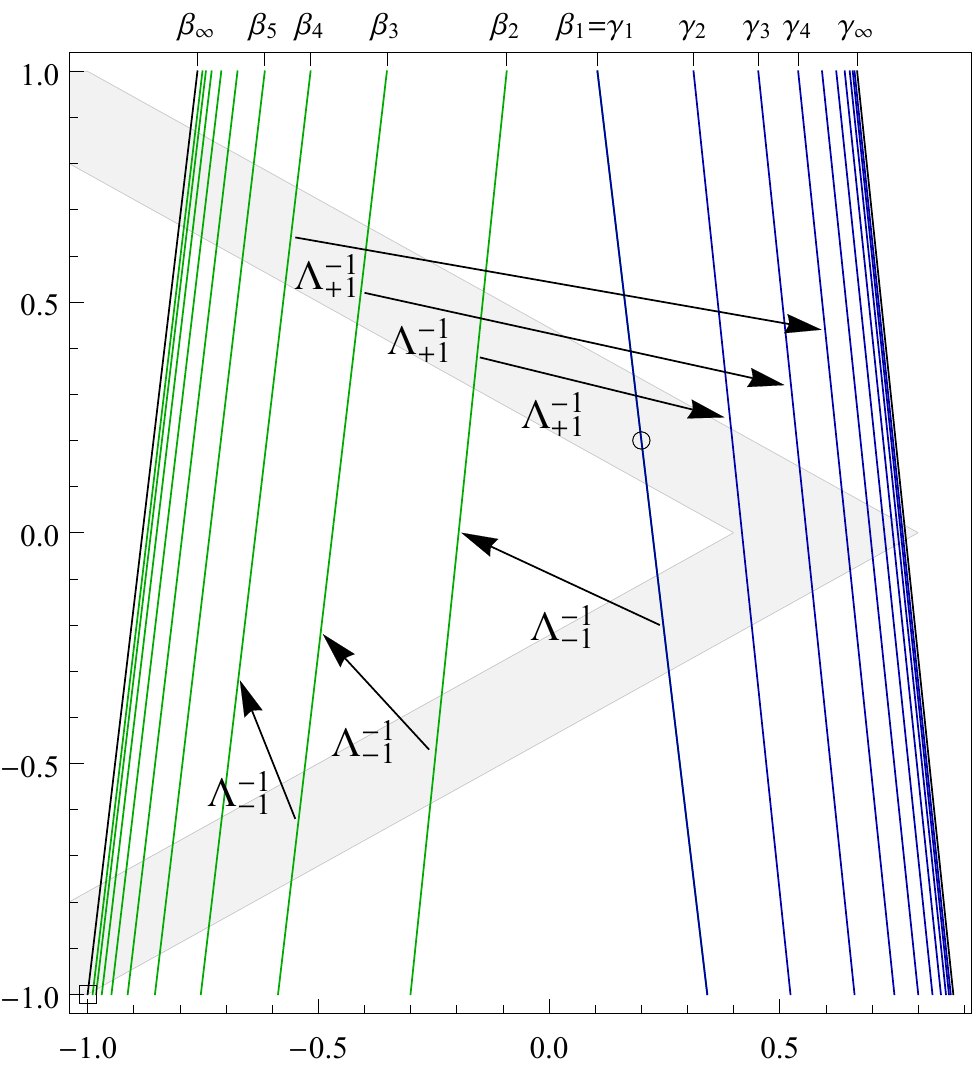}

}\hfill{}\subfloat[The sets $B$ and $C_{m}$.]{\includegraphics[scale=0.7]{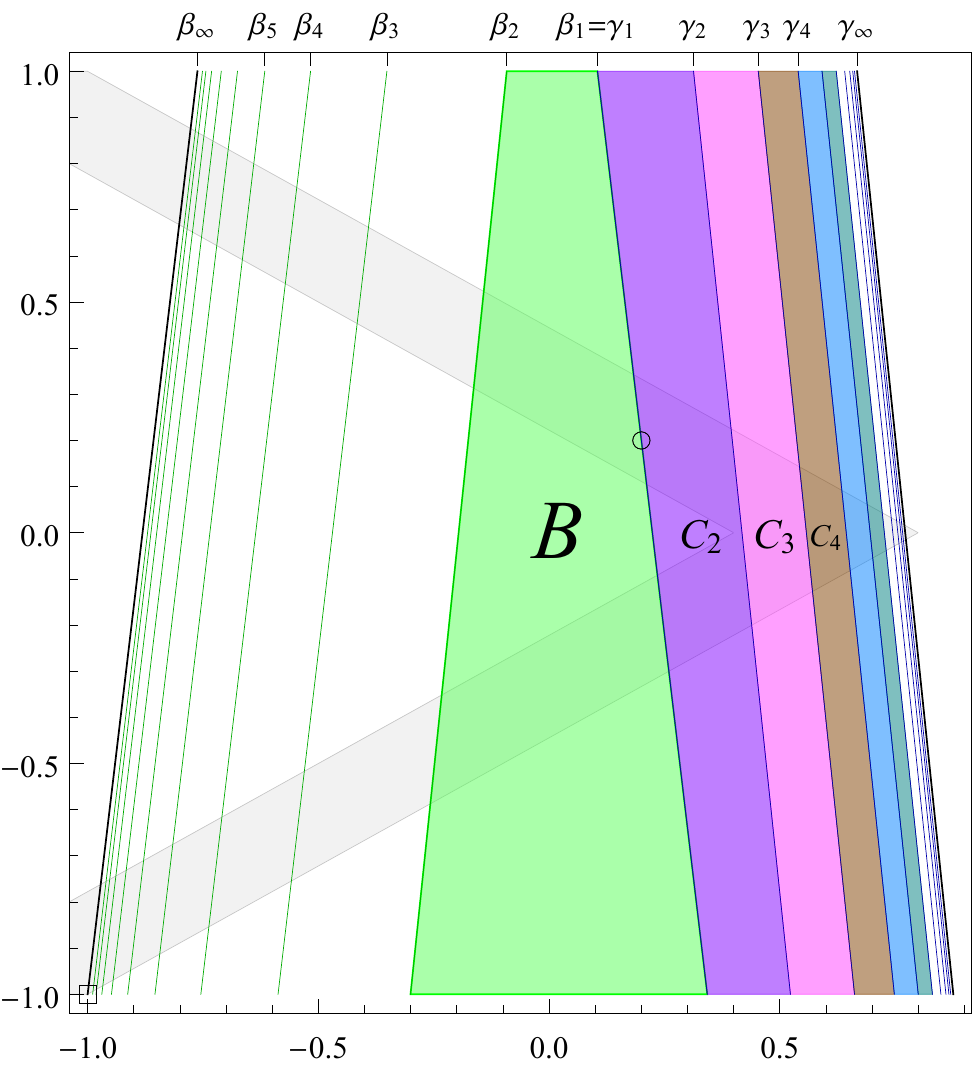}

}

\caption{\label{fig:Renormalization_Model}The renormalization model of the
Lozi map with $(a,b)=(1.8,0.2)$.}
\end{figure}

In this paper, we center on the renormalization model defined by this
partition. Let $\ParaModel=\{(a,b);a>3b+1\text{ and }0\leq b\leq1\}$.
We show that the renormalization model exists when $(a,b)\in\ParaModel$.

\begin{lem}
\label{lem:Bounds of lambda}We have 
\[
2b+1<\lambda\leq a
\]
for all $(a,b)\in\ParaModel$.
\end{lem}

\begin{lem}
\label{lem:Existence of beta_infty, gamma_infty}If $(a,b)\in\ParaModel$,
then $\beta_{\infty}$ and $\gamma_{\infty}$ exist. 

Moreover, let $L=[\beta_{\infty}(\min\Iv),\gamma_{\infty}(\min\Iv)]\times\{\min\Iv\}$,
$K=[\beta_{\infty}(\max\Iv),\gamma_{\infty}(\max\Iv)]\times\{\max\Iv\}$,
$L_{\Sign}=L\cap(\mathbb{R}_{\Sign}\times\mathbb{R})$, and $K_{\Sign}=K\cap(\mathbb{R}_{\Sign}\times\mathbb{R})$
for $\Sign\in S$. For each $J\in\{K,L\}$, $\Lozi(J)$ is the union
of two line segments $\Lozi(J)=\Lozi_{-}(J_{-})\cup\Lozi_{+}(J_{+})$.
The segment $\Lozi_{+}(J_{+})$ is located on the upper half plan;
while the segment $\Lozi_{-}(J_{-})$ is located on the lower half
plan. The left ends of the segments are on $\beta_{\infty}$ and the
right ends of the segments are on the $x$-axis.
\end{lem}

\begin{proof}
Let $v_{-}=-1+\frac{2b}{\lambda}$, $\boldsymbol{V}_{-}=(v_{-},1)$,
and $\boldsymbol{W}_{-}=z_{-}$. Clearly, $\boldsymbol{V}_{-},\boldsymbol{W}_{-}\in L(W_{0}^{S}(\boldsymbol{z}_{-}))$.
By Lemma \ref{lem:Bounds of lambda}, we have 
\[
v_{-}<-1+\frac{2b}{2b+1}<0.
\]
Thus, $\beta_{\infty}$ exists and $\beta_{\infty}=\overline{\boldsymbol{V}_{-}\boldsymbol{W}_{-}}$. 

Let $v_{+}=1-\frac{2\lambda+2}{a\lambda+b}b$, $w_{+}=1-\frac{2}{a\lambda+b}b$,
$\boldsymbol{V}_{+}=(v_{+},1)$, and $\boldsymbol{W}_{+}=(w_{+},-1)$.
By Lemma \ref{lem:Bounds of lambda} and the inequality of arithmetic
and geometric means, we have
\[
v_{+}=1-\frac{2b(\lambda+1)}{\lambda^{2}+2b}\geq1-\frac{2b(\lambda+1)}{\lambda^{2}}\geq1-\frac{2b(2b+2)}{(2b+1)^{2}}\geq0.
\]
Clearly, $v_{+}\leq w_{+}\leq1$ and hence $\Lozi(\boldsymbol{V}_{+}),\Lozi(\boldsymbol{W}_{+})\in\beta_{\infty}$.
Thus, $\gamma_{\infty}$ exists and $\gamma_{\infty}=\overline{\boldsymbol{V}_{+}\boldsymbol{W}_{+}}$.

Finally, by definition, we have $K=\overline{\boldsymbol{V}_{-}\boldsymbol{V}_{+}}$,
$L=\overline{\boldsymbol{W}_{-}\boldsymbol{W}_{+}}$, and $\Lozi(\boldsymbol{V}_{-}),\Lozi(\boldsymbol{V}_{+}),\Lozi(\boldsymbol{W}_{-}),\Lozi(\boldsymbol{W}_{+})\in\beta_{\infty}$.
This completes the proof.
\end{proof}
\begin{cor}
\label{cor:The image of D}Let $(a,b)\in\ParaModel$. Then $\Lozi(D)\subset\{(x,y)\in\mathbb{R}\times\Iv;x\geq\beta_{\infty}(y)\}$.
\end{cor}

\begin{proof}
The corollary follows immediately from Lemma \ref{lem:Existence of beta_infty, gamma_infty}.
\end{proof}
\begin{prop}
\label{prop:Pullback of a vertical segment}Let $(a,b)\in\ParaModel$,
$\omega\subset D$ be a vertical segment on $\Iv$, and $(w,0)$ be
the intersection point of $\omega$ and the $x$-axis. If $w\leq u^{L}$,
then $\Pi_{\Sign}(\omega)\subset D\cap(\mathbb{R}_{\Sign}\times\mathbb{R})$
is a vertical segment on $\Iv$ for each $\Sign\in S$. We have $\Lozi^{-1}(\omega)\cap(\mathbb{R}\times\Iv)=\Pi_{-}(\omega)\cup\Pi_{+}(\omega)$.
\end{prop}

\begin{proof}
Let $K_{\Sign}$ and $L_{\Sign}$ be the line segments defined in
Lemma \ref{lem:Existence of beta_infty, gamma_infty}. If $w\leq u^{L}$,
then $\omega$ and $\Lozi(J_{\Sign})$ have a unique intersection
point on $\mathbb{R}\times\mathbb{R}_{\Sign}$ for $J\in\{K,L\}$
by Lemma \ref{lem:Existence of beta_infty, gamma_infty}. The preimage
$\Lozi^{-1}(\omega\cap(\mathbb{R}\times\mathbb{R}_{\Sign}))$ is a
line segment connecting $K_{\Sign}$ and $L_{\Sign}$. Thus, $\Pi_{\Sign}(\omega)$
is a vertical segment on $\Iv$ and $\Pi_{\Sign}(\omega)=\Lozi^{-1}(\omega)\cap(\mathbb{R}_{\Sign}\times\Iv)\subset D\cap(\mathbb{R}_{\Sign}\times\mathbb{R})$.
\end{proof}
\begin{lem}
\label{lem:Existence of beta1}If $(a,b)\in\ParaModel$, then $\beta_{1}=\gamma_{1}$
exists and $r_{1}<u^{L}$.
\end{lem}

\begin{proof}
Clearly, $\beta_{1}=\{\zeta_{+}\}\times\Iv$ when $b=0$. When $b>0$,
$W_{0}^{S}(\boldsymbol{z}_{+})$ is the line segment from $(0,v_{1})$
to $(v_{1},v_{2})$ where $v_{1}=(1+\frac{\lambda}{b})\zeta_{+}$
and $v_{2}=[1-(\frac{\lambda}{b})^{2}]\zeta_{+}$. Note that $\zeta_{+}=\frac{(\lambda-1)(\lambda-b)}{(\lambda+1)(\lambda+b)}$.
By Lemma \ref{lem:Bounds of lambda}, we get 
\[
v_{1}=\frac{\lambda-1}{b}\frac{\lambda-b}{\lambda+1}>\frac{2b}{b}\frac{b+1}{2b+2}=\max\Iv.
\]
 Also, by Lemma \ref{lem:Bounds of lambda}, we get 
\[
v_{2}=-\frac{\lambda-1}{b}\frac{\lambda-b}{b}\frac{\lambda-b}{\lambda+1}<-\frac{2b}{b}\frac{b+1}{b}\frac{b+1}{2b+2}<\min\Iv.
\]
Thus, $\beta_{1}$ exist.

Moreover, by definition, we have $u^{L}=a-2b-1$ and $r_{1}=(1+\frac{b}{\lambda})\zeta_{+}$.
By Lemma \ref{lem:Bounds of lambda}, we get 
\[
u^{L}-r_{1}=\frac{\lambda}{\lambda+1}(\lambda-2b-1)>0.\qedhere
\]
\end{proof}
\begin{thm}
If $(a,b)\in\ParaModel$, then the renormalization model exists.
\end{thm}

\begin{proof}
The vertical segments $\beta_{\infty}$ and $\gamma_{\infty}$ exist
by Lemma \ref{lem:Existence of beta_infty, gamma_infty}.

We claim that $\beta_{m}$ exists and $w_{m}\leq u^{L}$ by induction
on $m\geq1$, where $(w_{m},0)$ is the intersection point of $\beta_{m}$
and the $x$-axis. The base case follows from Lemma \ref{lem:Existence of beta1}.
Suppose that $\beta_{m}$ is a vertical segment and $w_{m}\leq u^{L}$
for some $m\geq1$. By Proposition \ref{prop:Pullback of a vertical segment},
$\beta_{m+1}=\Pi_{-}(\beta_{m})$ exists and $w_{m+1}\leq0\leq u^{L}$.
Consequently, this proves the claim by induction.

By Proposition \ref{prop:Pullback of a vertical segment}, $\gamma_{m}=\Pi_{+}(\beta_{m})$
exists because $w_{m}\leq u^{L}$. This completes the proof.
\end{proof}
\begin{rem}
By choosing a different $\Iv$, we can show that the renormalization
model exists on $\{(a,b);a>3|b|+1\text{ and }-1\leq b\leq1\}$.
\end{rem}

Next, we study the orbit of $C_{m}$. Let $B_{m}=\Lozi^{m-1}(C_{m})$
and $U_{m}=\Lozi(B_{m})$. See Figures \ref{fig:B(m)} and \ref{fig:U(m)}
for illustrations. By the definition of the vertical segments, we
have $B_{m}\subset B$ for all $m$. The pieces $\{B_{m}\}_{2\leq m<\infty}$
converge to $W_{0}^{U}(\boldsymbol{z}_{-})$ exponentially.  The
$m$-th iterate $U_{m}$ returns to $C$, and is folded along the
$x$-axis. Thus, the $m$-fold iterate $\Lozi^{m}:C_{m}\rightarrow U_{m}$
forms a ``Lozi-like map'' in a microscopic scale. Let $(u_{m}^{L},0)$
and $(u_{m}^{R},0)$ be the left and right boundary turning points
of $U_{m}$ respectively. And let $(u_{\infty},0)$ be the intersection
point of $W_{0}^{U}(\boldsymbol{z}_{-})$ and the $x$-axis. The values
$u_{m}^{L}$ and $u_{m}^{R}$ serve as the ``critical values'' of
$\Lozi^{m}|_{C_{m}}$. They converge to $u_{\infty}$ exponentially
as $m\rightarrow\infty$ when $b>0$, and degenerate to a single value
when $b=0$. The position of the critical values govern the dynamics
in a microscopic scale. If $r_{m-1}\leq u_{m}^{R}$, then $C_{m}\cap U_{m}\neq\emptyset$,
and the orbit of a point in $C_{m}$ may have a recurrence in the
set. This is called the renormalization defined by one return to $C$. 
\begin{prop}
Suppose that $(a,b)\in\ParaModel$. Then the followings are true.
\begin{enumerate}
\item $r_{m}<r_{n}<r_{\infty}$ for all $1\leq m<n$.
\item $u^{L}\leq u_{m}^{L}\leq u_{m}^{R}\leq u_{n}^{L}\leq u_{\infty}\leq u^{R}$
for all $2\leq m<n$.
\end{enumerate}
\end{prop}

\begin{proof}
The properties can be proved by using the techniques developed in
Sections \ref{subsec:Geometry-forward iterates} and \ref{subsec:Geometry-backward iterates}.
The details are left to the reader.
\end{proof}
\begin{rem}
If the map is orientation reversing and the renormalization model
exists, then the order of the turning points will be flipping sides
alternatively:
\[
u^{L}\leq u_{3}^{L}\leq u_{3}^{R}\leq u_{5}^{L}\leq u_{5}^{R}\leq\cdots\leq u_{\infty}\leq\cdots\leq u_{4}^{L}\leq u_{4}^{R}\leq u_{2}^{L}\leq u_{2}^{R}\leq u^{R}.
\]
\end{rem}

\begin{figure}
\begin{minipage}[c]{0.48\columnwidth}%
\subfloat[The sets $B_{m}$ on the phase space. The parameters of the map are
$(a,b)=(1.8,0.2)$.]{\includegraphics[scale=0.7]{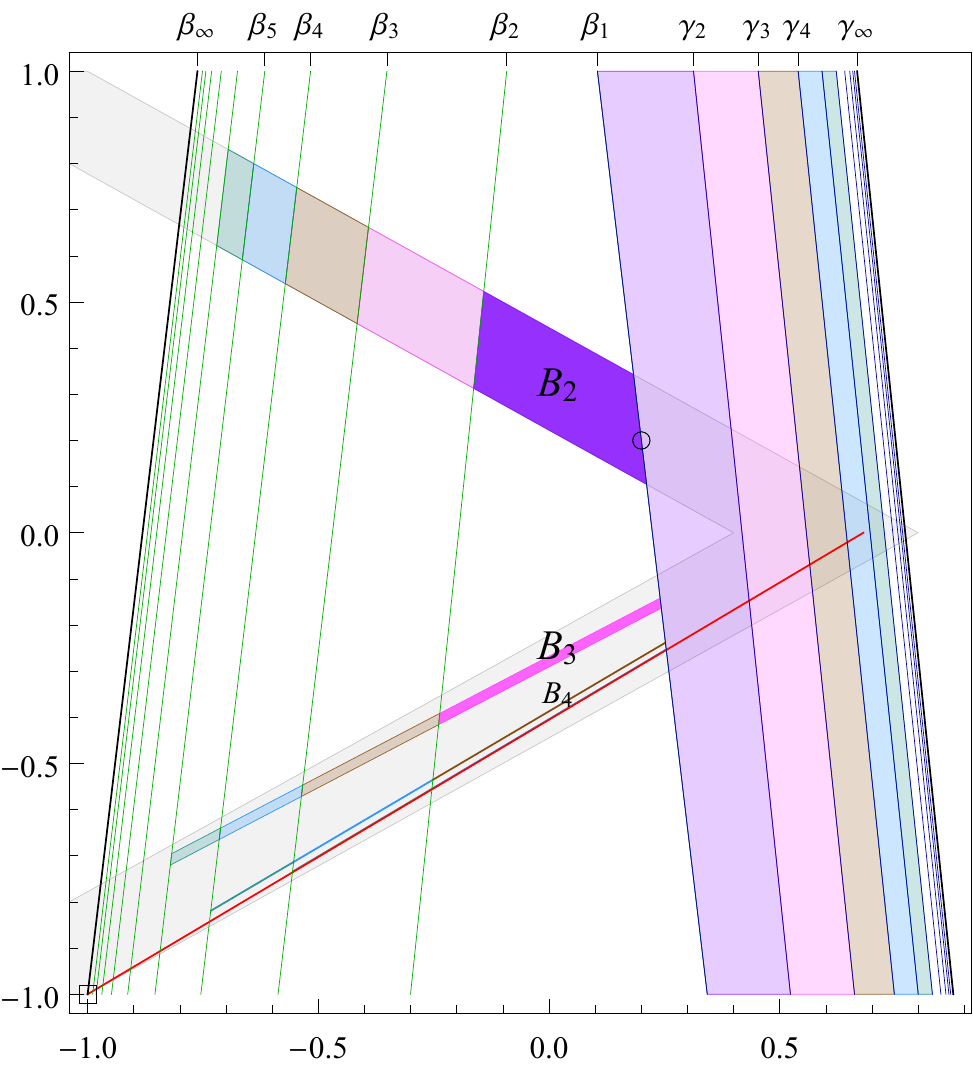}

}%
\end{minipage}\hfill{}%
\begin{minipage}[c]{0.48\columnwidth}%
\subfloat[The exponential convergence of $B_{m}$.]{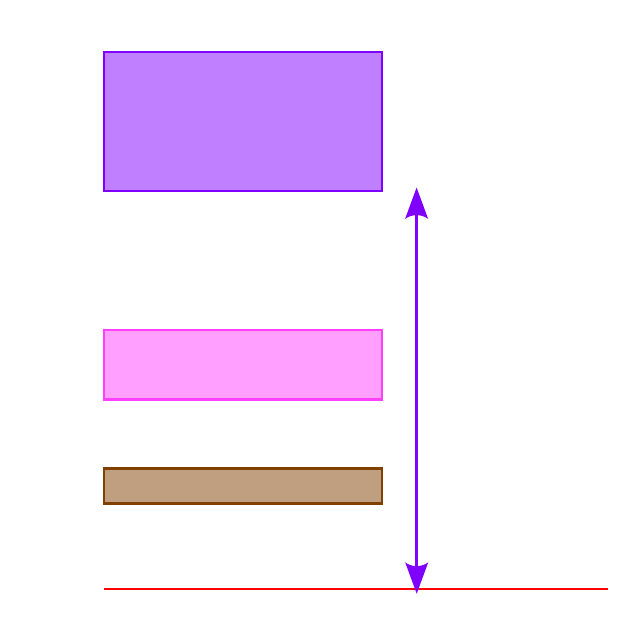

}%
\end{minipage}

\caption{\label{fig:B(m)}The sets $B_{m}$.}
\end{figure}

\begin{figure}
\begin{minipage}[c]{0.48\columnwidth}%
\subfloat[The sets $U_{m}$ on the phase space. The parameters of the map are
$(a,b)=(1.8,0.2)$.]{\includegraphics[scale=0.7]{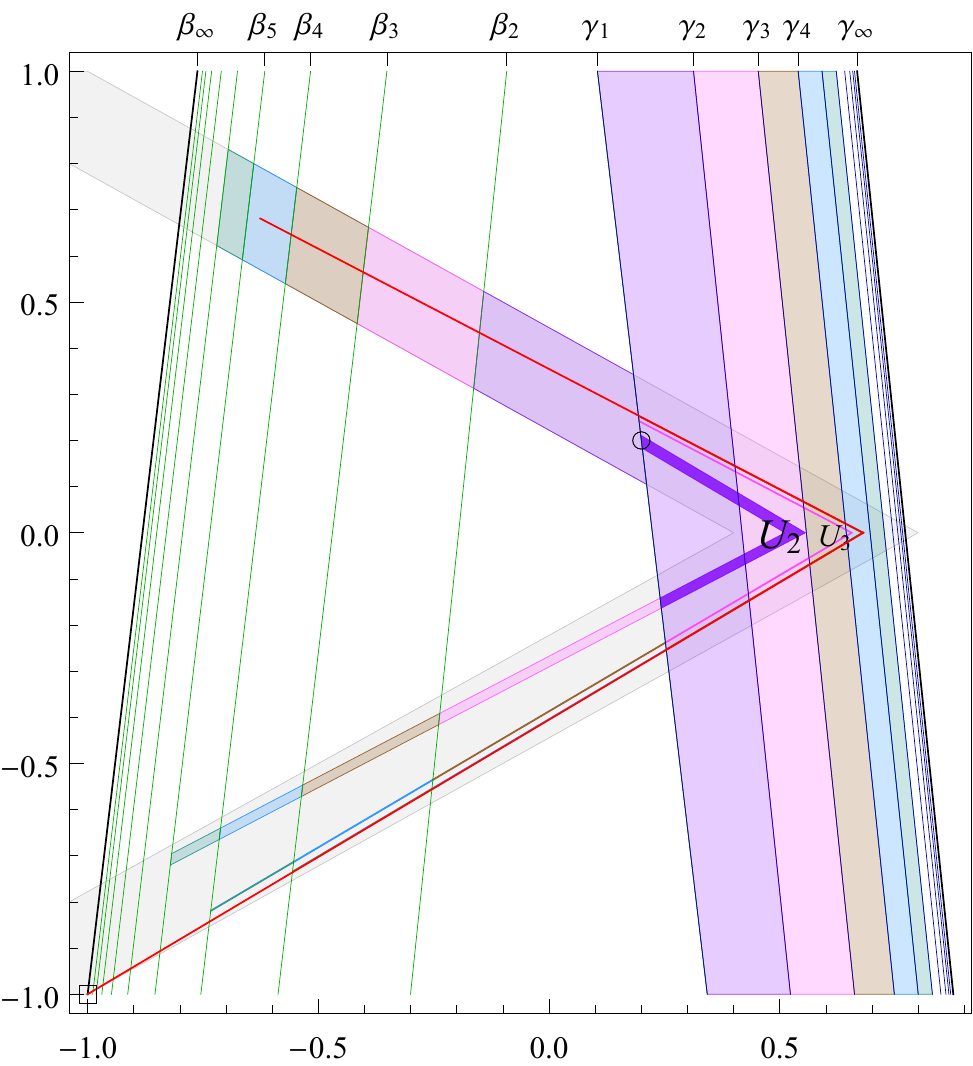}

}%
\end{minipage}\hfill{}%
\begin{minipage}[c]{0.48\columnwidth}%
\subfloat[The exponential convergence of $U_{m}$.]{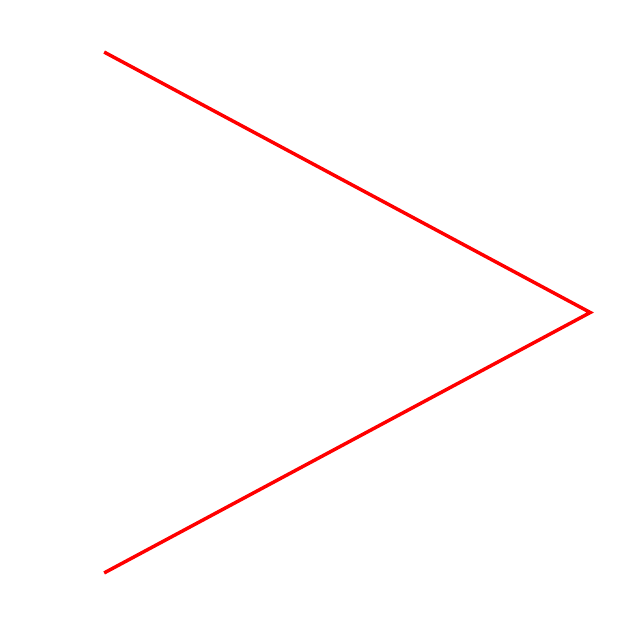

}%
\end{minipage}

\caption{\label{fig:U(m)}The sets $U_{m}$.}
\end{figure}

Moreover, we define a subpartition $\{C_{m,n}\}_{2\leq n<\infty}$
on $C_{m}$. Let $C_{m,n}=\Lozi^{-m}(U_{m}\cap C_{n})$ for $m,n\geq2$.
See Figure \ref{fig:The subpartition} for an illustration.
\begin{prop}
\label{prop:Left and right components of a subpartition}Suppose that
$(a,b)\in\ParaModel$ and $m,n\geq2$. If $r_{n}\leq u_{m}^{L}$,
then $C_{m,n}$ is the union of two disjoint vertical strips. The
vertical strips are bounded between vertical segments which are subsets
of $W^{S}(\boldsymbol{z}_{+})$. Let $C_{m,n}^{L}$ and $C_{m,n}^{R}$
be the left and right components respectively. 
\end{prop}

\begin{proof}
The proof is similar to Proposition \ref{prop:Pullback of a vertical segment}.
The details are left to the reader.
\end{proof}
If the conclusion of Proposition \ref{prop:Left and right components of a subpartition}
holds, let $B_{m,n}^{d}=\Lozi^{m+n-1}(C_{m,n}^{d})$ and $U_{m,n}^{d}=\Lozi(B_{m,n}^{d})$
for $d\in\{L,R\}$. We note that $B_{m,n}^{d}\subset B_{n}$ and $U_{m,n}^{d}\subset U_{n}$.
The image $\Lozi^{m+n}(C_{m,n}^{d})$ is folded along the $x$-axis.
Thus, the $(m+n)$-fold iterate $\Lozi^{m+n}:C_{m,n}^{d}\rightarrow U_{m,n}^{d}$
forms a ``Lozi-like map'' in a microscopic scale. If $C_{m,n}^{d}\cap U_{m,n}^{d}\neq\emptyset$,
then the orbit of a point in $C_{m,n}^{d}$ may have a recurrence
in the set. This is called the renormalization defined by two returns
to $C$. 

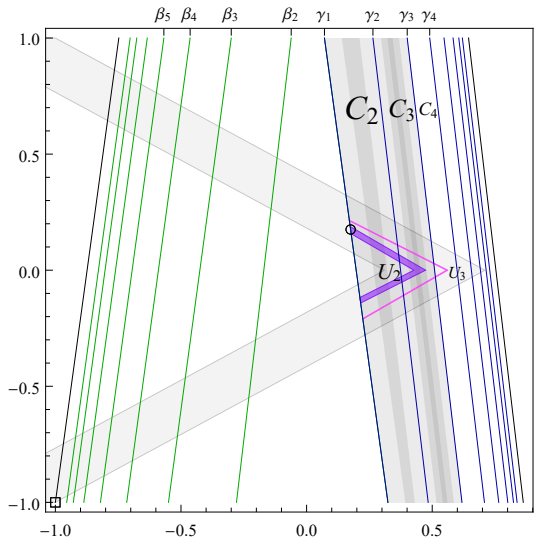
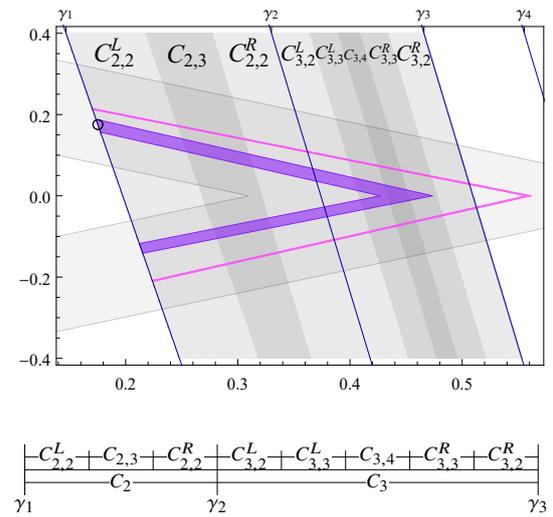
\begin{figure}
\begin{minipage}[c]{0.48\columnwidth}%
\subfloat[The subpartition, $U_{2}$, and $U_{3}$ on the phase space.]{\includegraphics[scale=0.7]{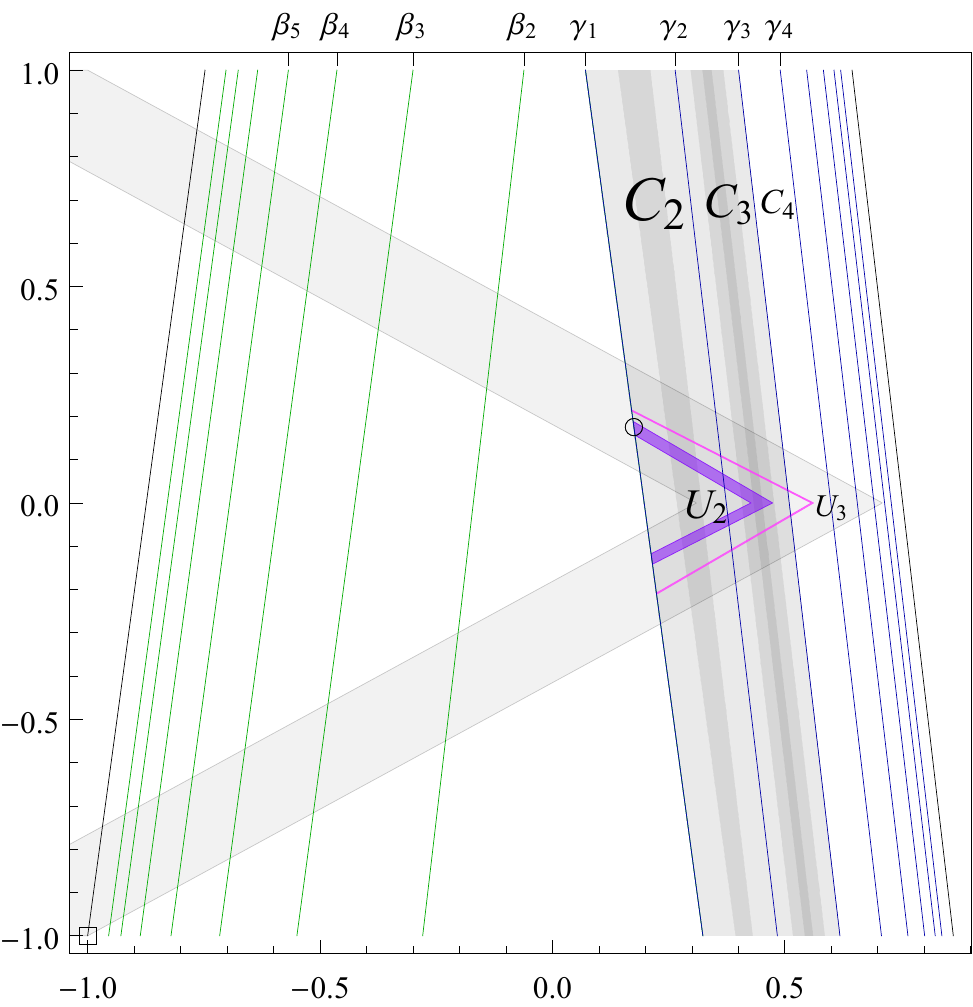}

}%
\end{minipage}\hfill{}%
\begin{minipage}[c]{0.48\columnwidth}%
\subfloat[A zoomed view of the subpartition.]{%
\begin{tabular}{c}
\includegraphics[scale=0.7]{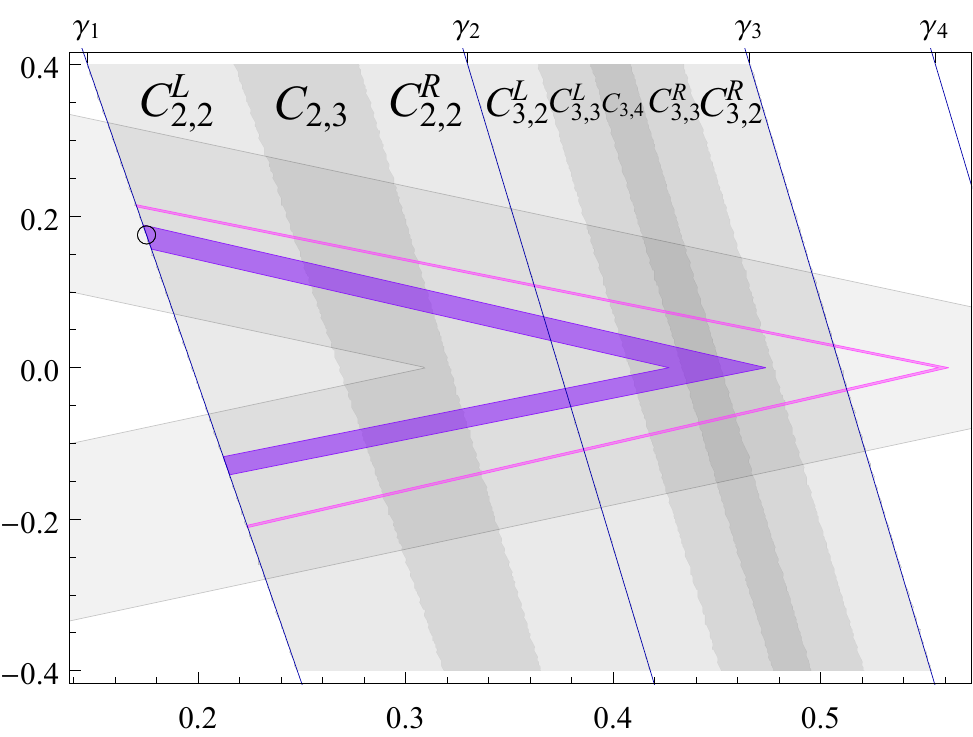}\tabularnewline
\smallskip{}
\tabularnewline
\resizebox{\columnwidth}{!}{%
\small\begin{tikzpicture}[scale=1,interval/.style={fill=white, pos=.5,inner sep=0,outer sep=0},>=stealth]

\coordinate (r1) at (2,0);
\draw ($(r1)-(0,0.4)$) -- ($(r1)+(0,0.4)$);
\node[below=0.35] at (r1) {$\gamma_{1}$};

\coordinate (r2) at (5,0);
\draw ($(r2)-(0,0.4)$) -- ($(r2)+(0,0.4)$);
\node[below=0.35] at (r2) {$\gamma_{2}$};

\coordinate (r3) at (10,0);
\draw ($(r3)-(0,0.4)$) -- ($(r3)+(0,0.4)$);
\node[below=0.35] at (r3) {$\gamma_{3}$};

\draw[-] ($(r1)-(0,0.2)$) to node[interval]{$C_2$} ($(r2)-(0,0.2)$);

\coordinate (left1) at (3,0);
\draw ($(left1)$) -- ($(left1)+(0,0.3)$);

\coordinate (left2) at (4,0);
\draw ($(left2)$) -- ($(left2)+(0,0.3)$);

\draw[-] ($(r1)+(0,0.2)$) to node[interval]{$C^{L}_{2,2}$} ($(left1)+(0,0.2)$);

\draw[-] ($(left1)+(0,0.2)$) to node[interval]{$C_{2,3}$} ($(left2)+(0,0.2)$);

\draw[-] ($(left2)+(0,0.2)$) to node[interval]{$C^{R}_{2,2}$} ($(r2)+(0,0.2)$);

\draw[-] ($(r2)-(0,0.2)$) to node[interval]{$C_3$} ($(r3)-(0,0.2)$);

\coordinate (mid1) at (6,0);
\draw ($(mid1)$) -- ($(mid1)+(0,0.3)$);

\coordinate (mid2) at (7,0);
\draw ($(mid2)$) -- ($(mid2)+(0,0.3)$);

\coordinate (mid3) at (8,0);
\draw ($(mid3)$) -- ($(mid3)+(0,0.3)$);

\coordinate (mid4) at (9,0);
\draw ($(mid4)$) -- ($(mid4)+(0,0.3)$);

\draw[-] ($(r2)+(0,0.2)$) to node[interval]{$C^{L}_{3,2}$} ($(mid1)+(0,0.2)$);

\draw[-] ($(mid1)+(0,0.2)$) to node[interval]{$C^{L}_{3,3}$} ($(mid2)+(0,0.2)$);

\draw[-] ($(mid2)+(0,0.2)$) to node[interval]{$C_{3,4}$} ($(mid3)+(0,0.2)$);

\draw[-] ($(mid3)+(0,0.2)$) to node[interval]{$C^{R}_{3,3}$} ($(mid4)+(0,0.2)$);

\draw[-] ($(mid4)+(0,0.2)$) to node[interval]{$C^{R}_{3,2}$} ($(r3)+(0,0.2)$);

\draw (r1) -- (r3);

\end{tikzpicture}}\tabularnewline
\end{tabular}

}%
\end{minipage}

\caption{\label{fig:The subpartition}The subpartition that establishes the
renormalization defined by two returns to $C$. The parameters of
the map are $(a,b)=(1.71,0.2)$.}
\end{figure}

\subsection*{An outline of the proof of the main theorem (Theorem \ref{thm:Main theorem})}

We consider the two pairs of periodic points $\FormalPeriodic_{-,m,2},\FormalPeriodic_{+,m,2}\in C_{m,2}^{L}$
and $\FormalPeriodic_{-,m,3},\FormalPeriodic_{+,m,3}\in C_{m,3}^{L}$
created by the renormalization defined by two returns to $C$. The
points depend analytically on the parameters $(a,b)$ whenever they
exist (Theorem \ref{thm:Formal periodic orbits}). For each pair $(\FormalPeriodic_{-,m,n},\FormalPeriodic_{+,m,n})$,
the two periodic points are created when there is a border collision
bifurcation \cite{Le59,NY92}. The bifurcation parameters $(a,b)$
form an analytic curve $a=l_{m,n}(b)$ on the parameter space (Theorem
\ref{thm:The boundary of admissibility}). The existence of the curve
is proved by using a geometrical characterization of the bifurcation
given in Proposition \ref{prop:Geometrical characterization of the BCB in C(m,n)}.
Our goal is to show that the two curves $l_{m,2}$ and $l_{m,3}$
have a unique intersection and the intersection is transverse.

We consider a curve $a=t(b)$ on the parameter space such that the
Lozi map $\Lozi_{a,b}$ has a homoclinic tangency $(r_{\infty},0)=(u_{\infty},0)$
(Proposition \ref{prop:Curve of homoclinic tangency}). Both the stable
laminations $(r_{m},0)$ and the turning points $(u_{m}^{L,R},0)$
converge exponentially to the homoclinic point as $m\rightarrow\infty$
(Propositions \ref{prop:Exponential convergence of u} and \ref{prop:Exponential convergence of r}).
Thus, we apply the logarithm coordinate transformation $T(x)=-\log_{\lambda}(r_{\infty}-x)$
to the $x$-coordinate. 

Under the coordinate change, the stable laminations $T(r_{m})$ are
located on the integral points $m$; while the turning points $T(u_{n}^{L,R})$
are located on $(n-1)\log_{\lambda}b^{-1}$. Since $\lambda$ is bounded
on the parameter curve $t$, the estimation $\log_{\lambda}\overline{b}^{-1}\gg2$
holds when $\overline{b}>0$ is sufficiently small (the condition
(\ref{eq:Condition when b is small})). Thus, when $(a,b)=(t(\overline{b}),\overline{b})$,
there exists an integer $m$ such that 
\begin{equation}
T(u_{2}^{R})<m-1=T(r_{m-1})<m=T(r_{m})<T(u_{3}^{L}).\label{eq:Configuration for reverse combinatorics}
\end{equation}
The geometry of the points is illustrated as in Figure \ref{fig:Reverse combinatorics}. 

We show that the order of creation of the pairs $(\FormalPeriodic_{-,m,2},\FormalPeriodic_{+,m,2})$
and $(\FormalPeriodic_{-,m,3},\FormalPeriodic_{+,m,3})$ is opposite
on the parameter lines $b=0$ and $b=\overline{b}$ while we vary
$a$. On the one hand, when $(a,b)=(t(\overline{b}),\overline{b})$,
the relation (\ref{eq:Configuration for reverse combinatorics}) holds.
We have $U_{2}\cap C_{m,2}^{L}=\emptyset$ since $T(u_{2}^{R})<T(r_{m-1})$.
Thus, $\FormalPeriodic_{-,m,2}$ and $\FormalPeriodic_{+,m,2}$ do
not exist. Also, $\Lozi^{m+3}:C_{m,3}^{L}\rightarrow U_{3}$ forms
a full horseshoe since $T(u_{3}^{L})>T(r_{m})$. Thus, $\FormalPeriodic_{-,m,3}$
and $\FormalPeriodic_{+,m,3}$ exist. This shows that $(\FormalPeriodic_{-,m,3},\FormalPeriodic_{+,m,3})$
is created before $(\FormalPeriodic_{-,m,2},\FormalPeriodic_{+,m,2})$
on the line $b=\overline{b}$. On the other hand, when $b=0$, we
have $U=U_{2}=U_{3}$. As the parameter $a$ increases, $U$ will
first intersects $C_{m,2}^{L}$ then intersects $C_{m,3}^{L}$. This
implies that $(\FormalPeriodic_{-,m,2},\FormalPeriodic_{+,m,2})$
is created before $(\FormalPeriodic_{-,m,3},\FormalPeriodic_{+,m,3})$
on the line $b=0$. See also Section \ref{subsec:Forcing relation (kneading)}
for the forcing relation in one dimension. Therefore, the order of
creation is opposite, and hence $l_{m,2}$ and $l_{m,3}$ has an intersection. 

Moreover, the intersection is unique and transverse because $\frac{\dif l_{m,2}}{\dif b}>\frac{\dif l_{m,3}}{\dif b}$
when $b$ is small (Corollary \ref{cor:The boundary of admissibility}).

\begin{figure}
\center 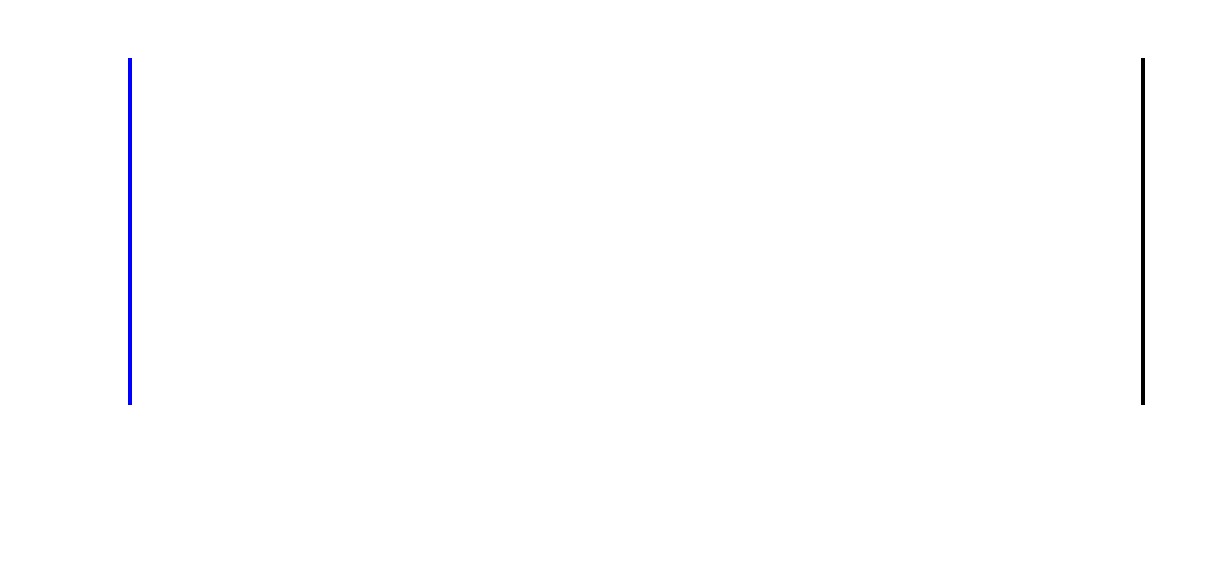

\caption{\label{fig:Reverse combinatorics}The configuration of a Lozi map
when the order of bifurcation is reversed.}
\end{figure}

\section{\label{sec:Formal-Periodic-Orbits}Symbolic dynamics and formal periodic
orbits}

We consider periodic orbits obtained from the two affine branches
$\Lozi_{-}$ and $\Lozi_{+}$. They are candidates of the periodic
orbits of the Lozi map. An itinerary $I=(I_{1},\cdots,I_{l(I)})$
is a sequence of alphabets in $S$ with a length $l(I)\in\{1,\cdots,\infty\}$.
An itinerary $I$ is finite if $l(I)<\infty$. When $I$ and $J$
are itineraries and $I$ is finite, we write $IJ$ for the concatenation
of $I$ and $J$, $I*$ for itineraries that start with $I$ and end
with any tail, and $I^{\infty}=II\cdots$. When $I$ is a finite itinerary,
we write $\Lozi_{I}=\Lozi_{I_{N}}\circ\cdots\circ\Lozi_{I_{1}}$,
where $N=l(I)$. The formal $I$-orbit $O_{I}(\FormalPeriodic)$ of
$\FormalPeriodic\in\mathbb{R}^{2}$ is the sequence $\{\FormalPeriodic_{m}\in\mathbb{R}^{2}\}_{m=0}^{l(I)}$
such that $\FormalPeriodic_{0}=\FormalPeriodic$ and 
\[
\FormalPeriodic_{m}=\Lozi_{(I_{1},\cdots,I_{m})}(\FormalPeriodic)
\]
for $1\leq m\leq l(I)$. A point $\FormalPeriodic\in\mathbb{R}^{2}$
is a formal $I$-periodic point if $I$ is finite and 
\begin{equation}
\FormalPeriodic=\Lozi_{I}(\FormalPeriodic).\label{eq:Formal Periodic Point}
\end{equation}
Its formal $I$-orbit is called a formal $I$-periodic orbit. 

Next, we introduce the notion of admissibility. Let $\FormalPeriodic$
be a point and $I$ be an itinerary. Let $\pi_{1}(x,y)=x$ and $\pi_{2}(x,y)=y$.
The point $\FormalPeriodic$ is $I$-admissible if $\pi_{1}(\FormalPeriodic)\in\R_{I_{1}}$
and $\pi_{1}\circ\Lozi_{(I_{1},\cdots,I_{m-1})}(\FormalPeriodic)\in\R_{I_{m}}$
for $m=2,\cdots,l(I)$. When $I$ is finite, the $I$-admissibility
function is defined as 
\begin{equation}
h_{I}(a,b,\FormalPeriodic)=\min\{I_{1}\pi_{1}(\FormalPeriodic),I_{2}\pi_{1}\circ\Lozi_{(I_{1})}(\FormalPeriodic),\cdots,I_{M}\pi_{1}\circ\Lozi_{(I_{1},\cdots,I_{M-1})}(\FormalPeriodic)\},\label{eq:Admissibility function}
\end{equation}
where $M=l(I)$. Thus, a point $\FormalPeriodic$ is $I$-admissible
at a parameter $(a,b)$ if and only if $h_{I}(a,b,\FormalPeriodic)\geq0$.
 If $\FormalPeriodic$ is $I$-admissible, then $\Lozi_{(I_{1},\cdots,I_{m})}(\FormalPeriodic)=\Lozi^{m}(\FormalPeriodic)$
for $m=1,\cdots,l(I)$. Thus, admissible formal periodic points are
periodic points of the Lozi map. We will show that for any given itinerary
$s$, there exists a unique formal $I$-periodic point. 

\subsection{Existence and uniqueness of formal periodic orbits}

To prove the existence of a formal periodic point, we write $\Lozi_{I}(\FormalPeriodic)=\boldsymbol{A}\FormalPeriodic-\boldsymbol{w}$,
where $I$ is a finite itinerary, $\boldsymbol{A}\equiv\Dif\Lozi_{I}$
is the linear term of the affine map $\Lozi_{I}$, and $\boldsymbol{w}$
is a vectored-polynomial in $a$ and $b$. Then (\ref{eq:Formal Periodic Point})
can be rewritten as a linear equation
\begin{equation}
(\boldsymbol{A}-\Identity)\FormalPeriodic=\boldsymbol{w}.\label{eq:Formal Periodic Point (Linear)}
\end{equation}
If the eigenvalues of $\boldsymbol{A}$ are not $1$, then (\ref{eq:Formal Periodic Point (Linear)})
has a unique solution, which yields the desired periodic point. To
ensure that the spectrum of $\Lozi_{-}$ is away from $1$, we restrict
our parameters to $\ParaFull$ in this section.
\begin{lem}
\label{lem:Spectral radius of the formal iterates}Suppose that $(a,b)\in\ParaFull$.
A lower bound of the spectral radius $\rho(\Dif\Lozi_{I})$ is provided
by 
\[
\rho(\Dif\Lozi_{I})\geq\lambda^{l(I)}.
\]
\end{lem}

\begin{proof}
By the expansion of the universal unstable cone (Theorem \ref{thm:Invaniance of the universal cones}),
we have 
\[
\left\Vert (\Dif\Lozi_{I})^{m}\right\Vert _{2}\geq\lambda^{l(I)m}
\]
for all $m\geq0$ where the norm is the $L^{2}$ operator norm. Thus,
the lemma follows from the Gelfand's formula \cite[P. 159]{La02}. 
\end{proof}
Finally, by using the fact that the spectrum is away from $1$, this
gives a new proof of the existence and uniqueness of a formal periodic
point.
\begin{thm}
\label{thm:Formal periodic orbits}Suppose that $(a,b)\in\ParaFull$.
For each finite itinerary $I$, there exists a unique $I$-formal
periodic point $\FormalPeriodic$. The point $\FormalPeriodic$ is
a saddle fixed point of $\Lozi_{I}$. In fact, consider the formal
periodic point as a map $\FormalPeriodic:\ParaFull\rightarrow\mathbb{R}^{2}$,
the maps $\pi_{1}\circ\FormalPeriodic$ and $\pi_{2}\circ\FormalPeriodic$
are rational functions in $a$ and $b$. 
\end{thm}

\begin{proof}
Let $\nu_{1}$ and $\nu_{2}$ be the eigenvalues of $\boldsymbol{A}$.
Without loss of generality, we assume that $|\nu_{1}|=\rho(\boldsymbol{A})$.
By Lemma \ref{lem:Expanding multiplier}, Corollary \ref{lem:Spectral radius of the formal iterates},
and $\det(\Dif\Lozi_{\Sign})=b$ for $\Sign\in S$, we get $|\nu_{1}|\geq\lambda^{M}>1$
and $|\nu_{2}|\leq\mu^{M}<1$ where $M=l(I)$. Thus, (\ref{eq:Formal Periodic Point (Linear)})
has a unique solution $\FormalPeriodic$, and $\FormalPeriodic$ is
a saddle fixed point of $\Lozi_{I}$. The solution is a rational function
because the entries of $\boldsymbol{A}$ and $\boldsymbol{w}$ are
polynomials in $a$ and $b$.
\end{proof}
\begin{rem}
Let $\hat{\lambda}=\frac{a+\sqrt{a^{2}-4|b|}}{2}$. By the existence
of the universal unstable cone, $\hat{\lambda}>1$, and $\frac{|b|}{\hat{\lambda}}<1$,
the theorem can be generalized to the parameter space $\{(a,b);a>|b|+1\}$.
\end{rem}

It follows immediately by the theorem and Theorem \ref{thm:global dynamics}
that formal periodic points are the only candidates having periodic
itineraries.
\begin{cor}
Suppose that $(a,b)\in\ParaFull$ and $I$ is a finite itinerary.
If $\Lozi$ has a $I^{\infty}$-admissible point $\boldsymbol{v}\in\mathbb{R}^{2}$,
then exactly one of the following holds.
\begin{enumerate}
\item If $\boldsymbol{v}$ has unbounded orbit, then $I^{\infty}=(-)^{\infty}$
and $\lim_{m\rightarrow\infty}\pi_{1}\circ\Lozi^{m}(\boldsymbol{v})=-\infty$.
\item If $\boldsymbol{v}$ has bounded orbit, then $\boldsymbol{v}$ is
the formal $I$-periodic point.
\end{enumerate}
\end{cor}

\subsection{Hyperbolicity and the border collision bifurcation}

A formal $I$-periodic point $\FormalPeriodic$ is called hyperbolic
if it is $I$-admissible and $\pi_{1}\circ\Lozi^{m}(\FormalPeriodic)\neq0$
for all $0\leq m<l(I)$. It is hyperbolic at a parameter $(a,b)$
if and only if $h_{I}(a,b,\FormalPeriodic(a,b))>0$. Thus, a hyperbolic
formal $I$-periodic point persists under a small perturbation of
a parameter.

Suppose that a formal $I$-periodic point $\FormalPeriodic$ is admissible
and $\pi_{1}\circ\Lozi^{m}(\FormalPeriodic)=0$ for some $m\geq0$
at a parameter $(a_{0},b_{0})$. This is exactly the case when $h_{I}(a_{0},b_{0},\FormalPeriodic(a_{0},b_{0}))=0$.
Without lose of generality, we may assume that $m=0$. Then $\FormalPeriodic$
is simultaneously $(-,I_{2},\cdots,I_{M})$- and $(+,I_{2},\cdots,I_{M})$-admissible
at $(a_{0},b_{0})$ where $M=l(I)$. After applying a small perturbation
to the parameter $(a_{0},b_{0})$, there may be a creation (or annihilation)
of two periodic orbits, each satisfying one of the itineraries. This
is called the border collision bifurcation \cite{Le59,NY92}. The
parameter $(a_{0},b_{0})$ is called a bifurcation parameter of the
itineraries $(\pm,I_{2},\cdots,I_{M})$.

\subsection{\label{subsec:Symbolic dynamics on the renormalization model}Symbolic
dynamics on the renormalization model}

In this section, we apply symbolic dynamics to periodic orbits created
by the renormalization defined by two returns to $C$. We find the
coding of periodic orbits in $C_{m,n}^{L}$. Recall that $\iota_{\Sign,m,n}=(+(-)^{m-2}++(-)^{n-2}\Sign)$
and $\FormalPeriodic_{\Sign,m.n}$ is the formal $\iota_{\Sign,m,n}$-periodic
point for $\Sign\in S$ and $m,n\geq2$.
\begin{prop}
\label{prop:Itinerary of a point in C_mn}Suppose that $(a,b)\in\ParaModel$,
$m,n\geq2$, and $C_{m,n}^{L}$ exists at $(a,b)$. If $\boldsymbol{v}\in C_{m,n}^{L}$,
then $\boldsymbol{v}$ is either $\iota_{-,m,n}$- or $\iota_{+,m,n}$-admissible
and $\pi_{1}\circ\Lozi^{j}(\boldsymbol{v})\neq0$ for all $j\in\{1,\cdots,m-2,m,\cdots,m+n-2\}$.
In addition, if $\boldsymbol{v}$ is a periodic point, then $\pi_{1}\circ\Lozi^{j}(\boldsymbol{v})\neq0$
for $j\in\{0,\cdots,m+n-2\}$.
\end{prop}

Conversely, if the formal $\iota_{\Sign,m,n}$-periodic point is admissible,
we show that the point is constraint in $C_{m,n}^{L}$. 
\begin{lem}
\label{lem:+-+}Suppose that $(a,b)\in\ParaModel$ and $m\geq3$.
Let $I=(+(-)^{m-2}+)$. If $\boldsymbol{v}\in\mathbb{R}\times\Iv$
is $I$-admissible, then one of the following is true.
\begin{enumerate}
\item $\boldsymbol{v}\in C_{m}$ and $\Lozi^{m-1}(\boldsymbol{v})\in B$.
\item $\boldsymbol{v}\in C_{m-1}$ and $\Lozi^{m-1}(\boldsymbol{v})\in C\cup C^{+}$.
\end{enumerate}
\end{lem}

\begin{proof}
Let $A^{-}=\{(x,y)\in\mathbb{R}\times I^{v};x\leq\beta(y)\}$ and
$A_{m}=V(\beta_{m},\beta_{m-1})$ for $m\geq2$. By definition, $\Lozi^{m-2}(\boldsymbol{v})\in\mathbb{R}_{-}\times\Iv\subset A^{-}\cup(\cup_{m=2}^{\infty}A_{m})$.
Since $\Lozi(A^{-}),\Lozi(A_{n})\subset(-\infty,0)\times\mathbb{R}$
for all $n\geq4$, we have $\Lozi^{m-2}(\boldsymbol{v})\in A_{2}\cup A_{3}$.
This proves the lemma.
\end{proof}
\begin{lem}
\label{lem:++}Suppose that $(a,b)\in\ParaModel$. If $\boldsymbol{v}\in\mathbb{R}\times\Iv$
is $(++)$-admissible, then one of the following is true.
\begin{enumerate}
\item $\boldsymbol{v}\in B$ and $\Lozi(\boldsymbol{v})\in C\cup C^{+}$.
\item $\boldsymbol{v}\in C_{2}$ and $\Lozi(\boldsymbol{v})\in B$.
\end{enumerate}
\end{lem}

\begin{proof}
By definition, $\mathbb{R}_{+}\times\Iv\subset B\cup(\cup_{m=2}^{\infty}C_{m})\cup C^{+}$.
Since $\Lozi(C_{m}),\Lozi(C^{+})\subset(-\infty,0)\times\Iv$ for
all $m\geq3$, the lemma follows.
\end{proof}
\begin{cor}
\label{cor:+-++}Suppose that $(a,b)\in\ParaModel$ and $m\geq3$.
Let $I=(+(-)^{m-2}++)$. If $\boldsymbol{v}\in\mathbb{R}\times\Iv$
is $I$-admissible, then one of the following is true.
\begin{enumerate}
\item $\boldsymbol{v}\in C_{m}$ and $\Lozi^{m}(\boldsymbol{v})\in C\cup C^{+}$.
\item $\boldsymbol{v}\in C_{m-1}$ and $\Lozi^{m}(\boldsymbol{v})\in B$.
\end{enumerate}
\end{cor}

\begin{prop}
\label{prop:Admissible periodic point in C}Suppose that $(a,b)\in\ParaModel$,
$\Sign\in S$, $m\geq3$, and $n\geq2$. If $\FormalPeriodic_{\Sign,m,n}$
is admissible, then $\FormalPeriodic_{\Sign,m,n}\in C_{m}$ and $\Lozi^{m}(\FormalPeriodic_{\Sign,m,n})\in C_{n}$.
\end{prop}

\begin{proof}
By Corollary \ref{cor:Invariant sets are bounded in D}, we have $\FormalPeriodic_{\Sign,m,n}\in D\subset\mathbb{R}\times\Iv$.
Also, $\FormalPeriodic_{\Sign,m,n}\notin W^{S}(\boldsymbol{z}_{+})$
by assumption.

First, consider the case $n=2$. By Lemma \ref{lem:+-+}, we have
$\Lozi^{m+2}(\FormalPeriodic_{\Sign})=\FormalPeriodic\in C$. This
forces $\Lozi^{m+1}(\FormalPeriodic_{\Sign,m,n})\in B$. Thus, $\FormalPeriodic_{\Sign,m,n}\in C_{m}$
and $\Lozi^{m}(\FormalPeriodic_{\Sign,m,n})\in C_{2}$ by Corollary
\ref{cor:+-++}.

Now, assume that $n\geq3$. Then $\Lozi^{m}(\FormalPeriodic_{\Sign,m,n})$
is $(+(-)^{n-1}+)$-admissible if $\FormalPeriodic_{\Sign,m,n}$ is
$\iota_{-,m,n}$-admissible, or $(+(-)^{n-2}+)$-admissible if $\FormalPeriodic_{\Sign,m,n}$
is $\iota_{+,m,n}$-admissible. By Lemma \ref{lem:+-+}, we get $\Lozi^{m}(\FormalPeriodic_{\Sign,m,n})\in C_{n-1}\cup C_{n}\cup C_{n+1}\subset C$.
This forces $\FormalPeriodic_{\Sign,m,n}\in C_{m}$ by Corollary \ref{cor:+-++}.
We obtain $\Lozi^{m}(\FormalPeriodic_{\Sign,m,n})\in C_{n}$ because
$\Lozi^{m+n}(\FormalPeriodic_{\Sign,m,n})=\FormalPeriodic_{\Sign,m,n}$.
\end{proof}
\begin{cor}
\label{cor:Admissible periodic point in Cmn}Suppose that $(a,b)\in\ParaModel$,
$\Sign\in S$, $m\geq3$, $n\geq2$, and $C_{m,n}^{L}$ exists. Then
$\boldsymbol{v}\in\mathbb{R}^{2}$ is the admissible formal $\iota_{\Sign,m,n}$-periodic
point if and only if $\boldsymbol{v}\in C_{m,n}^{L}$ is a periodic
point with period $m+n$.
\end{cor}

\begin{proof}
By Proposition \ref{prop:Admissible periodic point in C}, we have
$\FormalPeriodic_{\Sign,m,n}\in C_{m,n}$. Since $C_{m,n}^{L}$ exists,
we have $C_{m,n}=C_{m,n}^{L}\cup C_{m,n}^{R}$. Note that $\Lozi^{m-1}(C_{m,n}^{R})\backslash W^{S}(\boldsymbol{z}_{+})\subset(-\infty,0)\times\Iv$
and $\FormalPeriodic_{\Sign,m,n}\notin W^{S}(\boldsymbol{z}_{+})$.
Therefore, $\FormalPeriodic_{\Sign,m,n}\in C_{m,n}^{L}$.

The converse follows immediately from Proposition \ref{prop:Itinerary of a point in C_mn}.
\end{proof}

\subsection{\label{subsec:Forcing relation (kneading)}The forcing relation from
the kneading theory}

For completeness, we give a brief review of the forcing relation on
itineraries for unimodal maps. The materials are based on \cite{CE80}.
We explain why Theorem \ref{thm:Main theorem} gives a counterexample
to the one-dimensional forcing relation. The remaining part of this
paper is independent of this section. 

A continuous map $f:[-1,1]\rightarrow[-1,1]$ is unimodal if it has
a unique maximum point $c\in(-1,1)$ such that $f(c)=1$, $f(1)=-1$,
and $f$ is monotone on each component of $[-1,1]\backslash\{c\}$.
Let $I$ be an itinerary. A point $x$ (or orbit $O(x)$) is $I$-admissible
if $I_{m}(f^{m-1}(x)-c)\geq0$ for $m\in\{1,\cdots,l(I)\}$. Here,
the critical point is allowed to be encoded as either ``$+$'' or
``$-$''. This is consistent with the definition for Lozi maps.
A periodic point $x\in[-1,1]$ is an $I$-periodic point if $x$ is
$I$-admissible and $l(I)$ is the period. An itinerary $I$ is irreducible
if it cannot be expressed as $I=J^{m}$ for some $m\geq2$ and a finite
itinerary $J$.
\begin{lem}
\label{lem:Reduction of a periodic itinerary}Suppose that $I$ and
$J$ are finite itineraries such that $I=J^{m}$ for some $m\geq1$.
If a unimodal map $f$ has an $I^{\infty}$-admissible point, then
it has a $J$-periodic point.
\end{lem}

\begin{proof}
The set containing all points with the same itinerary is a closed
interval. In fact, $f^{n}$ is monotone on the interval for all $n\geq1$.
The interval is called a homterval. Let $H$ be the homterval consisting
points that are $I^{\infty}$-admissible. The interval $H$ is nonempty
by the assumption. We have $f^{N}(H)\subset H$ where $N=l(J)$. Thus,
$f^{N}$ has a fixed point in $H$, which yields the desired $J$-periodic
point.
\end{proof}
To apply the kneading theory and take care of the critical point,
we extend the symbolic space by letting $\widehat{S}=\{-,0,+\}$.
A U-itinerary is an infinite sequence of alphabets in $\widehat{S}$.
The modified coding $\widehat{I}:[-1,1]\rightarrow\widehat{S}$ is
defined as 
\[
\widehat{I}(x)=\begin{cases}
- & \text{if }x<c,\\
0 & \text{if }x=c,\text{ and}\\
+ & \text{if }x>c.
\end{cases}
\]
In contrast to the usual itineraries, here the critical point is encoded
as ``0''. The U-itinerary of an orbit $O(x)$ is the sequence $\widehat{I}\circ O(x)=\{\widehat{I}_{m}(x)\equiv\widehat{I}\circ f^{m-1}(x)\}_{m\geq1}$.
Let $\Shift$ be the shift map. Then $\widehat{I}\circ O\circ f(x)=\Shift\circ\widehat{I}\circ O(x)$
for all $x\in[-1,1]$.

We define a total order $\precsim$ on the space of U-itineraries.
Let $I$ and $J$ be distinct U-itineraries. Then the U-itineraries
can be expressed as the form $I=KP*$ and $J=KQ*$, where $K$ is
a finite itinerary and $P,Q\in\widehat{S}$ such that $P\neq Q$ or
$P=Q=0$. For each finite itinerary $K$, let $\epsilon(K)=(-1)^{N}$,
where $N$ is the number of $+$ in $K$. We say that $I\prec J$
if $\epsilon(I)P<\epsilon(I)Q$. Let $\sim$ be an equivalence relation
such that $I\sim J$ if $I=J$ or $P=Q=0$. For each finite itinerary
$K$, $\Shift^{l(K)}$ is strictly monotone on the cylindrical set
$\{K*\}$ of U-itineraries with the orientation $\epsilon(K)$. Moreover,
the coding map $\widehat{I}\circ O$ is orientation preserving, i.e.
$\widehat{I}\circ O(x)\precsim\widehat{I}\circ O(y)$ if $x<y$ \cite[Lemma II.1.3]{CE80}. 

We are ready to state the forcing relation given by the itinerary
of the critical orbit.
\begin{thm}[{\cite[Theorem II.3.8]{CE80}. See also \cite[Proposition 2.3]{Gu79}}]
\label{thm:Forcing condition on kneading sequence}Let $f$ be unimodal
and $J$ be an itinerary. Also, let $I=\widehat{I}\circ O(1)$ if
$c$ is not periodic and $I=\min\{(KL)^{\infty},(KR)^{\infty}\}$
if $\widehat{I}\circ O(1)=(K0)^{\infty}$ for some finite itinerary
$K$. Suppose that $J$ satisfies the forcing conditions 
\[
\widehat{I}\circ O(-1)\precsim J\text{ and }\Shift^{m}(J)\prec I
\]
 for all $m\geq0$. Then there exists $x\in[-1,1]$ such that $\widehat{I}\circ O(x)=J$.
\end{thm}

The theorem deduces a forcing relation on the itineraries of periodic
orbits. An infinite itinerary $I$ is maximum if $\Shift^{m}(I)\precsim I$
for all $m\geq0$; nontrivial if $I\neq\{-^{\infty},+^{\infty}\}$.
\begin{lem}
\label{lem:Nontrivial maximum itinerary}If a nontrivial infinite
itinerary $I$ is maximum, then $I$ has the form $(+-*)$.
\end{lem}

\begin{lem}
\label{lem:Nontrivial maximum periodic itinerary}Let $I$ be a finite
itinerary. If $I^{\infty}$ is nontrivial and maximum, then $\Shift(I^{\infty})$
is minimum, i.e. $\Shift(I^{\infty})\precsim\Shift^{m}(I^{\infty})$
for all $m\geq0$.
\end{lem}

\begin{prop}
Let $I$ and $J$ be finite itineraries such that $I$ is nontrivial.
Suppose that $J$ satisfies the forcing conditions 
\[
\Shift(I^{\infty})\prec J\text{ and }\Shift^{m}(J^{\infty})\prec I^{\infty}
\]
for all $m\geq0$. If a unimodal map $f$ has an $I$-periodic point,
then it has a $J$-periodic point. 
\end{prop}

\begin{proof}
By Lemmas \ref{lem:Nontrivial maximum itinerary} and \ref{lem:Nontrivial maximum periodic itinerary},
we may assume without loss of generality that $I^{\infty}$ is maximum.
By Lemma \ref{lem:Reduction of a periodic itinerary}, we may further
assume that $I$ is irreducible. Let $x$ be an $I$-periodic point
and $N=l(I)$.

If $O(x)$ does not contain $c$, then $\widehat{I}\circ O(x)=I^{\infty}$.
We get $\widehat{I}\circ O(-1)\precsim\Shift(I^{\infty})$ and $I^{\infty}\precsim\widehat{I}\circ O(1)$
since $\widehat{I}\circ O$ is monotone and $x,f(x)\in(-1,1)$. Thus,
$J$ satisfies the forcing conditions in Theorem \ref{thm:Forcing condition on kneading sequence}.
Consequently, $f$ has a $J$-periodic point by Lemma \ref{lem:Reduction of a periodic itinerary}.

If $O(x)$ contains $c$, then $f(c)=x=1$, $f^{N-1}(x)=c$, and $f^{j}(x)\neq c$
for all $0\leq j\leq N-2$ by the assumption of maximum and irreducible.
We note that $N\geq3$ because $-1,1,c$ are distinct points in the
critical orbit. Write $I=(+Ks)$, where $s\in S$ and $K$ is a finite
itinerary. Then, $\widehat{I}\circ O(1)=(+K0)^{\infty}$. Here, we prove that $J$ satisfies one forcing condition in Theorem
\ref{thm:Forcing condition on kneading sequence}. The proof of the
other forcing condition is similar. Therefore, $f$ has a $J$-periodic
point by the same reason.

Suppose that there exists $m\geq0$ such that $\Shift^{m}(J^{\infty})\prec(+Ks)^{\infty}$,
but 
\[
\Shift^{m}(J^{\infty})\succsim\min\{(+K-)^{\infty},(+K+)^{\infty}\}.
\]
If $\Shift(K)=-1$, then $s=+$, $\Shift^{m}(J^{\infty})=(+K+L)$,
and $L\succ(+K+)^{\infty}$, where $L$ is an infinite itinerary.
However, this is a contradiction because 
\[
L=\Shift^{m+N}(J^{\infty})\prec I^{\infty}=(+K+)^{\infty}.
\]
The case when $\Shift(K)=+1$ is similar. Thus, $\Shift^{m}(J^{\infty})\succ\min\{(+K-)^{\infty},(+K+)^{\infty}\}$
for all $m\geq0$.
\end{proof}
Finally, we apply the forcing relation to periodic orbits created
by renormalization defined by two returns to $C$. By Lemma \ref{prop:Itinerary of a point in C_mn},
the itineraries of such orbits are given by $\iota_{\Sign,m,n}$ for
$\Sign\in S$ and $m,n\geq2$. Therefore, Theorem \ref{thm:Main theorem}
shows that the one-dimensional forcing relation cannot be extended
to two dimensions.  
\begin{cor}
\label{cor:Forcing relation on C_m,n (unimodal)}Suppose that $\Sign\in S$
and $m>n_{1}>n_{2}\geq2$. If a unimodal map has an $\iota_{+,m,n_{1}}$-periodic
point, then it has an $\iota_{\Sign,m,n_{2}}$-periodic point.
\end{cor}

\begin{rem}
In terms of the renormalization model, the vertical strips $C_{m,2}^{L},C_{m,3}^{L},\cdots$
are aligned from left to right whenever they exist. When a Lozi map
is degenerate, the vertical strips share the same critical value $u$.
Corollary \ref{cor:Forcing relation on C_m,n (unimodal)} is true
because the critical value $u$ moves from left to right as the parameter
$a$ increases.
\end{rem}

\section{\label{sec:Criteria of admissibility}Criteria of admissibility}

In this section, we find conditions such that $\Lozi$ has a periodic
orbit with periodic $m+n$ in $C_{m,n}^{L}$ for $m>n\geq2$. We study
the case when $U_{n}$ intersects $C_{m}$, i.e. $r_{m-1}\leq u_{n}^{R}$.
Since $n<m$ and the map is orientation preserving, we have $u_{m}^{L}\geq u_{n}^{R}\geq r_{m-1}\geq r_{n}$.
Thus, $C_{m,n}^{L}$ exists. We fix the value of $b\geq0$ and vary
$a$. The signs of the quantities $u_{n}^{L}-r_{m}$ and $u_{n}^{R}-r_{m-1}$
divide the parameter space of $a$ into three regions. 

\paragraph{Large values of $a$.}

First, we start with large values of $a$ such that $r_{m}\leq u_{n}^{L}$.
See Figure \ref{fig:Large a}. The map $\Lozi^{m+n}:C_{m,n}^{L}\rightarrow U_{n}$
forms a full horseshoe. Thus, it has two saddle fixed points $\FormalPeriodic_{-,m,n}$
and $\FormalPeriodic_{+,m,n}$ in $C_{m,n}^{L}$ which are the formal
periodic points (Theorem \ref{prop:Full horseshoe}). 
\begin{prop}
\label{prop:Full horseshoe}Suppose that $(a,b)\in\ParaModel$, $m,n\geq2$,
and $C_{m,n}^{L}$ exists at $(a,b)$. Let $V(\omega^{L},\omega^{R})=C_{m,n}^{L}$
where $\omega^{L}$ and $\omega^{R}$ are the left and right boundaries
respectively. Let $(w,0)$ be the intersection of $\omega^{R}$ and
the x-axis. If $w\leq u_{n}^{L}$, then $\FormalPeriodic_{-,m,n}$
and $\FormalPeriodic_{+,m,m}$ are admissible and $\FormalPeriodic_{-,m,n},\FormalPeriodic_{+,m,n}\in C_{m,n}^{L}$.
\end{prop}

\begin{proof}
By \cite{KY01,Ga02}, the map $\Lozi$ has two periodic points in
$C_{m,n}^{L}$ with disjoint orbits. By Proposition \ref{prop:Itinerary of a point in C_mn},
each satisfies one of the itineraries $\iota_{-,m,n}$ and $\iota_{+,m,n}$.
This yields that $\FormalPeriodic_{-,m,n}$ and $\FormalPeriodic_{+,m,n}$
are admissible by the uniqueness of formal periodic points (Theorem
\ref{thm:Formal periodic orbits}).
\end{proof}

\paragraph{Intermediate values of $a$.}

Next, we consider intermediate values of $a$ such that $r_{m-1}\leq u_{n}^{R}$
and $u_{n}^{L}\leq r_{m}$. See Figure \ref{fig:Mid a}. Let $\Sign\in S$.
Since the condition of admissibility (\ref{eq:Admissibility function})
is a closed condition, the formal periodic point $\FormalPeriodic_{\Sign,m,n}\in C_{m,n}^{L}$
has a largest admissible continuation on a closed interval of parameters
$a\in[\hat{a}_{\Sign},\infty)$. The boundary parameter $(\hat{a}_{\Sign},b)$
is a bifurcation parameter of $\iota_{\pm,m,n}$. At the bifurcation
parameter, we have $\hat{a}\equiv\hat{a}_{-}=\hat{a}_{+}$, $\FormalPeriodic_{+,m,n}(\hat{a},b)=\FormalPeriodic_{-,m,n}(\hat{a},b)$,
and $\pi_{1}\circ\Lozi^{m+n-1}(\hat{a},b,\FormalPeriodic_{\pm,m,n}(\hat{a},b))=0$
(Proposition \ref{prop:Itinerary of a point in C_mn}).

\paragraph{Small values of $a$.}

Finally, we consider the case when $a$ is small such that $u_{n}^{R}<r_{m-1}$.
See Figure \ref{fig:Small a} for an illustration. Then $C_{m}\cap U_{n}=\emptyset$
and hence $\FormalPeriodic_{\Sign,m,n}$ is not admissible. Therefore,
the border collision bifurcation happens in the intermediate region. 
\begin{cor}
\label{cor:Disjoint C and U implies nonadmissible parameters}Suppose
that $(a,b)\in\ParaModel$, $\Sign\in S$, $m\geq3$, and $n\geq2$.
If $C_{m}\cap U_{n}=\emptyset$, then $\FormalPeriodic_{\Sign,m,n}$
is not admissible.
\end{cor}

\begin{proof}
The corollary follows from Proposition \ref{prop:Admissible periodic point in C}.
\end{proof}
\begin{figure}
\subfloat[\label{fig:Small a}Small $a$.]{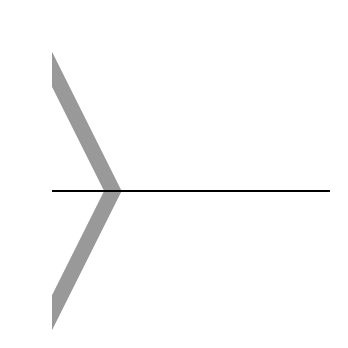

}\hfill{}\subfloat[\label{fig:Mid a}Intermediate $a$.]{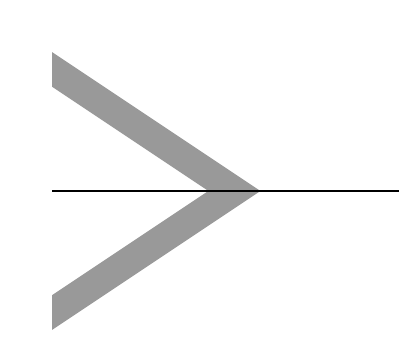

}\hfill{}\subfloat[\label{fig:Large a}Large $a$.]{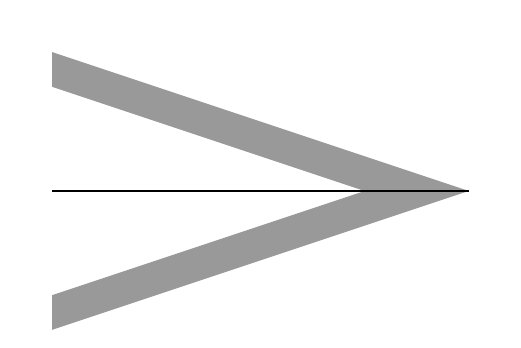

}

\caption{Position of the critical values.}
\end{figure}

We give a geometrical criterion that determines when the bifurcation
happens in $\Lozi^{m+n-1}(C_{m,n}^{L})\subset B$. First, consider
the horizontal line $\eta\equiv\mathbb{R}\times\{0\}$ and let $\eta_{m,n}=C_{m,n}^{L}\cap\eta$.
The iteration $\eta_{m,n}^{B}=\Lozi^{m+n-1}(\eta_{m,n})$ is a line
segment in $B$ with both ends attached to the boundaries $\beta_{1}$
and $\beta_{2}$. Then $\eta_{m,n}^{U}\equiv\Lozi^{m+n}(\eta_{m,n})$
is folded along the $x$-axis. Let $(p,0)$ be the turning point of
$\eta_{m,n}^{U}$. Next, consider the critical locus $\kappa=\{0\}\times I^{v}$.
The preimage $\kappa_{m,n}^{C}=(\Lozi^{m+n-1}|_{C_{m,n}^{L}})^{-1}\kappa$
is a vertical segment in $C_{m,n}^{L}$. Let $(q,0)$ be the intersection
point of $\kappa_{m,n}^{C}$ and the $x$-axis. See Figure \ref{fig:Definitions of p and q}
for an illustration. The points $p$ and $q$ are similar to the pruning
conditions defined by Ishii \cite[Definition 1.1]{Is97a} but not
the same. The pruning conditions in \cite{Is97a} are defined by the
candidates of the stable and unstable manifolds using the formal iterates.
In Theorem \ref{thm:The boundary of admissibility}, we will use the
values $p$ and $q$ to find bifurcation parameters, and prove that
the bifurcation parameters form an analytic curve in the parameter
space. 

\begin{figure}
\center 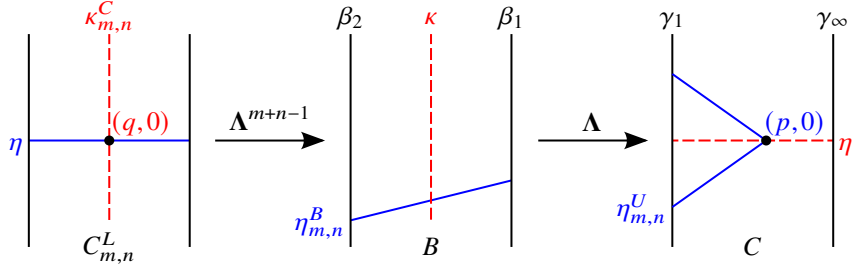

\caption{\label{fig:Definitions of p and q}Definitions of $p$ and $q$. Curves
with the same stroke are related by the maps indicated by the arrows.}
\end{figure}

\begin{prop}
\label{prop:Geometrical characterization of the BCB in C(m,n)}Suppose
that $(a,b)\in\ParaModel$, $m,n\geq2$, and $C_{m,n}^{L}$ exists
at $(a,b)$. There exist a periodic point $\boldsymbol{v}\in C_{m,n}^{L}$
with period $m+n$ such that $\pi_{1}\circ\Lozi^{m+n-1}(\boldsymbol{v})=0$
if and only if $p=q$.
\end{prop}

\begin{proof}
By definition, $\Lozi^{m+n-1}(q,0)$ is the intersection point of
$\kappa$ and $\eta_{m,n}^{B}$. This implies that 
\begin{equation}
\pi_{1}\circ\Lozi^{m+n-1}(q,0)=0\label{eq: q on the critical locus}
\end{equation}
and hence $\Lozi^{m+n}(q,0)$ is the turning point of $\eta_{m,n}^{U}$.
That is, 
\begin{equation}
(p,0)=\Lozi^{m+n}(q,0).\label{eq: p and q}
\end{equation}

Suppose that there exist a periodic point $\boldsymbol{v}\in C_{m,n}^{L}$
with period $m+n$ such that $\pi_{1}\circ\Lozi^{m+n-1}(\boldsymbol{v})=0$.
Then $\pi_{2}(\boldsymbol{v})=\pi_{1}\circ\Lozi^{m+n-1}(\boldsymbol{v})=0$
and $\boldsymbol{v}\in\kappa_{m,n}^{C}$. This implies that $\boldsymbol{v}=(q,0)$.
By (\ref{eq: p and q}), we get $p=q$.

Conversely, suppose that $p=q$. By (\ref{eq: p and q}), $\boldsymbol{v}\equiv(q,0)\in C_{m,n}^{L}$
is a periodic point with period $m+n$. And by (\ref{eq: q on the critical locus}),
we get $\pi_{1}\circ\Lozi^{m+n-1}(\boldsymbol{v})=0$.
\end{proof}
\begin{cor}
\label{cor:Geometrical characterization of the BCB in C(m,n)}Suppose
that $(a,b)\in\ParaModel$, $m\geq3$, $n\geq2$, and $C_{m,n}^{L}$
exists at $(a,b)$. Then $(a,b)$ is a $\iota_{\pm,m,n}$-bifurcation
parameter if and only if $p(a,b)=q(a,b)$.
\end{cor}

\begin{proof}
Let $\Sign\in S$. If $(a,b)$ is a $\iota_{\pm,m,n}$-bifurcation
parameter, then $\FormalPeriodic_{\Sign,m,n}$ is admissible at $(a,b)$.
By Corollary \ref{cor:Admissible periodic point in Cmn}, we have
$\FormalPeriodic_{\Sign,m,n}\in C_{m,n}^{L}$. Also, by Proposition
\ref{prop:Itinerary of a point in C_mn}, we have $\pi_{1}\circ\Lozi^{m+n-1}(\FormalPeriodic_{\Sign,m,n})=0$.
Consequently, $p=q$ by Proposition \ref{prop:Geometrical characterization of the BCB in C(m,n)}.

Conversely, if $p=q$ at $(a,b)$, there exist a periodic point $\boldsymbol{v}\in C_{m,n}^{L}$
with period $m+n$ such that $\pi_{1}\circ\Lozi^{m+n-1}(\boldsymbol{v})=0$
by Proposition \ref{prop:Geometrical characterization of the BCB in C(m,n)}.
Therefore, $(a,b)$ is a $\iota_{\pm,m,n}$-bifurcation parameter
by Corollary \ref{cor:Admissible periodic point in Cmn}.
\end{proof}

\section{\label{sec:Geometry}The geometry of the Lozi family}

In this section, we consider the forward and backward iterates of
lines by the branches $\Lozi_{-}$ and $\Lozi_{+}$. We show that
the quantities $p$ and $q$ used in Proposition \ref{prop:Geometrical characterization of the BCB in C(m,n)}
can be estimated by using $W^{S}(\boldsymbol{z}_{-})$ and $W^{U}(\boldsymbol{z}_{-})$. 

\subsection{\label{subsec:Geometry-forward iterates}The forward iterates of
a line}

We consider lines $L$ parameterized by its slope and the intersection
point of $L$ and $W^{S}(\Lozi_{\Sign},\boldsymbol{z}_{\Sign})$,
where $\Sign\in S$. We iterate $L$ by the branches $\Lozi_{-}$
and $\Lozi_{+}$. We derive the corresponding transformations for
the slope and the intersection point. By using the transformations,
we prove a version of the inclination lemma (Propositions \ref{prop:Convergence of the horizontal slope}
and \ref{prop:Convergence of the intersection point of mu}) for the
Lozi maps. The inclination lemma shows that we can use $W^{U}(\boldsymbol{z}_{-})$
to approximate the value of $p$. 

\subsubsection{The transformation for the slope}

Let $f_{\Sign}(s)=-\frac{1}{bs+\Sign a}$ where $\Sign\in S$. We
note that $-\Sign\frac{1}{\lambda}$ is the stable fixed point of
$f_{\Sign}$.
\begin{lem}
\label{lem:Iteration of lines}Let $L$ be a line and $s$ be the
slope of $L$. If $s\neq-\frac{\Sign a}{b}$, then $\Lozi_{\Sign}(L)$
is a line with the slope $f_{\Sign}(s)$.
\end{lem}

\begin{lem}
\label{lem:Invariance of the unstable cone}Suppose that $(a,b)\in\ParaFull$.
The following are true.
\begin{enumerate}
\item The interval $[-\frac{1}{\lambda},\frac{1}{\lambda}]$ is $f_{\Sign}$-invariant.
\item $0<f_{\Sign}^{\prime}(s)\leq\frac{b}{\lambda^{2}}$ for $s\in[-\frac{1}{\lambda},\frac{1}{\lambda}]$.
\item $f_{-}(s)\geq s$ and $f_{+}(s)\leq s$ for $s\in[-\frac{1}{\lambda},\frac{1}{\lambda}]$.
\end{enumerate}
\end{lem}

\begin{lem}
\label{lem:upper bound of b/l^2}We have 
\[
\frac{b}{\lambda^{2}}<\frac{1}{8}
\]
for all $(a,b)\in\ParaModel$.
\end{lem}

\begin{proof}
The lemma is clear when $b=0$. Suppose that $b>0$. By Lemma \ref{lem:Bounds of lambda},
we have 
\[
\frac{b}{\lambda^{2}}<\frac{b}{(2b+1)^{2}}.
\]
The right hand side has an upper bound $\frac{1}{8}$ on $[0,1]$.
\end{proof}
\begin{prop}
\label{prop:Exponential convergence of slope}Suppose that $(a,b)\in\ParaModel$.
There exists a constant $c\in(0,6.4)$ such that 
\[
\left(1-c\frac{b}{\lambda^{2}}\right)\left(\frac{b}{\lambda^{2}}\right)^{m}\left|s-s_{\infty}\right|\leq\left|s_{m}-s_{\infty}\right|\leq\left(\frac{b}{\lambda^{2}}\right)^{m}\left|s-s_{\infty}\right|
\]
for all $m\geq1$ and $s\in[-\frac{1}{\lambda},\frac{1}{\lambda}]$,
where $s_{m}=f_{\Sign}^{m}(s)$ and $s_{\infty}=-\Sign\frac{1}{\lambda}$.
\end{prop}

\begin{proof}
The upper bound follows from the mean value theorem and Lemma \ref{lem:Invariance of the unstable cone}.

By the mean value theorem, there exists $\xi\in(s_{m-1},s_{\infty})$
such that $s_{m}-s_{\infty}=f'_{\Sign}(\xi)(s_{m-1}-s_{\infty})$.
We get 
\begin{equation}
\left|s_{m}-s_{\infty}\right|\geq(\left|f'_{\Sign}(s_{\infty})\right|-\left|f'_{\Sign}(\xi)-f'_{\Sign}(s_{\infty})\right|)\left|s_{m-1}-s_{\infty}\right|.\label{eq:sm-sInfty (triangular inequality)}
\end{equation}
We have $f_{\Sign}^{\prime}(s)=b(f_{\Sign}(s))^{2}$, $\left|f_{\Sign}(\xi)+f_{\Sign}(s_{\infty})\right|\leq\frac{2}{\lambda}$,
and $\left|f_{\Sign}(\xi)-f_{\Sign}(s_{\infty})\right|\leq\frac{b}{\lambda^{2}}\left|s_{m-1}-s_{\infty}\right|$.
Hence, 
\[
\left|f'_{\Sign}(\xi)-f'_{\Sign}(s_{\infty})\right|=b\left|f_{\Sign}(\xi)+f_{\Sign}(s_{\infty})\right|\left|f_{\Sign}(\xi)-f_{\Sign}(s_{\infty})\right|\leq\frac{2b^{2}}{\lambda^{3}}\left|s_{m-1}-s_{\infty}\right|.
\]
Also, $f'_{\Sign}(s_{\infty})=\frac{b}{\lambda^{2}}$ and $\left|s_{m-1}-s_{\infty}\right|\leq\frac{2}{\lambda}(\frac{b}{\lambda^{2}})^{m-1}$.
The inequality (\ref{eq:sm-sInfty (triangular inequality)}) becomes
\begin{equation}
\left|s_{m}-s_{\infty}\right|\geq\frac{b}{\lambda^{2}}\left[1-4\left(\frac{b}{\lambda^{2}}\right)^{m}\right]\left|s_{m-1}-s_{\infty}\right|\label{eq:sm-sInfty (simplify)}
\end{equation}
We note that $4(\frac{b}{\lambda^{2}})^{m}<\frac{1}{2}$ by Lemma
\ref{lem:upper bound of b/l^2} and $\ln(1-x)\geq-(2\ln2)x$ when
$x\in[0,\frac{1}{2}]$. Thus,
\[
\ln\prod_{j=1}^{m}\left[1-4\left(\frac{b}{\lambda^{2}}\right)^{j}\right]\geq-(8\ln2)\sum_{j=1}^{m}\left(\frac{b}{\lambda^{2}}\right)^{j}\geq-\left(\frac{64}{7}\ln2\right)\frac{b}{\lambda^{2}}.
\]
Consequently, (\ref{eq:sm-sInfty (simplify)}) becomes
\begin{align*}
\left|s_{m}-s_{\infty}\right| & \geq\exp\left[-\left(\frac{64}{7}\ln2\right)\frac{b}{\lambda^{2}}\right]\left(\frac{b}{\lambda^{2}}\right)^{m}\left|s-s_{\infty}\right|\geq\left(1-c\frac{b}{\lambda^{2}}\right)\left(\frac{b}{\lambda^{2}}\right)^{m}\left|s-s_{\infty}\right|
\end{align*}
where $c=\frac{64}{7}\ln2$.
\end{proof}
A sequence of maps $\{f_{n}\}_{1\leq n<\infty}$ converges locally
uniformly on a topological space $X$ if for all $x\in X$ there exists
an open neighborhood $U$ of $x$ such that the sequence converges
uniformly on $U$. 
\begin{prop}
\label{prop:Convergence of the horizontal slope}Let $\ParaSpace\subset\ParaFull$
be relatively open, $L$ be a line, $s_{0}$ be the slope of $L$,
$\Sign\in S$, $s_{m}$ be the slope of $\Lozi_{\Sign}^{m}(L)$ for
$m>0$, and $s_{\infty}=-\Sign\frac{1}{\lambda}$ be the slope of
$W^{U}(\Lozi_{\Sign},\boldsymbol{z}_{\Sign})$. If $s_{0}$ depends
analytically on $(a,b)\in\ParaSpace$ and $|s_{0}|\leq\frac{1}{\lambda}$,
then
\[
\lim_{m\rightarrow\infty}\partial_{a}^{i}\partial_{b}^{j}s_{m}=\partial_{a}^{i}\partial_{b}^{j}s_{\infty}
\]
locally uniformly on $\ParaSpace$ for $(i,j)=(0,0),(0,1),(1,0)$.
\end{prop}

\begin{proof}
Let $(a_{0},b_{0})\in\ParaSpace$. If we view $f_{\Sign}$ as a map
on $(a,b,s)\in\mathbb{C}^{3}$, then it is a uniform contraction in
$s\in V$ on a complex convex neighborhood $U\times V\supset\{(a_{0},b_{0})\}\times[-\frac{1}{\lambda},\frac{1}{\lambda}]$
by Lemma \ref{lem:Invariance of the unstable cone} and continuity.
Thus, $f_{\Sign}(V)\subset V$ for all $(a,b)\in U$ since $V$ is
convex and $s_{\infty}\in V$ is an attracting fixed point of $f_{\Sign}$.

Furthermore, we may assume without lose of generality that $s_{0}$
has an analytic continuation on $U$. By continuity, we may also assume
that $s_{0}(U)\subset V$. Then $s_{m}(U)\subset V$ since $s_{m}=f_{\Sign}^{m}(s)$
for all $m\geq0$. Consequently, $\lim_{m\rightarrow\infty}s_{m}=s_{\infty}$
uniformly on $U$. The conclusion follows from the Weierstrass convergence
theorem \cite[P.73]{St10}. 
\end{proof}

\subsubsection{The transformation for the intersection point}

We now consider the orbit of a point on $W^{S}(\Lozi_{\Sign},\boldsymbol{z}_{\Sign})$
for $\Sign\in S$.
\begin{lem}
\label{lem:Iteration of points on the stable manifold}Suppose that
$(a,b)\in\ParaFull$. Let $\Sign\in S$, $\boldsymbol{p}_{1}\in W^{S}(\Lozi_{\Sign},\boldsymbol{z}_{\Sign})$,
and $\boldsymbol{p}_{m+1}=\Lozi_{\Sign}^{m}(\boldsymbol{p}_{1})$
for $m\geq1$. If $\boldsymbol{p}_{1}=(v_{1}+\zeta_{\Sign},v_{0}+\zeta_{\Sign})$,
then $\boldsymbol{p}_{m}=(v_{m}+\zeta_{\Sign},v_{m-1}+\zeta_{\Sign})$
for all $m\ge1$, where $v_{m}=-\Sign\mu v_{m-1}$.
\end{lem}

\begin{prop}
\label{prop:Convergence of the intersection point of mu}Let $\ParaSpace\subset\ParaFull$
be relatively open and $\{(v_{m}+\zeta_{\Sign},v_{m-1}+\zeta_{\Sign})\}_{m\geq1}$
be an orbit on $W^{S}(\Lozi_{\Sign},\boldsymbol{z}_{\Sign})$. If
$v_{0}$ depends analytically on $(a,b)\in\ParaSpace$, then
\[
\lim_{m\rightarrow\infty}\partial_{a}^{i}\partial_{b}^{j}v_{m}=0
\]
locally uniformly on $\ParaSpace$ for $(i,j)=(0,0),(0,1),(1,0)$.
\end{prop}

\begin{proof}
Let $(a_{0},b_{0})\in\ParaSpace$ and $U$ be a complex convex neighborhood
of $(a_{0},b_{0})$ such that $v_{0}$ has an analytic continuation
on $U$. We may also assume that $U$ is small enough such that $v_{0}$
is bounded and $|\mu|<c$ on $U$ for some $c\in(0,1)$. Then $v_{m}$
also has an analytic continuation and $\lim_{m\rightarrow\infty}v_{m}=0$
uniformly on $U$ by Lemma \ref{lem:Iteration of points on the stable manifold}.
Therefore, the conclusion follows from the Weierstrass' convergence
theorem \cite[P.73]{St10}. 
\end{proof}
\begin{lem}
\label{lem:map from the intersection point of Ws(Z-) to Ws(Z+)}Suppose
that $(a,b)\in\ParaFull$. Let $s$ be the slope of a line $L$ and
$(\Sign\mu v+\zeta_{-\Sign},v+\zeta_{-\Sign})$ be the intersection
point of $L$ and $W^{S}(\Lozi_{-\Sign},\boldsymbol{z}_{-\Sign})$.
If $s\neq-\frac{\Sign}{\mu}$, then $L$ and $W^{S}(\Lozi_{\Sign},\boldsymbol{z}_{\Sign})$
intersect at $(-\Sign\mu w+\zeta_{\Sign},w+\zeta_{\Sign})$, where
\[
w=\frac{1}{1+\Sign\mu s}\left[(1-\Sign\mu s)v-\Sign(1-s)(\zeta_{+}-\zeta_{-})\right].
\]
\end{lem}

\subsubsection{Estimation of $p_{m,n}$.}

We show that the turning point of $\eta_{m,n}^{U}$ can be estimated
by a turning point of the unstable manifold $W^{U}(\Lozi,\boldsymbol{z}_{-})$. 
\begin{lem}
\label{lem:Turning point}Suppose that $(a,b)\in\ParaFull$. Let $\Sign\in S$,
$s$ be the slope of a line $L$ and $(-\Sign\mu v+\zeta_{\Sign},v+\zeta_{\Sign})$
be the intersection point of $L$ and $W^{S}(\Lozi_{\Sign},\boldsymbol{z}_{\Sign})$.
Then the turning point of $\Lozi(L)$ is 
\[
((a-b-1)-(1+\Sign s\mu)bv-(1-s)b\zeta_{\Sign},0).
\]
\end{lem}

Let $(p_{\infty,n},0)$ be the turning point of $\Lozi\circ\Lozi_{-}^{n-2}\circ\Lozi_{+}^{2}(W^{U}(\Lozi_{-},\boldsymbol{z}_{-}))$.
\begin{prop}
\label{prop:Convergence of a horizontal line}Let $\ParaSpace\subset\ParaFull$
be relatively open, $L$ be a line, $s$ be the slope of $L$, $(-\mu v+\zeta_{+},v+\zeta_{+})$
be the intersection point of $L$ and $W^{S}(\Lozi_{+},\boldsymbol{z}_{+})$,
and $(p_{m,n},0)$ be the turning point of $\Lozi\circ\Lozi_{-}^{n-2}\circ\Lozi_{+}^{2}\circ\Lozi_{-}^{m-2}\circ\Lozi_{+}(L)$
for $m,n\geq2$. If $s$ and $v$ depends analytically on $(a,b)\in\ParaSpace$
and $|s|\leq\frac{1}{\lambda}$, then
\[
\lim_{m\rightarrow\infty}\partial_{a}^{i}\partial_{b}^{j}p_{m,n}=\partial_{a}^{i}\partial_{b}^{j}p_{\infty,n}
\]
locally uniformly on $\ParaSpace$ for $(i,j)=(0,0),(0,1),(1,0)$
and $n\geq2$.
\end{prop}

\begin{proof}
The proposition follows from the chain rule, Lemmas \ref{lem:Iteration of lines},
\ref{lem:Invariance of the unstable cone}, \ref{lem:Iteration of points on the stable manifold},
\ref{lem:map from the intersection point of Ws(Z-) to Ws(Z+)}, and
\ref{lem:Turning point}, and Propositions \ref{prop:Convergence of the horizontal slope}
and \ref{prop:Convergence of the intersection point of mu}.
\end{proof}
Recall that $\eta$ is the horizontal line $\{y=0\}$, $\eta_{m,n}^{B}=\Lozi^{m+n-1}(\eta\cap C_{m,n}^{L})$,
and $\eta_{m,n}^{U}=\Lozi(\eta_{m,n}^{B})$. We apply the proposition
to the turning point of $\eta_{m,n}^{U}$.
\begin{cor}
\label{cor:Convergence of p}Let $M,n\geq2$ and $\ParaSpace\subset\ParaModel$
be relatively open. Assume that $C_{m,n}^{L}$ exists for all $(a,b)\in\ParaSpace$
and $m\geq M$. Also, let $(p_{m,n},0)$ be the turning point of $\eta_{m,n}^{U}$.
Then
\[
\lim_{m\rightarrow\infty}\partial_{a}^{i}\partial_{b}^{j}p_{m,n}=\partial_{a}^{i}\partial_{b}^{j}p_{\infty,n}
\]
locally uniformly on $\ParaSpace$ for $(i,j)=(0,0),(0,1),(1,0)$.
\end{cor}

\begin{proof}
Let $\hat{\eta}_{m,n}^{B}=\Lozi_{-}^{n-2}\circ\Lozi_{+}^{2}\circ\Lozi_{-}^{m-2}\circ\Lozi_{+}(\eta)$
and $\hat{\eta}_{m,n}^{U}=\Lozi(\hat{\eta}_{m,n}^{B})$. Then $\eta_{m,n}^{B}=\hat{\eta}_{m,n}^{B}\cap B$.
Hence $\eta_{m,n}^{U}$ and $\hat{\eta}_{m,n}^{U}$ share the same
turning point $(p_{m,n},0)$. Therefore, the conclusion follows from
Proposition \ref{prop:Convergence of a horizontal line}.
\end{proof}

\subsection{The exponential convergence of the turning points}

We show that the turning points $\{u_{m}^{d}\}_{m=2}^{\infty}$, where
$d\in\{L,R\}$, converge exponentially.
\begin{prop}
\label{prop:Exponential convergence of u}There exist constants $c_{2}>c_{1}>0$
such that 
\[
c_{1}\left(1-\frac{1}{\lambda^{m-1}}\right)\lambda\left(\frac{b}{\lambda}\right)^{m-1}\leq u_{\infty}-u_{m}^{d}\leq c_{2}\left(1-\frac{1}{\lambda^{m-1}}\right)\lambda\left(\frac{b}{\lambda}\right)^{m-1}
\]
for all $(a,b)\in\ParaModel$, $m\geq2$, and $d\in\{L,R\}$.
\end{prop}

\begin{proof}
For $\Sign\in S$ and $m\geq2$, let $L_{2}^{\Sign}=\Lozi_{+}(\{y=\Sign\})$,
and $L_{m}^{\Sign}=\Lozi_{-}^{m-2}(L_{2}^{\Sign})$. We note that
$L_{m}^{-}$ and $L_{m}^{+}$ are parallel. Let $s_{m}$ be the slope
of $L_{m}^{\Sign}$, $(\frac{b}{\lambda}v_{m}^{\Sign}+\zeta_{-},v_{m}^{\Sign}+\zeta_{-})$
be the intersection point of $L_{m}^{\Sign}$ and $W^{S}(\Lozi_{-},\boldsymbol{z}_{-})$,
and $(0,k_{m}^{\Sign})$ be the intersection point of $L_{m}^{\Sign}$
and the critical locus $\kappa$. Moreover, let $L_{\infty}=W^{U}(\Lozi_{-},\boldsymbol{z}_{-})$,
$s_{\infty}=\frac{1}{\lambda}$, and $(0,k_{\infty})$ be the intersection
point of $L_{\infty}$ and the critical locus $\kappa$. Then $(u_{m}^{L},0)=\Lozi(0,k_{m}^{-})$,
$(u_{m}^{R},0)=\Lozi(0,k_{m}^{+})$, $(u_{\infty},0)=\Lozi(0,k_{\infty})$,
and 
\begin{equation}
k_{m}^{\Sign}-k_{\infty}=\left(1-\frac{b}{\lambda}s_{m}\right)v_{m}^{\Sign}-(s_{\infty}-s_{m})\label{eq:substraction of k}
\end{equation}
 for all $m\geq2$. Also, $s_{2}=-\frac{1}{a}$, $s_{m+1}=f_{-}(s_{m})$,
$v_{2}^{\Sign}=2-\frac{(1+\Sign)\lambda+2}{\lambda^{2}+2b}b$, and
$v_{m}^{\Sign}=(\frac{b}{\lambda})^{m-2}v_{2}^{\Sign}$ for all $m\geq2$.

To estimate the upper bound, we have $s_{m}\geq-\frac{1}{\lambda}$.
By Proposition \ref{prop:Exponential convergence of slope}, we have
\begin{align*}
s_{\infty}-s_{m} & \geq\left(1-c\frac{b}{\lambda^{2}}\right)\left(\frac{b}{\lambda^{2}}\right)^{m-2}(s_{\infty}-s_{2})=\frac{2}{\lambda}\left(1-\frac{cb}{\lambda^{2}}-\frac{b}{2\lambda a}+\frac{cb^{2}}{2\lambda^{3}a}\right)\left(\frac{b}{\lambda^{2}}\right)^{m-2}\\
 & \geq\frac{2}{\lambda}\left[1-\left(c+\frac{1}{2}\right)\frac{b}{\lambda^{2}}\right]\left(\frac{b}{\lambda^{2}}\right)^{m-2}
\end{align*}
for some constant $c>0$. Also, $v_{2}^{\Sign}\leq2$ and $s_{m}\geq-\frac{1}{\lambda}$.
Thus, (\ref{eq:substraction of k}) becomes 
\[
k_{m}^{\Sign}-k_{\infty}\leq2\left[1-\frac{1}{\lambda^{m-1}}+\left(c+\frac{3}{2}\right)\frac{b}{\lambda^{2}}\right]\left(\frac{b}{\lambda}\right)^{m-2}.
\]
We note that 
\[
\left(1-\frac{1}{\lambda^{m-1}}\right)^{-1}\frac{b}{\lambda^{2}}=\left(1+\frac{1}{\lambda^{m-1}-1}\right)\frac{b}{\lambda^{2}}\leq\frac{b}{\lambda^{2}}+\frac{b}{\lambda-1}<\frac{5}{8}
\]
by Lemmas \ref{lem:Bounds of lambda} and \ref{lem:upper bound of b/l^2}.
This proves that 
\[
k_{m}^{\Sign}-k_{\infty}\leq c_{2}\left(1-\frac{1}{\lambda^{m-1}}\right)\left(\frac{b}{\lambda}\right)^{m-2}
\]
for some $c_{2}>0$.

We estimate the lower bound. When $m=2$, we have $s_{2}=-\frac{1}{a}$,
$s_{\infty}-s_{2}\leq\frac{2}{\lambda}$, and $v_{2}^{\Sign}\geq2(1-\frac{\lambda+1}{\lambda a}b)$.
Also, by Lemma \ref{lem:Existence of beta_infty, gamma_infty}, we
have $v_{2}^{\Sign}>1$. Hence, (\ref{eq:substraction of k}) becomes
\begin{align*}
k_{2}^{\Sign}-k_{\infty} & =v_{2}^{\Sign}+\frac{b}{\lambda a}v_{2}^{\Sign}-(s_{\infty}-s_{2})\\
 & \geq2\left(1-\frac{\lambda+1}{\lambda a}b\right)+\frac{b}{\lambda a}-\frac{2}{\lambda}=2\left[1-\frac{1}{\lambda}-\left(1+\frac{1}{2\lambda}\right)\frac{b}{a}\right].
\end{align*}
By Lemma \ref{lem:Bounds of lambda}, we obtain 
\[
\left(1-\frac{1}{\lambda}\right)^{-1}\left(1+\frac{1}{2\lambda}\right)\frac{b}{a}=\left(b+\frac{b}{\lambda-1}\right)\left(1+\frac{1}{2\lambda}\right)\frac{1}{a}<\left(b+\frac{1}{2}\right)\frac{3}{2}\frac{1}{2b+1}=\frac{3}{4}.
\]
This yields
\[
k_{2}^{\Sign}-k_{\infty}\geq\frac{1}{4}\left(1-\frac{1}{\lambda}\right).
\]

When $m\geq3$, we have $s_{m}\leq\frac{1}{\lambda}$ and $v_{2}^{\Sign}\geq2(1-\frac{\lambda+1}{\lambda^{2}}b)$.
By Proposition \ref{prop:Exponential convergence of slope}, we have
\[
s_{\infty}-s_{m}\leq\left(\frac{b}{\lambda^{2}}\right)^{m-2}(s_{\infty}-s_{2})\leq\frac{2}{\lambda}\left(\frac{b}{\lambda^{2}}\right)^{m-2}.
\]
 Hence, (\ref{eq:substraction of k}) becomes
\[
k_{m}^{\Sign}-k_{\infty}\geq2\left(1-\frac{1}{\lambda^{m-1}}-\frac{\lambda+2}{\lambda^{2}}b\right)\left(\frac{b}{\lambda}\right)^{m-2}.
\]
Note that 
\[
\left(1-\frac{1}{\lambda^{m-1}}\right)^{-1}\frac{\lambda+2}{\lambda^{2}}b\leq\left(1-\frac{1}{\lambda^{2}}\right)^{-1}\frac{\lambda+2}{\lambda^{2}}b=\left(1+\frac{1}{\lambda+1}\right)\frac{b}{\lambda-1}<\frac{3}{4}.
\]
by Lemma \ref{lem:Bounds of lambda}. Thus,
\[
k_{m}^{\Sign}-k_{\infty}\geq\frac{1}{4}\left(1-\frac{1}{\lambda^{m-1}}\right)\left(\frac{b}{\lambda}\right)^{m-2}.
\]

Finally, we have $u_{\infty}-u_{m}^{L}=b(k_{m}^{-}-k_{\infty})$ and
$u_{\infty}-u_{m}^{R}=b(k_{m}^{+}-k_{\infty})$. This completes the
proof.
\end{proof}

\subsection{\label{subsec:Geometry-backward iterates}The backward iterates of
a line}

We consider lines $L$ parameterized by its vertical slope and the
intersection point of $L$ and $W^{U}(\Lozi_{\Sign},\boldsymbol{z}_{\Sign})$,
where $\Sign\in S$. We take the preimages of $L$ by the branches
$\Lozi_{-}$ and $\Lozi_{+}$. We derive the corresponding transformations
for the slope and the intersection point. By using the transformations,
we prove a version of the inclination lemma (Propositions \ref{prop:Convergence of the vertical slope}
and \ref{prop:Convergence of the reverse orbit on the unstable manifold})
for the Lozi maps. We first show that $\Lozi^{-(m+n-2)}(\kappa\cap B_{m,n})\rightarrow W^{S}(\Lozi_{-},\boldsymbol{z}_{-})$
uniformly as $m\rightarrow\infty$. Then we can use $W^{S}(\boldsymbol{z}_{-})$
to estimate the value of $q$.

\subsubsection{The transformation for the slope}

Let $g_{\Sign}(s)=-\frac{b}{s+\Sign a}$ where $\Sign\in S$. We note
that $-\Sign\mu$ is the stable fixed point of $g_{\Sign}$.
\begin{lem}
\label{lem:Preimage of lines}Let $L$ be a line and $s$ be the vertical
slope of $L$. If $s\neq-\Sign a$, then the vertical slope of $\Lozi_{\Sign}^{-1}(L)$
is $g_{\Sign}(s)$.
\end{lem}

\begin{lem}
\label{lem:Invariance of the stable cone}Suppose that $(a,b)\in\ParaFull$.
The following are true.
\begin{enumerate}
\item The interval $[-\mu,\mu]$ is $g_{\Sign}$-invariant.
\item $0<g_{\Sign}^{\prime}(s)\leq\frac{b}{\lambda^{2}}$ for $s\in[-\mu,\mu]$.
\item $g_{-}(s)\geq s$ and $g_{+}(s)\leq s$ for $s\in[-\mu,\mu]$.
\end{enumerate}
\end{lem}

\begin{prop}
\label{prop:Convergence of the vertical slope}Let $\ParaSpace\subset\ParaFull$
be relatively open, $L$ be a line, $s_{0}$ be the vertical slope
of $L$, $\Sign\in S$, $s_{m}$ be the vertical slope of $\Lozi_{\Sign}^{-m}(L)$
for $m>0$, and $s_{\infty}=-\Sign\mu$ be the vertical slope of $W^{S}(\Lozi_{\Sign},\boldsymbol{z}_{\Sign})$.
If $s_{0}$ depends analytically on $(a,b)\in\ParaSpace$ and $|s_{0}|\leq\mu$,
then
\[
\lim_{m\rightarrow\infty}\partial_{a}^{i}\partial_{b}^{j}s_{m}=\partial_{a}^{i}\partial_{b}^{j}s_{\infty}
\]
locally uniformly on $\ParaSpace$ for $(i,j)=(0,0),(0,1),(1,0)$.
\end{prop}

\begin{proof}
Let $(a_{0},b_{0})\in\ParaSpace$. If we view $g_{\Sign}$ as a map
on $(a,b,s)\in\mathbb{C}^{3}$, then it is a uniform contraction in
$s\in V$ on a complex convex neighborhood $U\times V\supset\{(a_{0},b_{0})\}\times[-\mu,\mu]$
by Lemma \ref{lem:Invariance of the stable cone} and continuity.
Then $g_{\Sign}(V)\subset V$ since $V$ is convex and $s_{\infty}\in V$
is an attracting fixed point of $g_{\Sign}$.

Furthermore, we may assume without lose of generality that $s_{0}$
has an analytic continuation on $U$. By continuity, we may also assume
that $s_{0}(U)\subset V$. Then $s_{m}(U)\subset V$ since $s_{m}=g_{\Sign}^{m}(s_{0})$
for all $m\geq0$. Consequently, $\lim_{m\rightarrow\infty}s_{m}=s_{\infty}$
uniformly on $U$. The conclusion follows from the Weierstrass' convergence
theorem \cite[P.73]{St10}. 
\end{proof}

\subsubsection{The transformation for the intersection point}

We now consider the backward orbit of a point on $W^{U}(\Lozi_{\Sign},\boldsymbol{z}_{\Sign})$
for $\Sign\in S$.
\begin{lem}
\label{lem:Reverse orbit on the unstable manifold}Suppose that $(a,b)\in\ParaFull$.
Let $\Sign\in S$, $\boldsymbol{p}_{1}\in W^{U}(\Lozi_{\Sign},\boldsymbol{z}_{\Sign})$,
and $\boldsymbol{p}_{m+1}=\Lozi_{\Sign}^{-m}(\boldsymbol{p}_{1})$
for $m\geq1$. If $\boldsymbol{p}_{1}=(v_{1}+\zeta_{\Sign},v_{2}+\zeta_{\Sign})$,
then $\boldsymbol{p}_{m}=(v_{m}+\zeta_{\Sign},v_{m+1}+\zeta_{\Sign})$
for all $m\geq1$, where $v_{m+1}=-\Sign\lambda^{-1}v_{m}$.
\end{lem}

\begin{prop}
\label{prop:Convergence of the reverse orbit on the unstable manifold}Let
$\ParaSpace\subset\ParaFull$ be relatively open and $\{(v_{m}+\zeta_{\Sign},v_{m+1}+\zeta_{\Sign})\}_{1\leq m<\infty}$
be a backward orbit on $W^{U}(\Lozi_{\Sign},\boldsymbol{z}_{\Sign})$.
If $v_{1}$ depends analytically on $(a,b)\in\ParaSpace$, then
\[
\lim_{m\rightarrow\infty}\partial_{a}^{i}\partial_{b}^{j}v_{m}=0
\]
locally uniformly on $\ParaSpace$ for $(i,j)=(0,0),(0,1),(1,0)$.
\end{prop}

\begin{proof}
Let $(a_{0},b_{0})\in\ParaSpace$ and $U$ be a complex convex neighborhood
of $(a_{0},b_{0})$ such that $v_{1}$ has an analytic continuation
on $U$. We may also assume that $U$ is small enough such that $v_{1}$
is bounded and $|\lambda^{-1}|\leq c$ on $U$ for some $c\in(0,1)$.
Then $v_{m}$ also has an analytic continuation and $\lim_{m\rightarrow\infty}v_{m}=0$
uniformly on $U$ by Lemma \ref{lem:Reverse orbit on the unstable manifold}.
Therefore, the proposition follows from the Weierstrass' convergence
theorem \cite[P.73]{St10}. 
\end{proof}
\begin{lem}
\label{lem:Flipping intersection of unstable manifolds}Suppose that
$(a,b)\in\ParaFull$. Let $s$ be the vertical slope of a line $L$
and $(v+\zeta_{-\Sign},\Sign\lambda^{-1}v+\zeta_{-\Sign})$ be the
intersection point of $L$ and $W^{U}(\Lozi_{-\Sign},\boldsymbol{z}_{-\Sign})$.
If $s\neq-\Sign\lambda$, then $L$ and $W^{U}(\Lozi_{\Sign},\boldsymbol{z}_{\Sign})$
intersect at $(w+\zeta_{\Sign},-\Sign\lambda^{-1}w+\zeta_{\Sign})$,
where 
\[
w=\frac{1}{1+\Sign\lambda^{-1}s}\left[(1-\Sign\lambda^{-1}s)v-\Sign(1-s)(\zeta_{+}-\zeta_{-})\right].
\]
\end{lem}

\begin{lem}
\label{lem:The intersection of a stable manifold and a horizontal line}Suppose
that $(a,b)\in\ParaFull$. Let $L$ be a line, $s$ be the vertical
slope of $L$, $(v+\zeta_{-},\lambda^{-1}v+\zeta_{-})$ be the intersection
of $L$ and $W^{U}(\Lozi_{-},\boldsymbol{z}_{-})$, and $(t,0)$ be
the intersection of $\Lozi_{+}^{-1}(L)$ and the $x$-axis. Then 
\[
t=1-\frac{b+2s+(1-\lambda^{-1}s)v}{a+s}=\frac{1}{a+s}\left[\left(\lambda-s\right)\left(1-\frac{v}{\lambda}\right)-\frac{b}{\lambda}\left(\lambda-1\right)\right].
\]
\end{lem}

\subsubsection{Estimation of $q_{m,n}$.}

Recall that $(r_{m},0)$ is the intersection point of $\gamma_{m}$
and the $x$-axis for $1\leq m\leq\infty.$
\begin{prop}
\label{prop:Convergence of q}Let $\ParaSpace\subset\ParaFull$ be
relatively open, $L$ be a line, $s$ be the vertical slope of $L$,
$(v+\zeta_{-},\lambda^{-1}v+\zeta_{-})$ be the intersection point
of $L$ and $W^{U}(\Lozi_{-},\boldsymbol{z}_{-})$, and $(v_{m},0)$
be the intersection point of $\Lozi_{+}^{-1}\circ\Lozi_{-}^{-(m-1)}(L)$
and the $x$-axis for $m\geq1$. If $s$ and $v$ depends analytically
on $(a,b)\in\ParaSpace$ and $|s|\leq\mu$, then
\[
\lim_{m\rightarrow\infty}\partial_{a}^{i}\partial_{b}^{j}v_{m}=\partial_{a}^{i}\partial_{b}^{j}r_{\infty}
\]
locally uniformly on $\ParaSpace$ for $(i,j)=(0,0),(0,1),(1,0)$.
\end{prop}

\begin{proof}
The proposition follows from the chain rule, Lemmas \ref{lem:Preimage of lines},
\ref{lem:Invariance of the stable cone}, \ref{lem:Reverse orbit on the unstable manifold},
and \ref{lem:The intersection of a stable manifold and a horizontal line},
and Propositions \ref{prop:Convergence of the vertical slope} and
\ref{prop:Convergence of the reverse orbit on the unstable manifold}.
\end{proof}
\begin{cor}
\label{cor:Convergence of r_m}We have
\[
\lim_{m\rightarrow\infty}\partial_{a}^{i}\partial_{b}^{j}r_{m}=\partial_{a}^{i}\partial_{b}^{j}r_{\infty}
\]
locally uniformly on $\ParaModel$ for $(i,j)=(0,0),(0,1),(1,0)$.
\end{cor}

\begin{proof}
We have $\gamma_{m}=\Lozi_{+}^{-1}\circ\Lozi_{-}^{-(m-1)}(W^{S}(\Lozi_{+},\boldsymbol{z}_{+}))\cap C$.
The intersection point of $W^{S}(\Lozi_{+},\boldsymbol{z}_{+})$ and
$W^{U}(\Lozi_{-},\boldsymbol{z}_{-})$ is $(v+\zeta_{-},\lambda^{-1}v+\zeta_{-})$
where $v=\frac{2\lambda}{\lambda+1}$. And the slope of $W^{S}(\Lozi_{+},\boldsymbol{z}_{+})$
is $-\mu$. Thus, the corollary follows from Proposition \ref{prop:Convergence of q}.
\end{proof}
Recall that $\kappa=\{0\}\times I^{v}$ is the critical locus and
$\kappa_{m,n}^{C}=(\Lozi^{m+n-1}|_{C_{m,n}^{L}})^{-1}\kappa$.
\begin{cor}
\label{cor:Convergence of q}Let $M,n\geq2$ and $\ParaSpace\subset\ParaModel$
be relatively open. Assume that $C_{m,n}^{L}$ exists for all $(a,b)\in\ParaSpace$
and $m\geq M$. Also, let $(q_{m,n},0)$ be the intersection point
of $\kappa_{m,n}^{C}$ and the $x$-axis. Then
\[
\lim_{m\rightarrow\infty}\partial_{a}^{i}\partial_{b}^{j}q_{m,n}=\partial_{a}^{i}\partial_{b}^{j}r_{\infty}
\]
locally uniformly on $\ParaSpace$ for $(i,j)=(0,0),(0,1),(1,0)$.
\end{cor}

\begin{proof}
Let $\hat{\kappa}=\{0\}\times\mathbb{R}$, $\hat{\kappa}_{n}^{B}=\Lozi_{+}^{-2}\circ\Lozi_{-}^{-(n-2)}(\hat{\kappa})$,
$s_{n}$ be the slope of $\hat{\kappa}_{n}^{B}$, $(v_{n}+\zeta_{-},\lambda^{-1}v_{n}+\zeta_{-})$
be the intersection point of $\hat{\kappa}_{n}^{B}$ and $W^{U}(\Lozi_{-},\boldsymbol{z}_{-})$,
and $\hat{\kappa}_{m,n}^{C}=\Lozi_{+}^{-1}\circ\Lozi_{-}^{-(m-2)}(\hat{\kappa}_{n}^{B})$.
We note that $q_{m,n}$ is also the intersection point of $\hat{\kappa}_{m,n}^{C}$
and the $x$-axis because $C_{m,n}^{L}$ exists and $\hat{\kappa}_{m,n}^{C}=\kappa_{m,n}^{C}\cap C$.
By Lemmas \ref{lem:Preimage of lines}, \ref{lem:Invariance of the stable cone},
\ref{lem:Reverse orbit on the unstable manifold}, and \ref{lem:Flipping intersection of unstable manifolds},
$s_{n}$ and $v_{n}$ are analytic functions on $\ParaModel$, and
$|s_{n}|\leq\mu$. Therefore, the corollary follows from Proposition
\ref{prop:Convergence of q}.
\end{proof}

\subsection{The exponential convergence of the stable manifolds}

We proved in Corollary \ref{cor:Convergence of r_m} that $\lim_{m\rightarrow\infty}r_{m}=r_{\infty}$
locally uniformly. In fact, here we show that the convergence is exponential.
\begin{prop}
\label{prop:Exponential convergence of r}There exist constants $c_{2}>c_{1}>0$
 such that 
\[
c_{1}\left(\frac{1}{\lambda}\right)^{m}<\left|r_{m}-r_{\infty}\right|<c_{2}\left(\frac{1}{\lambda}\right)^{m}
\]
for all $(a,b)\in\ParaModel$ and $m\geq2$\@.
\end{prop}

\begin{proof}
Let $s_{m}$ be the slope of $\beta_{m}$ and $(v_{m}+\zeta_{-},\lambda^{-1}v_{m}+\zeta_{-})$
be the intersection point of $\beta_{m}$ and $W^{U}(\Lozi_{-},\boldsymbol{z}_{-})$.
Then 
\begin{equation}
r_{\infty}-r_{m}=\frac{\lambda-s_{m}}{a+s_{m}}\frac{v_{m}}{\lambda}+\left(\frac{b+2s_{m}}{a+s_{m}}-\frac{b+2s_{\infty}}{a+s_{\infty}}\right)\label{eq:r_infty - r_m}
\end{equation}
by Lemma \ref{lem:The intersection of a stable manifold and a horizontal line}.

To estimate the first term, by Lemma \ref{lem:bounds of the coefficient of the exponential convergence},
we have 
\[
\frac{7}{10}\leq\frac{\lambda-s_{m}}{a+s_{m}}\leq\frac{9}{8}.
\]
By Lemmas \ref{lem:Reverse orbit on the unstable manifold} and \ref{lem:Bounds of the intersection point of the stable and the unstable manifolds},
we have 
\[
\left(\frac{1}{\lambda}\right)^{m}<\frac{v_{m}}{\lambda}<2\left(\frac{1}{\lambda}\right)^{m}.
\]

To estimate the second term, by the mean value theorem and Lemmas
\ref{lem:Invariance of the stable cone} and \ref{lem:bounds of the remaining term of the exponential convergence},
we have 
\[
0\geq\frac{b+2s_{m}}{a+s_{m}}-\frac{b+2s_{\infty}}{a+s_{\infty}}\geq\frac{9}{4\lambda}(s_{m}-s_{\infty})\geq\frac{9}{4\lambda}\left(\frac{b}{\lambda^{2}}\right)^{m-1}(s_{1}-s_{\infty})=-\frac{9}{2}\left(\frac{b}{\lambda^{2}}\right)^{m}.
\]
We note that $s_{m}\nearrow s_{\infty}.$ 

Thus, (\ref{eq:r_infty - r_m}) becomes 
\[
\left[\frac{7}{10}-\frac{9}{2}\left(\frac{b}{\lambda}\right)^{m}\right]\left(\frac{1}{\lambda}\right)^{m}<r_{\infty}-r_{m}<\frac{9}{4}\left(\frac{1}{\lambda}\right)^{m}.
\]
By Lemma \ref{lem:Bounds of b/lambda}, we get 
\[
\frac{7}{10}-\frac{9}{2}\left(\frac{b}{\lambda}\right)^{m}>\frac{7}{10}-\frac{9}{2}\left(\frac{1}{3}\right)^{m}\geq0.2
\]
for all $m\geq2$. This proves the proposition.
\end{proof}
\begin{lem}
\label{lem:bounds of the coefficient of the exponential convergence}Let
$h(a,b,s)=\frac{\lambda-s}{a+s}$. Then 

\[
\frac{7}{10}\leq h(a,b,s)\leq\frac{9}{8}
\]
for all $(a,b)\in\ParaModel$ and $s\in[-\mu,\mu]$.
\end{lem}

\begin{proof}
The map $h$ is decreasing in $s$ on $[-\mu,\mu]$. We have 
\begin{equation}
h(a,b,\mu)\leq h(a,b,s)\leq h(a,b,-\mu).\label{eq:bounds of the coefficient of exponential convergence}
\end{equation}

To find the lower bound of (\ref{eq:bounds of the coefficient of exponential convergence}),
we compute 
\[
h(a,b,\mu)=1-\frac{3b}{\lambda^{2}+2b}\geq1-\frac{3b}{(2b+1)^{2}+2b}.
\]
The last inequality follows from Lemma \ref{lem:Bounds of lambda}.
The term $\frac{3b}{(2b+1)^{2}+b}$ has an upper bound $\frac{3}{10}$
on $b\in[0,1]$. Thus, $h(a,b,s)\geq\frac{7}{10}$.

To find the upper bound of (\ref{eq:bounds of the coefficient of exponential convergence}),
we get 
\[
h(a,b,-\mu)=1+\frac{b}{\lambda^{2}}\leq\frac{9}{8}
\]
by Lemma \ref{lem:upper bound of b/l^2}. Thus, $h(a,b,s)\leq\frac{9}{8}$.
\end{proof}
\begin{lem}
\label{lem:Bounds of the intersection point of the stable and the unstable manifolds}We
have
\[
1<v_{1}<2
\]
for all $(a,b)\in\ParaModel$.
\end{lem}

\begin{proof}
The value $v_{1}$ has a closed form
\begin{equation}
v_{1}=\frac{2\lambda}{\lambda+1}.\label{eq:formula of the intersection point of the stable and the unstable manifolds}
\end{equation}
The bounds follow immediately from $v_{1}-1=\frac{\lambda-1}{\lambda+1}$
and Lemma \ref{lem:Expanding multiplier}.
\end{proof}
\begin{lem}
\label{lem:bounds of the remaining term of the exponential convergence}Let
$h(a,b,s)=\frac{b+2s}{a+s}$. Then 

\[
\frac{\partial h}{\partial s}(a,b,s)\leq\frac{9}{4}\lambda^{-1}
\]
for all $(a,b)\in\ParaModel$ and $s\in[-\mu,\mu]$.
\end{lem}

\begin{proof}
Compute
\[
\frac{\partial h}{\partial s}(a,b,s)=\frac{2a-b}{(a+s)^{2}}\leq\frac{2a}{(a-\frac{b}{\lambda})^{2}}=\frac{2}{\lambda}\left(1+\frac{b}{\lambda^{2}}\right)<\frac{9}{4}\lambda^{-1}.
\]
The last part of the inequality is obtained by Lemma \ref{lem:upper bound of b/l^2}.
\end{proof}
\begin{lem}
\label{lem:Bounds of b/lambda}We have 
\[
\frac{b}{\lambda}<\frac{1}{3}
\]
for all $(a,b)\in\ParaModel$.
\end{lem}

\begin{proof}
The lemma is clear when $b=0$. Suppose that $b>0$. By Lemma \ref{lem:Bounds of lambda},
we have 
\[
\frac{b}{\lambda}<\frac{b}{2b+1}.
\]
The right hand side has an upper bound $\frac{1}{3}$ on $b\in[0,1]$. 
\end{proof}

\subsection{The parameters exhibiting a full horseshoe}

We find parameters when a Lozi map forms a full horseshoe. By Proposition
\ref{prop:Full horseshoe}, these are admissible parameters. Thus,
the bifurcation parameters are contained in a compact subset of parameters
$(a,b)\in[1,4]\times[0,1]$.
\begin{prop}
\label{prop:Parameter space of full horseshoe}Suppose that $(a,b)\in\ParaModel$.
If $a\geq2b+2$, then $r_{\infty}\leq u^{L}$\@.

In addition, if $(a,b)=(2,0)$, then $r_{\infty}=u^{L}$.
\end{prop}

\begin{proof}
We have 
\[
r_{\infty}=1-\frac{\lambda+2}{a\lambda+b}b\leq1
\]
and 
\[
u^{L}=a-2b-1\geq1.\qedhere
\]
\end{proof}

\subsection{The parameters of period-doubling renormalizable maps}

In this section, we find parameters such that a Lozi map is period-doubling
renormalizable. For these parameters, only the phase space $C_{2}$
has interesting dynamical aspects. In fact, these are non-admissible
parameters by Corollary \ref{cor:Disjoint C and U implies nonadmissible parameters}.
\begin{prop}
\label{prop:Period-doubling renormalizable parameters}Suppose that
$(a,b)\in\ParaModel$. If $a<\sqrt{2}(1-3b)$, then $u_{\infty}<r_{2}$\@.
\end{prop}

\begin{proof}
Let $s_{2}$ be the vertical slope of $\beta_{2}$. By Lemmas \ref{lem:Reverse orbit on the unstable manifold}
and \ref{lem:The intersection of a stable manifold and a horizontal line}
and (\ref{eq:formula of the intersection point of the stable and the unstable manifolds}),
we have 
\[
r_{2}=\left[\left(1-\frac{s_{2}}{\lambda}\right)\left(1+\frac{1}{\lambda+1}\right)-\frac{b}{\lambda}\right]\frac{\lambda-1}{a+s_{2}}.
\]
By Lemma \ref{lem:Invariance of the stable cone}, we have $s_{2}\leq\frac{b}{\lambda}$.
Hence, 
\begin{align*}
\frac{\lambda}{\lambda-1}r_{2} & \geq\left[\left(1-\frac{b}{\lambda^{2}}\right)\left(1+\frac{1}{\lambda+1}\right)-\frac{b}{\lambda}\right]\left(1-\frac{2b}{\lambda^{2}}\right)=\left[\lambda+2-b\left(1+\frac{2}{\lambda}+\frac{2}{\lambda^{2}}\right)\right]\left(1-\frac{2b}{\lambda^{2}}\right)\frac{1}{\lambda+1}\\
 & \geq\left[\lambda+2-b\left(1+\frac{4}{\lambda}+\frac{6}{\lambda^{2}}\right)\right]\frac{1}{\lambda+1}.
\end{align*}
Since $u_{\infty}=\lambda-1$ and $\lambda>1$, we get
\begin{equation}
\frac{\lambda(\lambda+1)}{\lambda-1}(r_{2}-u_{\infty})\geq2-\lambda^{2}-11b.\label{eq:r2-uInfty}
\end{equation}

By the assumption $a<\sqrt{2}(1-3b)$, we have $\lambda<\frac{2-7b}{\sqrt{2}}$
and $b<1-\frac{2\sqrt{2}}{3}$. Thus, (\ref{eq:r2-uInfty}) becomes
\[
\frac{\lambda(\lambda+1)}{\lambda-1}(r_{2}-u_{\infty})>b(3-\frac{49}{2}b)\geq0.\qedhere
\]
\end{proof}
\begin{cor}
\label{cor:Disjoint C and U when a is small}Suppose that $(a,b)\in\ParaModel$,
$m\geq3$, and $n\geq2$. If $a<\sqrt{2}(1-3b)$, then $C_{m}\cap U_{n}=\emptyset$\@.
\end{cor}

\begin{proof}
By Proposition \ref{prop:Period-doubling renormalizable parameters},
we have $u_{n}^{R}\leq u_{\infty}<r_{2}\leq r_{m-1}$. Thus, $C_{m}\cap U_{n}=\emptyset$.
\end{proof}

\section{The geometry near the tent family}

In this section, we study the geometry of degenerate maps ($b=0$).
By continuity, we can then extend some properties to a neighborhood
of $b=0$ in the parameter space.

\subsection{Intersections of the critical value and the $\gamma$ stable manifolds}

First, we study how the critical value $u$ intersect the stable manifolds
$\gamma_{m}$ in terms of the parameter $a$. When a Lozi map is degenerate,
all of the $u$ variables $u_{\infty},u^{L},u_{2}^{L},\cdots,u^{R},u_{2}^{R},\cdots$
have the same value $u=a-1$, and the $r$ variables have a closed
form $r_{\infty}=1$ and $r_{m}=1-\frac{2}{a^{m-1}(a+1)}$ for $1\leq m<\infty$.
\begin{prop}
\label{prop:Tent map geometry}Let $b=0$. There exists an increasing
sequence $\{a_{m}\}_{1\leq m\leq\infty}$ with $a_{1}=1$, $a_{2}=\sqrt{\text{2}}$,
and $a_{\infty}=2$, such that the following holds.
\begin{enumerate}
\item $u=r_{m}$ when $a=a_{m}$ for $2\leq m\leq\infty$.
\item $r_{m}<u<r_{m+1}$ when $a\in(a_{m},a_{m+1})$ for $1\leq m<\infty$.
\end{enumerate}
\end{prop}

\begin{proof}
Let $h_{m}(a)=u(a)-r_{m}(a)$. We construct an increasing sequence
$\{a_{m}\}_{2\leq m<\infty}$ such that $a_{m}\in[\sqrt{2},2)$, $r_{m-1}<u<r_{m}$
when $a\in(a_{m-1},a_{m})$, and $u=r_{m}$ when $a=a_{m}$ by induction
on $m$.

Let $a_{2}=\sqrt{2}$. Clearly, $h_{2}(a_{2})=0$. By Lemma \ref{lem:Existence of beta1},
we have $h_{1}(a)>0$ for all $a>a_{1}$. By Proposition \ref{prop:Period-doubling renormalizable parameters},
we have $h_{2}(a)<0$ for all $a\in(a_{1},a_{2})$. This proves the
case for $m=2$.

Suppose that the induction hypothesis holds for $m\geq2$. By the
induction hypothesis, we have $h_{m+1}(a_{m})<h_{m}(a_{m})=0$. Also,
by Proposition \ref{prop:Parameter space of full horseshoe}, we have
$h_{m+1}(a_{\infty})>h_{\infty}(a_{\infty})=0$. Thus, $h_{m+1}$
has a root $a_{m+1}\in(a_{m},a_{\infty})$ by the intermediate value
theorem. Moreover, $r_{m}<u<r_{m+1}$ holds for all $a\in(a_{m},a_{m+1})$
because $h_{m}$ and $h_{m+1}$ are increasing on $[\sqrt{2},\infty)$
(Lemma \ref{lem:Increasing critical value (degenerate)}).

Therefore, the proposition is proved by induction.
\end{proof}
\begin{lem}
\label{lem:Increasing critical value (degenerate)}Let $b=0$. There
exists a constant $c>0$ such that $\frac{\dif}{\dif a}(u-r_{m})>c$
on $[\frac{7}{5},\infty)$ for all $2\leq m<\infty$.
\end{lem}

\begin{proof}
Compute
\[
\frac{\dif}{\dif a}(u-r_{m})=1-\frac{2}{a+1}\left(\frac{m-1}{a^{m}}\right)-\frac{2}{a^{m-1}(a+1)^{2}}\geq1-\frac{5}{6}\frac{m-1}{\left(\frac{7}{5}\right)^{m}}-\frac{2}{\frac{7}{5}\left(\frac{12}{5}\right)^{2}}.
\]
We note that the map $x\rightarrow\frac{x-1}{a^{x}}$ has a maximal
value $\frac{1}{ea\ln a}$. We get
\[
\frac{\dif}{\dif a}(u-r_{m})\geq\frac{379}{504}-\frac{25}{42e\ln\frac{7}{5}}>0.\qedhere
\]
\end{proof}

\subsection{The parameters of renormalization defined by two returns to $C$}

For $\overline{b}\in(0,1)$, let $\ParaNbd(\overline{b})=\ParaModel\cap\{0\leq b\leq\overline{b}\}$.
The set $\ParaNbd(\overline{b})$ is called a neighborhood of the
tent family.
\begin{prop}
\label{prop:Full cycle around a neighborhood of the tent family}For
all $N\geq2$, there exist $\underline{a}\in(\sqrt{2},2)$ and $\overline{b}\in(0,1)$
such that the following properties hold for all $(a,b)\in\ParaNbd(\overline{b})$,
$m\geq M$, and $2\leq n\leq N$, where $M=N+2$.
\begin{enumerate}
\item If $a\geq\underline{a}$, then $C_{m,n}^{L}$ and $C_{m,n}^{R}$ exist.
\item If $a\leq\underline{a}$, then $C_{m}\cap U_{n}=\emptyset$.
\end{enumerate}
\end{prop}

\begin{proof}
By Lemma \ref{lem:Increasing critical value (degenerate)}, continuity
of the partial derivative, and the compactness of the interval $[\frac{7}{5},4]$,
there exists a constant $\overline{b}_{1}>0$ such that $[\frac{7}{5},\infty)\times[0,\overline{b}_{1}]\subset\ParaModel$
and 
\begin{equation}
\frac{\partial}{\partial a}(u^{d}-r_{\nu})>0\label{eq:Monotone u-r}
\end{equation}
 for $(a,b)\in[\frac{7}{5},4]\times[0,\overline{b}_{1}]$, $d\in\{L,R\}$,
and $\nu\in\{N,N+1\}$. Moreover, by Proposition \ref{prop:Tent map geometry},
let $\underline{a}\in(\sqrt{2},2)$ be such that $r_{N}<u<r_{N+1}$
when $(a,b)=(\underline{a},0)$. By continuity, there exists $\overline{b}\in(0,\overline{b}_{1}]$
such that 
\begin{equation}
r_{N}<u^{L}\leq u^{R}<r_{N+1}\label{eq:Condition at the left boundary}
\end{equation}
 for all $(a,b)\in\{\underline{a}\}\times[0,\overline{b}]$. Since
$\sqrt{2}>\frac{7}{5}$, we may also assume that $\overline{b}>0$
is small enough such that $\sqrt{2}(1-3b)\geq\frac{7}{5}$ for all
$b\in[0,\overline{b}]$.

Let $(a,b)\in\ParaNbd(\overline{b})$. If $a<\frac{7}{5}$, then $a<\sqrt{2}(1-3b)$.
By Corollary \ref{cor:Disjoint C and U when a is small}, we have
$C_{m}\cap U_{n}=\emptyset$. If $\frac{7}{5}\leq a\leq\underline{a}$,
then $u_{n}^{R}\leq u^{R}<r_{N+1}\leq r_{m-1}$ by (\ref{eq:Monotone u-r})
and (\ref{eq:Condition at the left boundary}). Thus, $C_{m}\cap U_{n}=\emptyset$.
If $\underline{a}\leq a\leq4$, then $r_{n}\leq r_{N}<u^{L}\leq u_{m}^{L}$
by (\ref{eq:Monotone u-r}) and (\ref{eq:Condition at the left boundary}).
Thus, $C_{m,n}^{L}$ and $C_{m,n}^{R}$ exist by Proposition \ref{prop:Left and right components of a subpartition}.
If $a>4$, then $r_{n}<r_{\infty}\leq u^{L}\leq u_{m}^{L}$ by Proposition
\ref{prop:Parameter space of full horseshoe}. Thus, $C_{m,n}^{L}$
and $C_{m,n}^{R}$ exist by Proposition \ref{prop:Left and right components of a subpartition}.
\end{proof}

\subsection{The parameter curves of the boarder collision bifurcation}

Ishii \cite[Theorem 1.2(i)]{Is97b} proved that there are only creations
of periodic orbits but no annihilation as $a$ increases. Here, we
improve his theorem for some types of itineraries by showing that
the bifurcation parameters are analytic curves.
\begin{lem}
\label{lem:Tent p}When $a>1$, we have the following: 
\begin{itemize}
\item $\left.\frac{\partial p_{\infty,n}}{\partial a}\right|_{b=0}=1$ for
$n\geq2$.
\item $\left.\frac{\partial p_{\infty,2}}{\partial b}\right|_{b=0}=\frac{1}{a}-2$.
\item $\left.\frac{\partial p_{\infty,n}}{\partial b}\right|_{b=0}=-\frac{1}{a}$
for $n\geq3$.
\end{itemize}
\end{lem}

\begin{proof}
When $b=0$, we have 
\[
p_{\infty,n}=a-1
\]
for all $n\geq2$. This proves the first equality.

For $n\geq2$, let $L_{n}=\Lozi_{-}^{n-2}\circ\Lozi_{+}^{2}(W^{U}(\Lozi_{-},\boldsymbol{z}_{-}))$
and $(0,k_{n})$ be the intersection point of $L_{n}$ and the $y$-axis.
Then $p_{\infty,n}=a-b(k_{n}+1)-1$. Hence, 
\[
\left.\frac{\partial p_{\infty,n}}{\partial b}\right|_{b=0}=-k_{n}|_{b=0}-1.
\]
When $b=0$, we have $k_{2}=1-\frac{1}{a}$ and $k_{n}=\frac{1}{a}-1$
for all $n\geq3$. This proves the remaining two equalities.
\end{proof}
\begin{lem}
\label{lem:Tent q}When $a>1$, we have 
\[
\left.\frac{\partial r_{\infty}}{\partial a}\right|_{b=0}=0.
\]
\end{lem}

\begin{proof}
When $b=0$, we have
\[
r_{\infty}=1.\qedhere
\]
\end{proof}
Recall that $\FormalPeriodic_{\Sign,m.n}$ is the formal $\iota_{\Sign,m,n}$-periodic
point for $\Sign\in S$ and $m,n\geq2$.
\begin{thm}
\label{thm:The boundary of admissibility}For all $N\geq2$, there
exist $\overline{b}\in(0,1)$ and an integer $M\geq N+2$ such that
the following properties hold for all $m\geq M$ and $2\leq n\leq N$.

There exists an analytic curve $l_{m,n}:[0,\overline{b}]\rightarrow(\sqrt{2},4)$
that divides the parameter space $\ParaNbd(\overline{b})$ into two
regions of admissibility: Let $\Sign\in S$ and $(a,b)\in\ParaNbd(\overline{b})$.
\begin{enumerate}
\item If $a>l_{m,n}(b)$, then $\FormalPeriodic_{\Sign,m,n}$ is hyperbolic
and $\FormalPeriodic_{\Sign,m,n}\in C_{m,n}^{L}$.
\item If $a=l_{m,n}(b)$, then $(a,b)$ is an $\iota_{\pm,m,n}$-bifurcation
parameter. In particular, $\FormalPeriodic_{\Sign,m,n}$ is admissible,
$\FormalPeriodic_{-,m,n}=\FormalPeriodic_{+,m,n}\in C_{m,n}^{L}$,
and $\pi_{1}\circ\Lozi^{m+n-1}(\FormalPeriodic_{\Sign,m,n})=0$.
\item If $a<l_{m,n}(b)$, then $\FormalPeriodic_{\Sign,m,n}$ is not admissible. 
\end{enumerate}
\end{thm}

\begin{proof}
First, we define a parameter space such that $C_{m,n}^{L}$ exists.
By Proposition \ref{prop:Full cycle around a neighborhood of the tent family},
let $\underline{a}\in(\sqrt{2},2)$ and $\overline{b}_{0}\in(0,1)$
be constants such that the following properties hold for all $(a,b)\in\ParaNbd(\overline{b}_{0})$,
$m\geq M_{0}$, and $2\leq n\leq N$, where $M_{0}=N+2$. 
\begin{enumerate}[label={(C\arabic*)}]
\item If $a\geq\underline{a}$, then $C_{m,n}^{L}$ and $C_{m,n}^{R}$
exist.
\item \label{enu:Small a, Empty intersection}If $a\leq\underline{a}$,
then $C_{m}\cap U_{n}=\emptyset$.
\end{enumerate}

Second, we use Corollary \ref{cor:Geometrical characterization of the BCB in C(m,n)}
to show that for each $b\in[0,\overline{b}_{0}]$, $m\geq N+2$, and
$2\leq n\leq N$ there exists $l_{m,n}(b)\in(\underline{a},4)$ such
that $(l_{m,n}(b),b)$ is a bifurcation parameter of $\iota_{\pm,m,n}$.
By Proposition \ref{prop:Parameter space of full horseshoe}, we have
\[
p_{m,n}-q_{m,n}>0
\]
when $a\geq4$. By the condition \ref{enu:Small a, Empty intersection},
we have 
\[
p_{m,n}-q_{m,n}<0
\]
when $a=\underline{a}$. Thus, for each $m\geq N+2$, $2\leq n\leq N$,
and $b\in[0,\overline{b}_{0}]$, there exists a root $a=l_{m,n}(b)\in(\underline{a},4)$
of 
\begin{equation}
p_{m,n}(a,b)-q_{m,n}(a,b)=0\label{eq:bifurcation p=00003Dq}
\end{equation}
by the intermediate value theorem. The parameter $(l_{m,n}(b),b)$
is a bifurcation parameter of $\iota_{\pm,m,n}$ by Corollary \ref{cor:Geometrical characterization of the BCB in C(m,n)}.

Third, we use the estimations from the tent family to show that the
root of (\ref{eq:bifurcation p=00003Dq}) is unique on a neighborhood
of the tent family. For each $2\leq n\leq N$, we have 
\[
\left.\frac{\partial}{\partial a}\right|_{b=0}(p_{\infty,n}-r_{\infty})=1
\]
for all $a\geq\underline{a}$ by Lemmas \ref{lem:Tent p} and \ref{lem:Tent q}.
By continuity of the partial derivative and compactness of the interval
$[\underline{a},4]$, there exists $\overline{b}\in(0,\overline{b}_{0}]$
such that 
\[
\frac{\partial}{\partial a}(p_{\infty,n}-r_{\infty})>\frac{1}{2}
\]
for all $(a,b)\in[\underline{a},4]\times[0,\overline{b}]$ and $2\leq n\leq N$.
Moreover, by Corollaries \ref{cor:Convergence of p} and \ref{cor:Convergence of q},
$\lim_{m\rightarrow\infty}p_{m,n}=p_{\infty,n}$ and $\lim_{m\rightarrow\infty}q_{m,n}=r_{\infty}$
uniformly on compact subsets. There exists $M\geq M_{0}$ such that
\begin{equation}
\frac{\partial}{\partial a}(p_{m,n}-q_{m,n})>0\label{eq:increasing in a}
\end{equation}
for all $(a,b)\in[\underline{a},4]\times[0,\overline{b}]$, $m\geq M$,
and $n\in\{2,\cdots,N\}$. Consequently, the root of (\ref{eq:bifurcation p=00003Dq})
is unique when $m\geq M$, $2\leq n\leq N$, and $b\in[0,\overline{b}]$.

Finally, by the uniqueness of the root of (\ref{eq:bifurcation p=00003Dq}),
nonzero partial derivative (\ref{eq:increasing in a}), and the implicit
function theorem \cite[Theorem 2.3.1]{KP02}, we deduce that the curve
$l_{m,n}:[0,\overline{b}]\rightarrow(\underline{a},4)$ is analytic.
The curve divides the parameter space $\ParaNbd(\overline{b})\cap\{a\geq\underline{a}\}$
into two connected components. Each component is either fully hyperbolic
or fully non-admissible by the continuity of the admissibility function
(\ref{eq:Admissibility function}), Corollary \ref{cor:Geometrical characterization of the BCB in C(m,n)},
and the uniqueness of the root of (\ref{eq:bifurcation p=00003Dq}).
By Propositions \ref{prop:Full horseshoe} and \ref{prop:Parameter space of full horseshoe},
$\FormalPeriodic_{\Sign,m,n}(a,b)$ is admissible when $a\geq4$;
by the condition \ref{enu:Small a, Empty intersection} and Corollary
\ref{cor:Disjoint C and U implies nonadmissible parameters}, $\FormalPeriodic_{\Sign,m,n}(a,b)$
is not admissible when $a\leq\underline{a}$. Therefore, the left
component is fully non-admissible and the right component is fully
hyperbolic.
\end{proof}
\begin{cor}
\label{cor:The boundary of admissibility}There exist $\overline{b}\in(0,1)$,
an integer $M\geq5$, and an analytic curve $l_{m,n}:[0,\overline{b}]\rightarrow(\sqrt{2},4)$
for each $m\geq M$ and $n\in\{2,3\}$ such that $l_{m,n}$ satisfies
the properties in Theorem \ref{thm:The boundary of admissibility}
and 
\begin{equation}
\frac{\dif l_{m,2}}{\dif b}>\frac{\dif l_{m,3}}{\dif b}\label{eq:Transversal intersection}
\end{equation}
for all $b\in[0,\overline{b}]$.
\end{cor}

\begin{proof}
Let $N=3$ in Theorem \ref{thm:The boundary of admissibility}. By
the implicit function theorem, we have 
\begin{equation}
\frac{\dif l_{m,n}}{\dif b}(b)=-\frac{\frac{\partial}{\partial b}(p_{m,n}-q_{m,n})}{\frac{\partial}{\partial a}(p_{m,n}-q_{m,n})}(l_{m,n}(b),b)\label{eq:implicit derivative of l_m,n}
\end{equation}
for $n\in\{2,3\}$. For the limiting case $m=\infty$, we have 
\[
\left[-\frac{\frac{\partial}{\partial b}|_{b=0}(p_{\infty,2}-r_{\infty})}{\frac{\partial}{\partial a}|_{b=0}(p_{\infty,2}-r_{\infty})}\right]-\left[-\frac{\frac{\partial}{\partial b}|_{b=0}(p_{\infty,3}-r_{\infty})}{\frac{\partial}{\partial a}|_{b=0}(p_{\infty,3}-r_{\infty})}\right]=2-\frac{2}{a}\ge2-\sqrt{2}>0
\]
for all $a\in[\sqrt{2},4]$ by Lemmas \ref{lem:Tent p} and \ref{lem:Tent q}.
By continuity of the partial derivative and the compactness of the
interval $[\sqrt{2},4]$, we may assume $\overline{b}>0$ is small
enough such that 
\begin{equation}
\left[-\frac{\frac{\partial}{\partial b}(p_{\infty,2}-r_{\infty})}{\frac{\partial}{\partial a}(p_{\infty,2}-r_{\infty})}\right]-\left[-\frac{\frac{\partial}{\partial b}(p_{\infty,3}-r_{\infty})}{\frac{\partial}{\partial a}(p_{\infty,3}-r_{\infty})}\right]>0\label{eq:Db limiting case}
\end{equation}
for all $(a,b)\in[\sqrt{2},4]\times[0,\overline{b}]$. Moreover, by
(\ref{eq:implicit derivative of l_m,n}) and Corollaries \ref{cor:Convergence of p}
and \ref{cor:Convergence of q}, $\frac{\dif l_{m,2}}{\dif b}-\frac{\dif l_{m,3}}{\dif b}$
converges uniformly to the left hand side of (\ref{eq:Db limiting case})
on $[0,\overline{b}]$ as $m\rightarrow\infty$. Therefore, there
exists $M\geq5$ large enough such that $\frac{\dif l_{m,2}}{\dif b}-\frac{\dif l_{m,3}}{\dif b}>0$
on $[0,\overline{b}]$ for all $m\geq M$.
\end{proof}

\subsection{The parameter curve of the first homoclinic tangency of $z_{-}$}

We show that there exists an analytic curve in the parameter space
such that $u_{\infty}=r_{\infty}$. Thus, on the parameter curve,
we may apply a logarithm coordinate change to the turning points and
the stable laminations.
\begin{prop}
\label{prop:Curve of homoclinic tangency}For all $\epsilon>0$, there
exist $\overline{b}\in(0,1)$ and an analytic curve $t:[0,\overline{b}]\rightarrow(2-\epsilon,2+\epsilon)$
such that 
\[
u_{\infty}(a,b)=r_{\infty}(a,b)
\]
for all $b\in[0,\overline{b}]$ and $a=t(b)$.
\end{prop}

\begin{proof}
We have $u_{\infty}(2,0)=r_{\infty}(2,0)=1$ and $\frac{\partial(u_{\infty}-r_{\infty})}{\partial a}(2,0)=1$.
Therefore, the existence of the curve follows from the implicit function
theorem \cite[Theorem 2.3.1]{KP02}.
\end{proof}

\section{The reverse order of bifurcation}

In this section, we prove our main theorem. 
\begin{thm}
\label{thm:Main theorem}For all $\hat{b}\in(0,1)$, there exist $\overline{b}\in(0,\hat{b})$,
an integer $m\geq5$, and two analytic curves $l_{m,2},l_{m,3}:[0,\overline{b}]\rightarrow(\sqrt{2},4)$
such that the following properties hold. 
\begin{enumerate}
\item For each $(a,b)\in\ParaNbd(\overline{b})$, $\Sign\in S$, and $n\in\{2,3\}$,
$\FormalPeriodic_{\Sign,m,n}$ is admissible at $(a,b)$ if and only
if $a\geq l_{m,n}(b)$. In fact, the border collision bifurcation
occurs if and only if $a=l_{m,n}(b)$.
\item When $b=0$, we have $l_{m,2}(0)<l_{m,3}(0)$; when $b=\overline{b}$,we
have $l_{m,2}(\overline{b})>l_{m,3}(\overline{b})$. 
\item The intersection of the two curves $l_{m,2}$ and $l_{m,3}$ is unique
and transverse.
\end{enumerate}
\end{thm}

\begin{proof}
The existence of the two curves is provided by Corollary \ref{cor:The boundary of admissibility}.
Precisely, there exists a neighborhood of the tent family $\ParaNbd(\overline{b}_{1})$
and a constant $M\geq5$ such that the conclusion of the corollary
holds. For $m\geq M$, let $l_{m,2}$ and $l_{m,3}$ be the two analytic
curves (boundaries of admissible parameters) given by the corollary.

We want to find the two constants $\overline{b}\in(0,\hat{b})$ and
$m\geq M$ such that the order of bifurcation reverses. Let $t:[0,\overline{b}_{2}]\rightarrow[\sqrt{2},4]$
be the parameter curve of homoclinic tangency given by Proposition
\ref{prop:Curve of homoclinic tangency}. We may assume that $\overline{b}_{2}>0$
is small enough such that $[\sqrt{2},4]\times[0,\overline{b}_{2}]\subset\ParaModel$.
When $b\in[0,\overline{b}_{2}]$ and $a=t(b)$, we have $u_{\infty}=r_{\infty}$.
We apply the logarithm coordinate transformation $T(x)=-\log_{\lambda}(r_{\infty}-x)=-\log_{\lambda}(u_{\infty}-x)$
to the $x$-coordinate of the phase space. By Proposition \ref{prop:Exponential convergence of r},
there exist constants $c_{2}>c_{1}>0$ such that
\begin{equation}
m+\log_{\lambda}c_{1}<T(r_{m})<m+\log_{\lambda}c_{2}\label{eq:T(r_m) bounds}
\end{equation}
for all $m\geq2$ and $b\in[0,\overline{b}_{2}]$. Then
\begin{equation}
T(r_{m})-T(r_{m-2})<2+\log_{\lambda}\frac{c_{2}}{c_{1}}\label{eq:Upper bound of T(r_m)-T(r_m-2)}
\end{equation}
for all $m\geq4$ and $b\in[0,\overline{b}_{2}]$. Also, since $[\sqrt{2},4]\times[0,\overline{b}_{2}]$
is a compact subset of $\ParaModel$ and $0<(1-\frac{1}{\lambda^{n-1}})\frac{1}{\lambda^{n-2}}<1$,
Proposition \ref{prop:Exponential convergence of u} can be simplified
as
\begin{equation}
(n-1)\log_{\lambda}b^{-1}+\log_{\lambda}c_{3}\leq T(u_{n}^{d})\leq(n-1)\log_{\lambda}b^{-1}+\log_{\lambda}c_{4}\label{eq:T(u_m) bounds}
\end{equation}
for all $n\in\{2,3\}$, $d\in\{R,L\}$, and $b\in[0,\overline{b}_{2}]$,
where $c_{4}>c_{3}>0$ are constants. This implies that
\begin{equation}
T(u_{3}^{L})-T(u_{2}^{R})\geq\log_{\lambda}b^{-1}+\log_{\lambda}\frac{c_{3}}{c_{4}}\label{eq:Lower bound of T(u_3)-T(u_2)}
\end{equation}
for all $b\in(0,\overline{b}_{2}]$. Since $\lambda$ is bounded on
$\ParaFull\cap\{a\leq4\}$, there exists $\overline{b}_{3}>0$ small
enough such that 
\begin{equation}
\log_{\lambda}b^{-1}+\log_{\lambda}\frac{c_{3}}{c_{4}}>2+\log_{\lambda}\frac{c_{2}}{c_{1}}\label{eq:Condition when b is small}
\end{equation}
holds for all $b\in(0,\overline{b}_{3}]$. Finally, since $\lambda$
is bounded on $\ParaFull\cap\{a\leq4\}$, there exists $\overline{b}\in(0,\min(\overline{b}_{1},\overline{b}_{2},\overline{b}_{3},\hat{b})]$
small enough such that 
\begin{equation}
\log_{\lambda}\overline{b}+\log_{\lambda}\frac{c_{3}}{c_{2}}\geq M-1.\label{eq:condition on M}
\end{equation}
Let $m$ be the integer such that 
\begin{equation}
T(r_{m-2})\leq T(u_{2}^{R})<T(r_{m-1})\label{eq:definition of m}
\end{equation}
 when $(a,b)=(t(\overline{b}),\overline{b})$. Then $m>M$ by (\ref{eq:T(r_m) bounds}),
(\ref{eq:T(u_m) bounds}), and (\ref{eq:condition on M}).

Consider the case when $(a,b)=(t(\overline{b}),\overline{b})$. We
have $T(u_{2}^{R})<T(r_{m-1})$ and
\begin{align*}
T(u_{3}^{L})-T(r_{m}) & =[T(u_{3}^{L})-T(u_{2}^{R})]-[T(r_{m})-T(r_{m-2})]+[T(u_{2}^{R})-T(r_{m-2})]\\
 & >\left(\log_{\lambda}\overline{b}^{-1}+\log_{\lambda}\frac{c_{3}}{c_{4}}\right)-\left(2+\log_{\lambda}\frac{c_{2}}{c_{1}}\right)>0
\end{align*}
by (\ref{eq:Upper bound of T(r_m)-T(r_m-2)}), (\ref{eq:Lower bound of T(u_3)-T(u_2)}),
(\ref{eq:Condition when b is small}), and (\ref{eq:definition of m}).
Thus, $u_{2}^{R}<r_{m-1}$ and $u_{3}^{L}>r_{m}$ since $T$ is monotone
increasing. By Proposition \ref{prop:Full horseshoe} and Corollary
\ref{cor:Disjoint C and U implies nonadmissible parameters}, this
shows that, at the parameter $(t(\overline{b}),\overline{b})$, $\FormalPeriodic_{\Sign,m,3}$
is admissible, whereas $\FormalPeriodic_{\Sign,m,2}$ is not admissible
for each $\Sign\in S$. Consequently, $l_{m,3}(\overline{b})\leq t(\overline{b})<l_{m,2}(\overline{b})$.

Consider the case when $b=0$. We have $u=u_{2}^{L}=u_{3}^{R}=a-1$.
Let $\{w_{m}\}\times\Iv=C_{m,2}^{L}\cap C_{m,3}^{L}$. Then $u\leq r_{2}<r_{m-1}<w_{m}$
when $a=\sqrt{2}$ by Proposition \ref{prop:Period-doubling renormalizable parameters},
and $w_{m}<r_{m}<u$ when $a=2$ by Proposition \ref{prop:Parameter space of full horseshoe}.
The value $w_{m}$ depends continuously on $a$. By the intermediate
value theorem, there exists $a_{0}\in[\sqrt{2},2)$ such that $u=w_{m}$
when $a=a_{0}$. By Propositions \ref{prop:Full horseshoe} and \ref{prop:Admissible periodic point in C},
this shows that, at the parameter $(a_{0},0)$, $\FormalPeriodic_{\Sign,m,2}$
is admissible, whereas $\FormalPeriodic_{\Sign,m,3}$ is not admissible
for each $\Sign\in S$. Consequently, $l_{m,2}(0)\leq a_{0}<l_{m,3}(0)$.

Finally, the intersection of $l_{m,2}$ and $l_{m,3}$ is unique and
transverse because (\ref{eq:Transversal intersection}) holds.
\end{proof}

\appendix

\section{The universal cones}

The universal stable and unstable cones are $K^{S}=\{(x,y)\in\mathbb{R}^{2};|x|\leq\frac{b}{\lambda}\left|y\right|\}$
and $K^{U}=\{(x,y)\in\mathbb{R}^{2};|y|\leq\frac{1}{\lambda}\left|x\right|\}$
respectively. The cones were first introduced by \cite{MS18} for
the orientation reversing case, and then generalized to the orientation
preserving case by \cite{Pr21}. Here, we include the proof of the
orientation preserving case for completeness. Let $\left\Vert \cdot\right\Vert _{n}$
be the $L^{n}$ norm for $n=1,2,\cdots,\infty$.
\begin{thm}
\label{thm:Invaniance of the universal cones}Suppose that $(a,b)\in\ParaFull$
and $n=1,2,\cdots,\infty$. The followings are true.
\begin{enumerate}
\item $(\Dif\Lozi_{\Sign})K^{U}\subset K^{U}$ and $\left\Vert (\Dif\Lozi_{\Sign})\boldsymbol{v}\right\Vert _{n}\geq\lambda\left\Vert \boldsymbol{v}\right\Vert _{n}$
for $\Sign\in S$ and $\boldsymbol{v}\in K^{U}$.
\item $(\Dif\Lozi_{\Sign}^{-1})K^{S}\subset K^{S}$ and $\left\Vert \boldsymbol{v}\right\Vert _{n}\leq\mu\left\Vert (\Dif\Lozi_{\Sign}^{-1})\boldsymbol{v}\right\Vert _{n}$
for $\Sign\in S$ and $\boldsymbol{v}\in K^{S}$.
\end{enumerate}
\end{thm}

\begin{proof}
To prove that $K^{U}$ is $\Dif\Lozi_{\Sign}$-invariant, let $(x_{1},y_{1})\in K^{U}$
and $(x_{2},y_{2})=(\Dif\Lozi_{\Sign})(x_{1},y_{1})$. Without loss
of generality, we may assume that $x_{1}=1$. Then $|y_{1}|\leq\frac{1}{\lambda}$,
$y_{2}=1$, and $x_{2}=-\Sign a-by_{1}$. Compute
\begin{equation}
\frac{1}{\lambda}|x_{2}|\geq\frac{1}{\lambda}(a-b|y_{1}|)\geq\frac{1}{\lambda}(a-\frac{b}{\lambda})=1=|y_{2}|.\label{eq:Unstable cone}
\end{equation}
Thus, $(\Dif\Lozi_{\Sign})K^{U}\subset K^{U}$. Moreover, by (\ref{eq:Unstable cone}),
we get $|x_{2}|\geq\lambda|y_{2}|=\lambda|x_{1}|$. Also, by assumption,
we have $|y_{2}|=1\geq\lambda|y_{1}|$. This proves that $\left\Vert (x_{2},y_{2})\right\Vert _{n}\geq\lambda\left\Vert (x_{1},y_{1})\right\Vert _{n}$.

To prove that $K^{S}$ is $\Dif\Lozi_{\Sign}^{-1}$-invariant, let
$(x_{1},y_{1})\in K^{S}$ and $(x_{2},y_{2})=(\Dif\Lozi_{\Sign}^{-1})(x_{1},y_{1})$.
Without loss of generality, we may assume that $y_{1}=1$. Then $|x_{1}|\leq\frac{b}{\lambda}$,
$x_{2}=1$, and $x_{1}=-\Sign a-by_{2}$. Compute
\begin{equation}
b|y_{2}|\geq a-|x_{1}|\geq a-\frac{b}{\lambda}=\lambda|x_{2}|.\label{eq:Stable cone}
\end{equation}
Thus, $(\Dif\Lozi_{\Sign}^{-1})K^{S}\subset K^{S}$. Moreover, by
(\ref{eq:Stable cone}), we get $b|y_{2}|\geq\lambda|x_{2}|=\lambda|y_{1}|$.
Also, by assumption, we have $b|x_{2}|=b\geq\lambda|x_{1}|$. This
proves that $\frac{b}{\lambda}\left\Vert (x_{2},y_{2})\right\Vert _{n}\geq\left\Vert (x_{1},y_{1})\right\Vert _{n}$.
\end{proof}

\section{The global dynamics}

In this section, we study the global dynamics of the phase space.
We prove a theorem similar to \cite{BSV09}. They studied the global
dynamics of orientation reversing Lozi maps when one fixed point does
not have a homoclinic intersection, and gave an explicit description
of the basin of the Misiurewicz's strange attractor \cite{Mi80}.
Here we generalize their theorem to the orientation preserving maps.
We show that an orbit is either unbounded or eventually trapped inside
a trapping region. Our theorem is not restricted to the condition
of no homoclinic intersection. On the other hand, we do not study
the geometrical structure of the basins to shorten the proof.

Denote by $\omega(x)$ the omega limit set of $x$. 
\begin{thm}
\label{thm:global dynamics}Suppose that $(a,b)\in\ParaFull$. There
exists a compact set $T\subset\mathbb{R}^{2}$ such that for all $\boldsymbol{v}\in\mathbb{R}^{2}$
exactly one of the following is true:
\begin{enumerate}
\item $\lim_{n\rightarrow\infty}\pi_{1}\circ\Lozi^{n}(\boldsymbol{v})=\lim_{n\rightarrow\infty}\pi_{2}\circ\Lozi^{n}(\boldsymbol{v})=-\infty$.
\item $\omega(\boldsymbol{v})\subset T$.
\end{enumerate}
\end{thm}

The set $T$ in the theorem is the trapping region described in \cite{Pr21},
and contains the trapping region defined in \cite{CL98}.
\begin{proof}
Let $\chi=W^{S}(\Lozi_{-},\boldsymbol{z}_{-})=\{x=f(y)\}$, $\phi_{1}=W^{U}(\Lozi_{-},\boldsymbol{z}_{-})=\{y=g_{1}(x)\}$,
and $\phi_{2}=\Lozi_{+}(\phi_{1})=\{y=g_{2}(x)\}$ where $f$, $g_{1}$,
and $g_{2}$ are affine maps. Clearly, $\chi$ and $\phi_{1}$ intersect
at $\boldsymbol{z}_{-}$; $\phi_{1}$ and $\phi_{2}$ intersect at
$(u_{\infty},0)$, where $u_{\infty}=\lambda-1>0$.

We claim that $\chi$ and $\phi_{2}$ intersect at the second quadrant.
If $b=0$, then $\chi$ and $\phi_{2}$ intersect at the point $(-1,1)$.
Suppose that $b>0$. Let $(0,j)$ be the intersection point of $\chi$
and the $y$-axis. Let $(0,k)$ be the intersection point of $\phi_{2}$
and the $y$-axis. We have $j=\frac{\lambda}{b}-1>0$ and $k=\frac{\lambda-1}{\lambda+\frac{2b}{\lambda}}>0$.
Compute 
\[
j-k>\frac{\lambda-1}{b}-\frac{\lambda-1}{\lambda}=(\lambda-1)\left(\frac{1}{b}-\frac{1}{\lambda}\right)>0.
\]
Thus, $\chi$ and $\phi_{2}$ intersect at the second quadrant. 

First, we partition the phase space. Let $T$ be the closed set enclosed
by $\chi$, $\phi_{1}$, and $\phi_{2}$. Let $E=\{x<f(y)\text{ and }x<0\}$,
$F=\{x>f(y),y<g_{1}(x)\text{, and }x\leq0\}$, $G=\{x>0,y<g_{1}(x)\text{, and }y\leq0\}$,
$H=\{y>0,y>g_{2}(x)\text{, and }x\geq0\}$, $J=\{x\geq f(y),y>g_{2}(x)\text{, and }x<0\}$,
and $S=\{x=f(y)\text{ and }y<-1\}$. See Figure \ref{fig:Partition of the phase space}
for an illustration. By definition, the sets $E$, $F$, $G$, $H$,
$J$, $S$, and $T$ form a partition of the phase space.

\begin{figure}
\begin{minipage}[c]{0.4\columnwidth}%
\center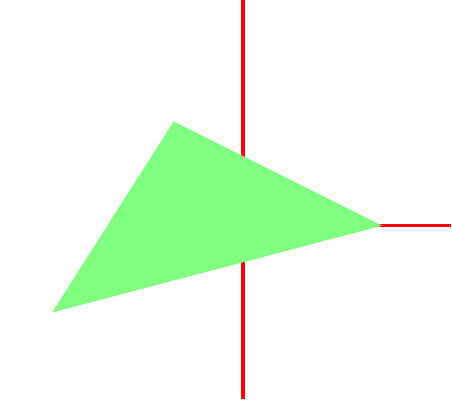

\caption{\label{fig:Partition of the phase space}A partition of the phase
space.}
\end{minipage}\hfill{}%
\begin{minipage}[c]{0.6\columnwidth}%
\[
S\rightarrow\text{converges to \ensuremath{z_{-}\in T}}
\]
\[
\xymatrix{ &  &  &  & E\ar[r] & \text{escapes to infinity}\\
F\ar[rr]^{\text{finite iterations}} &  & G\ar[r]\ar[dr]\ar[urr] & H\ar[d]\ar[dr]\ar[ur]\\
 &  &  & J\ar[r] & T\ar[uu]\ar@(dr,ur)
}
\]

\caption{\label{fig:Global dynamics}Trajectories in the partition.}
\end{minipage}

\end{figure}

Next, we study the iterations of the sets. The set $E$ is $\Lozi$-invariant
and $\lim_{n\rightarrow\infty}\pi_{1}\circ\Lozi^{n}(\boldsymbol{v})=\lim_{n\rightarrow\infty}\pi_{2}\circ\Lozi^{n}(\boldsymbol{v})=-\infty$
for all $\boldsymbol{v}\in E$. The set $S$ is $\Lozi$-invariant
and $\lim_{n\rightarrow\infty}\Lozi^{n}(\boldsymbol{v})=z_{-}$ for
all $\boldsymbol{v}\in S$. By definition, $\Lozi(F)=F\cup G$. In
fact, for all $\boldsymbol{v}\in F$, there exists $n>0$ such that
$\Lozi^{n}(\boldsymbol{v})\in G$. This is because $F$ is a subset
of a component separated by the stable and unstable manifolds of the
affine map $\Lozi_{-}$. Moreover, we have $\Lozi(G)\subset\{y>g_{2}(x)\text{ and }y>0\}\subset E\cup J\cup H$
and $\Lozi(J)\subset T$.

We claim that $\Lozi(H)\subset T\cup J\cup E$. Let $\phi_{3}=\Lozi_{+}(\phi_{2})=\{y=g_{3}(x)\}$
where $g_{3}$ is an affine map. First, the point $\Lozi_{+}(u_{\infty},0)$
is the intersection of $\phi_{2}$ and $\phi_{3}$. Compute 
\[
\pi_{1}\circ\Lozi_{+}(u_{\infty},0)=-(\lambda-1)\left(\lambda-1+\frac{2b}{\lambda}\right)<0.
\]
Thus, the intersection point lies on the second quadrant. Second,
let $m_{2}$ and $m_{3}$ be the slopes of $\phi_{2}$ and $\phi_{3}$
respectively. By Lemma \ref{lem:Invariance of the unstable cone},
we get $m_{3}<m_{2}<0$. Consequently, by the two consequences, we
conclude $\Lozi(H)\subset\{y<g_{3}(x)\text{ and }y\geq0\}\subset T\cup J\cup E$.

We claim that $\Lozi(T)\subset T\cup E$. We break the set $T$ into
two components: $T\cap\{x\leq0\}$ and $T\cap\{x>0\}$. We have $\Lozi(T\cap\{x\leq0\})\subset\Lozi(\{x\geq f(y),y\geq g_{1}(x)\text{, and }x\leq0\})=\{x\geq f(y),y\geq g_{1}(x)\text{, and }y\leq0\}\subset T$
and $\Lozi(T\cap\{x>0\})\subset\Lozi(\{x>0\text{ and }y\geq g_{1}(x)\})=\{y>0\text{ and }y\leq g_{2}(x)\}\subset T\cup E$.

Finally, the orbit of a point follows the diagram in Figure \ref{fig:Global dynamics}.
This proofs the theorem.
\end{proof}
The corollary follows immediately from the theorem and $\Lozi(T\backslash D)\subset E$
when $(a,b)\in\ParaModel$.
\begin{cor}
\label{cor:Invariant sets are bounded in D}Suppose that $(a,b)\in\ParaModel$.
If $K\subset\mathbb{R}^{2}$ is a bounded $\Lozi$-invariant set,
then $K\subset D\cap T$.
\end{cor}

\bibliographystyle{alphaurl}
\phantomsection\addcontentsline{toc}{section}{\refname}\nocite{*}
\bibliography{Reference}

\end{document}